\documentclass[11 pt]{amsart}
\usepackage{latexsym,amscd,amssymb, graphicx, amsthm, bm, amsmath}  
\usepackage{young}
\usepackage{tikz}

\usepackage[margin=1in]{geometry}

\numberwithin{equation}{section}

\newtheorem{theorem}{Theorem}[section]
\newtheorem{proposition}[theorem]{Proposition}
\newtheorem{corollary}[theorem]{Corollary}
\newtheorem{lemma}[theorem]{Lemma}

\newtheorem{observation}[theorem]{Observation}

\newtheorem{problem}[theorem]{Problem}

\theoremstyle{definition}

\newtheorem{defn}[theorem]{Definition}
\newtheorem{example}[theorem]{Example}

\newcommand{\comaj}{{\mathrm {comaj}}}
\newcommand{\maj}{{\mathrm {maj}}}
\newcommand{\inv}{{\mathrm {inv}}}

\newcommand{\sign}{{\mathrm {sign}}}

\newcommand{\Des}{{\mathrm {Des}}}
\newcommand{\Val}{{\mathrm {Val}}}
\newcommand{\Rise}{{\mathrm {Rise}}}
\newcommand{\Stir}{{\mathrm {Stir}}}
\newcommand{\Hilb}{{\mathrm {Hilb}}}
\newcommand{\grFrob}{{\mathrm {grFrob}}}
\newcommand{\des}{{\mathrm {des}}}
\newcommand{\asc}{{\mathrm {asc}}}
\newcommand{\coinv}{{\mathrm {coinv}}}
\newcommand{\rev}{{\mathrm {rev}}}
\newcommand{\Asc}{{\mathrm {Asc}}}

\newcommand{\Ind}{{\mathrm {Ind}}}

\newcommand{\SYT}{{\mathrm {SYT}}}

\newcommand{\Frob}{{\mathrm {Frob}}}

\newcommand{\initial}{{\mathrm {in}}}

\newcommand{\shape}{{\mathrm {shape}}}

\newcommand{\symm}{{\mathfrak{S}}}

\newcommand{\CC}{{\mathbb {C}}}
\newcommand{\QQ}{{\mathbb {Q}}}
\newcommand{\ZZ}{{\mathbb {Z}}}

\newcommand{\OP}{{\mathcal{OP}}}
\newcommand{\DDD}{{\mathcal{D}}}

\newcommand{\WWW}{{\mathcal{W}}}
\newcommand{\NNN}{{\mathcal{N}}}
\newcommand{\FFF}{{\mathcal{F}}}
\newcommand{\CCC}{{\mathcal{C}}}

\newcommand{\AAA}{{\mathcal{A}}}
\newcommand{\MMM}{{\mathcal{M}}}
\newcommand{\BBB}{{\mathcal{B}}}

\newcommand{\EGS}{{\mathcal{EGS}}}
\newcommand{\ED}{{\mathcal{ED}}}

\newcommand{\zz}{{\mathbf {z}}}
\newcommand{\xx}{{\mathbf {x}}}
\newcommand{\II}{{\mathbf {I}}}
\newcommand{\yy}{{\mathbf {y}}}
\newcommand{\TT}{{\mathbf {T}}}
\newcommand{\mm}{{\mathbf {m}}}

\newcommand{\blambda}{{ \bm{\lambda} }}
\newcommand{\bT}{{ \bm{T}}}
\newcommand{\bbeta}{{ \bm{\beta} }}

\begin{document}

\title[Generalized coinvariant algebras for wreath products]
{Generalized coinvariant algebras for wreath products}

\author{Kin Tung Jonathan Chan}
\address
{Department of Mathematics \newline \indent
University of California, San Diego \newline \indent
La Jolla, CA, 92093-0112, USA}
\email{joc090@ucsd.edu, bprhoades@math.ucsd.edu}

\author{Brendon Rhoades}

\begin{abstract}
Let $r$ be a positive integer and
let $G_n$ be the  reflection group  of $n \times n$ monomial matrices whose
entries are $r^{th}$ complex roots of unity and let $k \leq n$.  We define and study 
two new graded
quotients $R_{n,k}$ and $S_{n,k}$ of the polynomial ring $\CC[x_1, \dots, x_n]$
in $n$ variables.  When $k = n$, both of these quotients coincide with the classical coinvariant 
algebra attached to $G_n$.  
The algebraic properties of our quotients are governed by the combinatorial properties of 
$k$-dimensional faces in the Coxeter complex attached to $G_n$ (in the case of $R_{n,k}$)
and $r$-colored ordered set partitions of $\{1, 2, \dots, n\}$ with $k$ blocks
(in the case of $S_{n,k}$).
Our work generalizes a construction of Haglund, Rhoades, and Shimozono  from
the symmetric group $\symm_n$ to the more general wreath products $G_n$.
\end{abstract}

\maketitle

\section{Introduction}
\label{Introduction}

The coinvariant algebra of the symmetric group $\symm_n$ is among the most important
$\symm_n$-modules in combinatorics.   It is a graded version of the regular representation of 
$\symm_n$, has structural properties deeply tied to the combinatorics of permutations,
and gives a combinatorially
accessible model for the action of $\symm_n$ on the cohomology ring $H^{\bullet}(G/B)$
of the flag manifold $G/B$.    

Haglund, Rhoades, and Shimozono \cite{HRS} recently defined a generalization
of the $\symm_n$-coinvariant algebra which depends on an integer parameter $k \leq n$.
The structure of their graded $\symm_n$-module is governed by the combinatorics of 
ordered set partitions of $[n] := \{1, 2, \dots, n \}$ with $k$ blocks.
The graded Frobenius images of this module is  (up to a minor twist) either of the combinatorial
expressions $\Rise_{n,k}(\xx;q,t)$ or
$\Val_{n,k}(\xx;q,t)$ appearing in the {\em Delta Conjecture} of Haglund, Remmel, and Wilson \cite{HRW}
upon setting $t = 0$.  The Delta Conjecture is a generalization
of the Shuffle Conjecture in the field of 
Macdonald polynomials; this gives the first example of a `naturally constructed' module 
with Frobenius image related to
the Delta Conjecture.

A linear transformation $t \in GL_n(\CC)$ is a {\em reflection} if the fixed space of $t$ has codimension $1$
in $\CC^n$ and $t$ has finite order.  A finite subgroup $W \subseteq GL_n(\CC)$ is called a 
{\em reflection group} if $W$ is generated by reflections.
Given any complex reflection group $W$, there is a coinvariant algebra $R_W$ attached to $W$.
The algebra $R_W$ is a graded $W$-module with structural properties closely related to the combinatorics
of $W$.  In this paper we provide a Haglund-Rhoades-Shimozono style generalization of $R_W$
in the case where $R_W$ belongs to the family of reflection groups $G(r,1,n) = \ZZ_r \wr \symm_n$.

The general linear group $GL_n(\CC)$ acts on the polynomial ring 
$\CC[\xx_n] := \CC[x_1, \dots, x_n]$ by linear substitutions.
If $W \subset GL_n(\CC)$ is any finite subgroup,
let 
\begin{equation*}
\CC[\xx_n]^W := \{ f(\xx_n) \in \CC[\xx_n] \,:\, w.f(\xx_n) = f(\xx_n) \text{ for all $w \in W$} \}
\end{equation*}
denote the associated subspace of {\em invariant polynomials} and
let $\CC[\xx_n]^W_+ \subset \CC[\xx_n]^W$ denote 
the collection of invariant polynomials with vanishing constant term.
The {\em invariant ideal} $I_W \subset \CC[\xx_n]$ is 
the ideal $I_W := \langle \CC[\xx_n]^W_+ \rangle$ generated by $\CC[\xx_n]^W_+$ and the 
{\em coinvariant algebra} is $R_W := \CC[\xx_n]/I_W$.
The quotient $R_W$ is a graded $W$-module.
A celebrated result of Chevalley \cite{C} states that if $W$ is a complex reflection group,
then $R_W$ is isomorphic to the regular representation $\CC[W]$ as a $W$-module.

\begin{quote}
{\bf Notation.}  {\em Throughout this paper $r$ will denote a positive integer.  Unless otherwise stated, we 
assume $r \geq 2$.  Let $\zeta := e^{\frac{2 \pi i}{r}} \in \CC$ and let $G := \langle \zeta \rangle$ be the 
multiplicative 
group of $r^{th}$ roots of unity in $\CC^{\times}$.}
\end{quote}

Let us introduce the family of reflection groups we will focus on.
A  matrix is {\em monomial } if it has a unique nonzero entry in every row and column.
Let $G_n$ be the group of $n \times n$ monomial matrices whose nonzero entries lie in $G$.
For example, if $r = 3$ we have
\begin{equation*}
g = 
\begin{pmatrix}
0 & 0 & \zeta & 0 \\
0 & 1 & 0 & 0 \\
0 & 0 & 0 & \zeta^2 \\
\zeta & 0 & 0 & 0 
\end{pmatrix} \in G_4.
\end{equation*}
Matrices in $G_n$ may be thought of combinatorially as {\em $r$-colored permutations}
$\pi_1^{c_1} \dots \pi_n^{c_n}$, where $\pi_1 \dots \pi_n$ is a permutation in $\symm_n$ and
$c_1 \dots c_n$ is a sequence of `colors' in the set $\{0, 1, \dots, r-1\}$ representing powers of $\zeta$.
For example, the above element of $G_4$ may be represented combinatorially as
$g = 4^1 2^0 1^1 3^2$.  

In the usual classification of complex reflection groups we have $G_n = G(r,1,n)$.  The group
$G_n$ is isomorphic to the wreath product $\ZZ_r \wr \symm_n = (\ZZ_r \times \cdots \times \ZZ_r) \rtimes \symm_n$,
where the symmetric group $\symm_n$ acts on the $n$-fold direct product of cyclic groups
$\ZZ_r \times \cdots \times \ZZ_r$  by coordinate permutation.
For the sake of legibility, we suppress reference to $r$ in our notation for $G_n$ and related objects.

Let $I_n \subseteq \CC[\xx_n]$ be the invariant ideal associated to $G_n$.  We have
$I_n = \langle e_1(\xx_n^r), \dots, e_n(\xx_n^r) \rangle$, where 
\begin{equation*}
e_d(\xx_n^r) = e_d(x_1^r, \dots, x_n^r) := \sum_{1 \leq i_1 < \cdots < i_d \leq n} x_{i_1}^r \cdots x_{i_d}^r
\end{equation*}
is the $d^{th}$ elementary symmetric function in the variable powers $x_1^r, \dots, x_n^r$.
Let $R_n := \CC[\xx_n]/I_n$ denote the coinvariant ring attached to $G_n$.

The algebraic properties of the quotient $R_n$ are governed by the combinatorial properties of 
$r$-colored permutations in $G_n$.  
Chevalley's result \cite{C} implies that $R_n \cong \CC[G_n]$ as ungraded $G_n$-modules.
The fact that $e_1(\xx_n^r), \dots, e_n(\xx_n^r)$ is a regular sequence in $\CC[\xx_n]$ gives the 
following expression for the Hilbert series of $R_n$:
\begin{equation}
\Hilb(R_n; q) =  \prod_{i = 1}^n \frac{1-q^{ri}}{1-q}  = \sum_{g \in G_n} q^{\maj(g)},
\end{equation}
where $\maj$ is the {\em major index} statistic on $G_n$ 
(also known as the {\em flag-major index}; see \cite{HLR}).
Bango and Biagoli \cite{BB} described a {\em descent monomial basis}
$\{b_g \,:\, g \in G_n\}$ of $R_n$ whose elements satisfy $\deg(b_g) = \maj(b_g)$.
Stembridge \cite[Thm. 6.6]{Stembridge} described the graded $G_n$-module structure of $R_n$ 
using (the $r \geq 1$ generalization of) standard Young tableaux.

When $r = 1$ and $G_n = \symm_n$ is the symmetric group, Haglund, Rhoades, and Shimozono
\cite[Defn. 1.1]{HRS} introduced and studied a generalization of the coinvariant algebra $R_n$ depending
on a positive integer $k \leq n$.
In this paper we extend \cite[Defn. 1.1]{HRS} to $r \geq 2$ by introducing the following {\em two} families of ideals
$I_{n,k}, J_{n,k} \subseteq \CC[\xx_n]$.

\begin{defn}
\label{main-definition}
Let $n, k,$ and $r$ be nonnegative integers which satisfy $n \geq k, n \geq 1$, and $r \geq 2$.  
We define
two quotients of the polynomial ring $\CC[\xx_n]$ as follows.

\begin{enumerate}
\item  Let $I_{n,k} \subseteq \CC[\xx_n]$ be the ideal
\begin{equation*}
I_{n,k} := \langle x_1^{kr+1}, x_2^{kr+1}, \dots, x_n^{kr+1}, e_n(\xx_n^r), e_{n-1}(\xx_n^r), \dots, e_{n-k+1}(\xx_n^r) \rangle
\end{equation*}
and let $R_{n,k}$ be the corresponding quotient:
\begin{equation*}
R_{n,k} := \CC[\xx_n]/I_{n,k}.
\end{equation*}
\item  Let $J_{n,k} \subseteq \CC[\xx_n]$ be the ideal
\begin{equation*}
J_{n,k} := \langle x_1^{kr}, x_2^{kr}, \dots, x_n^{kr}, e_n(\xx_n^r), e_{n-1}(\xx_n^r), \dots, e_{n-k+1}(\xx_n^r) \rangle
\end{equation*}
and let $S_{n,k}$ be the corresponding quotient:
\begin{equation*}
S_{n,k} := \CC[\xx_n]/J_{n,k}.
\end{equation*}
\end{enumerate}
\end{defn}

Both of the ideals $I_{n,k}$ and $J_{n,k}$ are homogeneous and stable under the action of 
$G_n$ on $\CC[\xx_n]$.  It follows that the quotients $R_{n,k}$ and $S_{n,k}$ are graded 
$G_n$-modules.  The ring introduced in \cite[Defn. 1.1]{HRS} is the ideal $S_{n,k}$ with $r = 1$.

When $k = n$, it can be shown 
\footnote{By \cite[Sec. 7.2]{Bergeron} under the change of variables $(x_1, \dots, x_n) \mapsto (x_1^r, \dots, x_n^r)$
we have $x_n^{nr} \in I_n$, and the ideal $I_n$ is stable under $\symm_n$.}
that for any $1 \leq i \leq n$, the variable power $x_i^{nr}$ lies in the invariant ideal
$I_n$, so that $I_{n,n} = J_{n,n} = I_n$, and $R_{n,n} = S_{n,n}$ are both equal to the classical 
coinvariant algebra $R_n$ for $G_n$.
At the other extreme, we have $R_{n,0} \cong \CC$ (the trivial representation in degree $0$)
and $S_{n,0} = 0$.

The reader may wonder why we are presenting two generalizations of the ring of \cite{HRS} rather than one.
The combinatorial reason for this is the presence of {\em zero blocks} in the $G_n$-analog of ordered 
set partitions.  These zero blocks do not appear in the case of \cite{HRS} when $r = 1$ 
(or in the case of the classical coinvariant algebra when $k = n$).
Roughly speaking, the ring $S_{n,k}$ will be a `zero block free' version of  $R_{n,k}$.
These rings will be related in a nice way (see Proposition~\ref{r-to-s-reduction}), and  
 $S_{n,k}$ will be easier to analyze directly.
 
 The generators of the ideal $I_{n,k}$ defining the quotient $R_{n,k}$ come in two flavors:
 \begin{itemize}
 \item  high degree invariant polynomials $e_n(\xx_n^r), e_{n-1}(\xx_n^r), \dots, e_{n-k+1}(\xx_n^r)$, and
 \item  a collection of polynomials $x_1^{kr+1}, \dots, x_n^{kr + 1}$ whose linear span
 $\mathrm{span} \{x_1^{kr+1}, \dots, x_n^{kr+1} \}$ is stable under the action of $G_n$ and carries the 
 dual of the defining action of $G_n$ on $\CC^n$.
 \end{itemize}
 This extends the two flavors of generators for the ideal of \cite{HRS}.  In the context of the 0-Hecke
 algebra $H_n(0)$ attached to the symmetric group, Huang and Rhoades \cite{HuangRhoades}
 defined another ideal  (denoted in \cite{HuangRhoades} by $J_{n,k} \subseteq \mathbb{F}[\xx_n]$,
 where $\mathbb{F}$ is any field) with analogous types of generators: high degree $H_n(0)$-invariants
 together with a copy of the defining representation of $H_n(0)$ sitting in homogeneous degree $k$.
 It would be interesting to see if the favorable properties of the corresponding quotients
 could be derived from this choice of generator selection in a more conceptual way.

In this paper we will prove that the structures of the rings 
$R_{n,k}$ and $S_{n,k}$ are controlled by $G_n$-generalizations of ordered set 
partitions.  
We will use the usual $q$-analog notation
\begin{align*}
[n]_q := 1 + q + \cdots + q^{n-1} &   &[n]!_q := [n]_q [n-1]_q \cdots [1]_q  \\
{n \brack a_1, \dots , a_r}_q := \frac{[n]!_q}{[a_1]!_q \cdots [a_r]!_q} 
& &{n \brack a}_q := \frac{[n]!_q}{[a]!_q [n-a]!_q}.
\end{align*}
We also let $\rev_q$ be the operator which reverses the coefficient sequences in polynomials in the
variable $q$ (over any ground ring).
For example, we have
\begin{equation*}
\rev_q(8q^2 + 7q + 6) = 6q^2 + 7q + 8.
\end{equation*}
Let $\Stir(n,k)$ be the (signless) Stirling number of the second kind counting set partitions of $[n]$ into $k$ blocks
and let $\Stir_q(n,k)$ denote the {\em $q$-Stirling number} 
defined by the recursion 
\begin{equation*}
\Stir_q(n,k) = [k]_q \cdot \Stir_q(n-1,k) + \Stir_q(n-1,k-1) 
\end{equation*}
for $n, k \geq 1$ and the 
initial condition $\Stir_q(0,k) = \delta_{0,k}$.
Deferring various definitions to Section~\ref{Background}, we state our main results.

\begin{itemize}
\item  As {\em ungraded} $G_n$-modules we have
\begin{center}
$R_{n,k} \cong \CC[\FFF_{n,k}]$ and $S_{n,k} \cong \CC[\OP_{n,k}]$,
\end{center}
where $\FFF_{n,k}$ is the set of $k$-dimensional faces in the Coxeter complex attached
to $G_n$ and $\OP_{n,k}$ is the set of $r$-colored ordered set partitions of 
$[n]$ with $k$ blocks (Corollary~\ref{ungraded-isomorphism-type}).  
In particular, we have
\begin{align*}
\dim(R_{n,k}) &= \sum_{z = 0}^{n-k}  {n \choose z} \cdot r^{n-z} \cdot k! \cdot \Stir(n-z,k), \\
\dim(S_{n,k}) &= r^n \cdot k! \cdot \Stir(n,k).
\end{align*}

\item The Hilbert series $\Hilb(R_{n,k}; q)$ and $\Hilb(S_{n,k};q)$ are given by
(Corollary~\ref{hilbert-series-corollary})
\begin{align*}
\Hilb(R_{n,k}; q) &= \sum_{z = 0}^{n-k} {n \choose z}  \cdot q^{krz} \cdot 
\rev_q( [r]_q^{n-z} \cdot [k]!_{q^r} \cdot \Stir_{q^r}(n-z,k)), \\
\Hilb(S_{n,k}; q) &= \rev_q( [r]_q^n \cdot [k]!_{q^r} \cdot \Stir_{q^r}(n,k)).
\end{align*}

\item  Endow monomials in $\CC[\xx_n]$ with the lexicographic term order.  The standard monomial 
basis of $R_{n,k}$ is the collection of monomials $m = x_1^{a_1} \cdots x_n^{a_n}$ whose exponent
sequences $(a_1, \dots, a_n)$ are componentwise $\leq$ some shuffle of the sequences
$(r-1, 2r-1, \dots, kr-1)$ and $(\underbrace{kr, \dots, kr}_{n-k})$.  

The standard monomials basis of $S_{n,k}$ is the collection of monomials 
$m = x_1^{b_1} \cdots x_n^{b_n}$ whose exponent sequences $(b_1, \dots, b_n)$ are componentwise
$\leq$ some shuffle of the sequences $(r-1, 2r-1, \dots, kr-1)$ and $(\underbrace{kr-1, \dots, kr-1}_{n-k})$
(Theorem~\ref{artin-basis}).

\item There is a generalization of Bango and Biagoli's  descent monomial basis of $R_n$
to the rings $R_{n,k}$ and $S_{n,k}$ (Theorems~\ref{s-gs-basis-theorem} and \ref{r-gs-basis-theorem}).

\item  We have an explicit description of the {\em graded} isomorphism type of the $G_n$-modules 
$R_{n,k}$ and $S_{n,k}$ in terms of standard Young tableaux 
(Theorem~\ref{graded-isomorphism-type}). 
\end{itemize}

Although the properties of the rings $R_{n,k}$ (and $S_{n,k}$) shown above give natural extensions of the 
corresponding properties of $R_n$, the proofs of these results will be quite different.
Since the classical invariant ideal $I_n$ is cut out by a regular sequence 
$e_1(\xx_n^r), \dots, e_n(\xx_n^r)$, standard tools from commutative algebra (the {\em Koszul complex})
can be used to derive the graded isomorphism type of $R_n$.
Since neither the dimension
$\dim(R_{n,k}) = \sum_{z = 0}^{n-k}  {n \choose z} \cdot r^{n-z} \cdot k! \cdot \Stir(n-z,k)$ nor
$\dim(S_{n,k}) = r^n \cdot k! \cdot \Stir(n,k)$
have nice product formulas, we cannot hope to apply this technology to our situation.

Replacing the commutative algebra machinery used to analyze $R_n$
will be {\em combinatorial} commutative algebra machinery (Gr\"obner theory and straightening laws)
which will determine the structure of $R_{n,k}$.
Although some portions of our analysis will follow  from the arguments of \cite{HRS}
after making the change of variables $(x_1, \dots, x_n) \mapsto (x_1^r, \dots, x_n^r)$,
other arguments will have to be significantly adapted to account for the possible presence of zero blocks.

The rest of the paper is organized as follows.
In {\bf Section~\ref{Background}} we give background material related to $r$-colored ordered set partitions,
the Coxeter complex of $G_n$, symmetric functions, the representation theory of $G_n$, and 
Gr\"obner theory.
In {\bf Section~\ref{Polynomial}} we  prove some polynomial and symmetric function identities that will
be helpful in later sections.
In {\bf Section~\ref{Hilbert}} we calculate the standard monomial bases of $R_{n,k}$ and $S_{n,k}$
with respect to the lexicographic term order and calculate the Hilbert series of these quotients.
In {\bf Section~\ref{Descent}} we present our generalizations of the Bango-Biagoli descent monomial
basis of $R_n$ to obtain descent monomial-type bases for $R_{n,k}$ and $S_{n,k}$.
In {\bf Section~\ref{Frobenius}} we derive the graded isomorphism type of the 
$G_n$-modules $R_{n,k}$ and $S_{n,k}$.
We close in {\bf Section~\ref{Conclusion}} with some open questions.

\section{Background}
\label{Background}

\subsection{$r$-colored ordered set partitions}
We will make use of two orders on the 
alphabet 
\begin{equation*}
\AAA_r := \{i^c \,:\, i \in \ZZ_{> 0} \text{ and } 0 \leq c \leq r-1 \} 
\end{equation*}
of $r$-colored positive integers.  The first order $<$
weights colors more heavily than letter values, with higher colors being smaller:
\begin{equation*}
1^{r-1} < 2^{r-1} < \cdots < 1^{r-2} < 2^{r-2} < \cdots < 1^0 < 2^0 < \cdots.
\end{equation*}
The second order $\prec$ weights letter values more heavily than colors:
\begin{equation*}
1^{r-1} \prec 1^{r-2} \prec \cdots \prec 1^0 \prec 2^{r-1} \prec 2^{r-2} \prec \cdots \prec 2^0 \prec \cdots.
\end{equation*}

Let $w = w_1^{c_1} \dots w_n^{c_n}$ be any word in the alphabet $\AAA_r$.  
The {\em descent set} and {\em ascent set} of $w$ are defined using the order $<$:
\begin{equation}
\Des(w) := \{1 \leq i \leq n-1 \,:\, w_i^{c_i} > w_{i+1}^{c_{i+1}} \}, \hspace{0.2in}
\Asc(w) := \{1 \leq i \leq n-1 \,:\, w_i^{c_i} < w_{i+1}^{c_{i+1}} \}.
\end{equation}
We write $\des(w) := |\Des(w)|$ and $\asc(w) := |\Asc(w)|$ for the number of descents 
and ascents in $w$.
The {\em major index} $\maj(w)$ is given by the formula
\begin{equation}
\maj(w) := c(w) + r \cdot \sum_{i \in \Des(w)} i,
\end{equation}
where $c(w)$ denotes the sum of the colors of the letters in $w$.
This version of major index was defined by Haglund, Loehr, and Remmel in \cite{HLR}
(where it was termed `flag-major index').

Since we may view elements of $G_n$ as  $r$-colored permutations, the objects 
defined in the above paragraph make sense for $g \in G_n$.
For example, if $r = 3$ and $g = 3^0 4^1 6^2 2^0 5^2 1^2 \in G_6$, we have
$\Des(g) = \{1,2,4,5\}, \Asc(g) = \{3\}, \des(g) = 4, \asc(g) = 1,$
and 
\begin{equation*}
\maj(g) = (0 + 1 + 2 + 0 + 2 + 2) + 3 \cdot (1 + 2 + 4 + 5) = 43.
\end{equation*}

An {\em ordered set partition} is a set partition equipped with a total order on its blocks.
An {\em $r$-colored ordered set partition of size $n$} 
is an ordered set partition $\sigma$ of $[n]$ in which every letter is
assigned a color in the set $\{0, 1, \dots, r-1\}$.
For example,
\begin{equation*}
\sigma = \{3^0,4^1\} \prec \{6^2\} \prec \{1^2,2^0,5^0\}
\end{equation*}
is a $3$-colored ordered set partition of size $6$ with $3$ blocks.
We let $\OP_{n,k}$ be the collection of $r$-colored ordered set partitions of size $n$ with $k$ blocks.
We have 
\begin{equation}
|\OP_{n,k}| = r^n \cdot k! \cdot \Stir(n,k).
\end{equation}

We will often use bars to represent colored ordered set partitions more 
succinctly.  Here we write block elements in increasing order with respect to $\prec$.  Our example ordered set partition
becomes
\begin{equation*}
\sigma = (3^0 4^1 \mid 6^2 \mid 1^2 2^0 5^2 ).
\end{equation*}

We also have a descent starred notation for colored ordered set partitions, where we order elements within blocks 
in a decreasing fashion with respect to $<$.  Our example ordered set partition becomes
\begin{equation*}
\sigma = 3^0_* 4^1 \, \, 6^2 \, \, 2^0_* 5^2_* 1^2.
\end{equation*}
Notice that we use the order $\prec$ for the bar notation, but the order $<$ for the star notation.
The star notation represents $\sigma \in \OP_{n,k}$ as a pair $\sigma = ( g, S)$,
where $g \in G_n$, $|S| = n-k$ and $S \subseteq \Des(g)$.  
Our example ordered set partition becomes
\begin{equation*}
\sigma = (3^0 4^1 6^2 2^0 5^2 1^2, \{1,4,5\}).
\end{equation*}

Let $\sigma \in \OP_{n,k}$ and let $(g,S)$ be the descent starred representation of $\sigma$.
The {\em major index} of $\sigma = (g, S)$ is
\begin{equation}
\maj(\sigma) = \maj(g, S) = c(\sigma) + r \cdot \left[  \sum_{i \in \Des(g)} i 
- \sum_{i \in S} |\Des(g) \cap \{i, i+1, \dots, n\}| \right],
\end{equation}
where $c(\sigma)$ denotes the sum of the colors in $\sigma$.
In the example above, we have
\begin{equation*}
\maj(3^0_* 4^1 \, \, 6^2 \, \, 2^0_* 5^2_* 1^2) =
(0 + 1 + 2 + 0 + 2 + 2) + 3 \cdot [ (1 + 2 + 4 + 5) - (4 + 2 + 1) ] = 22.
\end{equation*}

Whereas the definition of $\maj$ for colored ordered set partitions used the order $<$ to compare elements, 
the definition of $\coinv$ uses the order $\prec$.  In particular, let $\sigma$ be a colored ordered set partition.
A {\em coinversion pair} in $\sigma$ is a pair of colored letters $i^c \preceq j^d$ appearing in $\sigma$ such that 
\begin{equation*}
\begin{cases}
\text{at least one of $i^c$ and $j^d$ is $\prec$-minimal in its block in $\sigma$,} \\
\text{$i^c$ and $j^d$ belong to different blocks of $\sigma$, and} \\
\text{if $i^c$'s block is to the right of $j^d$'s block, then only $j^d$ is $\prec$-minimal in its block.}
\end{cases}
\end{equation*}
In our example $\sigma = (3^0 4^1 \mid 6^2 \mid 1^2 2^0 5^2 )$, 
the 
coinversion pairs are  $3^0 6^2,  2^0 3^0 , 3^0 5^2, 2^0 6^2, 4^1 6^2,$ and $5^2 6^2$.
The statistic $\coinv(\sigma)$ is defined by
\begin{equation}
\coinv(\sigma) = [n\cdot(r-1) - c(\sigma)] + r \cdot (\text{number of coinversion pairs in $\sigma$}).
\end{equation}
In our example we have
\begin{equation*}
\coinv(3^0 4^1 \mid 6^2 \mid 1^2 2^0 5^2 ) = [6 \cdot 2 - (0 + 1 + 2 + 2 + 0 + 2)] + 3 \cdot 6 = 23.
\end{equation*}
In particular, whereas the statistic $\maj$ involves a sum over colors, the statistic $\coinv$ involves a sum
over {\em complements} of colors.
The statistic $\coinv$ on $r$-colored $k$-block ordered set partitions of $[n]$ is complementary to the statistic
$\inv$ defined in \cite[Sec. 4]{Rhoades}.

We need an extension of colored set partitions involving repeated letters.
An {\em $r$-colored ordered multiset partition} $\mu$ is a sequence of finite nonempty 
sets $\mu = (M_1, \dots, M_k)$ of elements from the alphabet $\AAA_r$.
The {\em size} of $\mu$ is $|M_1| + \cdots |M_k|$ and we say that $\mu$ has {\em $k$ blocks}.
For example, we have that $\mu = (2^1 2^0 3^1  \mid 1^2 3^1 \mid 2^0 4^2 )$ is a 
$3$-colored ordered multiset partition
of size $7$ with $3$ blocks.  

We emphasize that the blocks of ordered multiset partitions are {\em sets}; there
are no repeated letters within blocks (although the same letter can occur with different colors within a single block).
If $\mu$ is an ordered multiset partition, the statistics $\coinv(\mu)$ and $\maj(\mu)$ have the same definitions
as in the case of no repeated letters.

\subsection{$G_n$-faces}
To describe the combinatorics of the rings $R_{n,k}$, 
we introduce the following concept of a $G_n$-face.
In the following definition we require $r \geq 2$.

\begin{defn}
\label{g-face}
A {\em $G_n$-face} is an ordered set partition
$\sigma = (B_1 \mid B_2 \mid \cdots \mid B_m)$ of $[n]$ such that the letters in every block of $\sigma$,
with the possible exception of the first block $B_1$, are decorated by the colors $\{0, 1, \dots, r-1\}$.
\end{defn}

Let $\sigma = (B_1 \mid B_2 \mid \cdots \mid B_m)$ be an $G_n$-face.  If the letters in $B_1$ are uncolored, then
$B_1$ is called the {\em zero block} of $\sigma$.  The {\em dimension} of $\sigma$ is the number of nonzero blocks
in $\sigma$.  Let $\FFF_{n,k}$ denote the set of $G_n$-faces of dimension $k$.
For example, if $r = 3$ we have
\begin{align*}
( 2 5  \mid 1^1 3^2 6^2   \mid 4^1 ) &\in \FFF_{6,2} \text{ and } \\
( 2^2 5^1 \mid 1^1  3^2 6^2  \mid 4^1) &\in \FFF_{6,3},
\end{align*}
where the lack of colors on the letters of the first block $\{2,5\}$ of the top face indicates that $\{2,5\}$ is a 
zero block.  When $k = n$, we have $\FFF_{n,n} = \OP_{n,n} = G_n$ as there cannot be a zero block.

The notation {\em face} in Definition~\ref{g-face} comes from the identification of the $k$-dimensional 
$G_n$-faces with the $k$-dimensional faces in the Coxeter complex of $G_n$.
The set $\FFF_{n,k}$ may also be identified with the 
collection of rank $k$ elements in the Dowling lattice $Q_n(\Gamma)$
to a group $\Gamma$ of size $r$ (see \cite{Dowling}).  By considering the possible sizes of zero
blocks, we see that the number of faces in $\FFF_{n,k}$ is 
\begin{equation}
|\FFF_{n,k}| = \sum_{z = 0}^{n-k} {n \choose z} \cdot r^{n-z} \cdot k! \cdot \Stir(n-z,k).
\end{equation}

We will consider an action of the group $G_n$ on $\FFF_{n,k}$.  To describe this action it suffices
to describe the action of permutation matrices $\symm_n \subseteq G_n$
and the diagonal subgroup $\ZZ_r \times \cdots \times \ZZ_r \subseteq G_n$.
If $\pi = \pi_1 \dots \pi_n \in \symm_n$, then 
$\pi$ acts on $G_n$ by swapping letters while preserving colors. 
For example, if $\pi = 614253 \in \symm_6$, then
\begin{equation*}
\pi. (25 \mid 1^1 3^2 6^2 \mid 4^1) = (15 \mid 6^1 4^2 3^2 \mid 2^1) = (15 \mid 3^2 4^2 6^1 \mid 2^1).
\end{equation*}
A diagonal matrix $g = \mathrm{diag}(\zeta^{c_1}, \dots, \zeta^{c_n})$ acts by increasing the color of the letter
$i$ by $c_i$ (mod $r$), while leaving elements in the zero block uncolored.
For example, if $r = 3$ 
an example action of the diagonal matrix $g = \mathrm{diag}(\zeta, \zeta^2, \zeta^2, \zeta, \zeta^2, \zeta) \in G_6$ is
\begin{equation*}
g. (25 \mid 1^1 3^2 6^2 \mid 4^1) = (25 \mid 1^2 3^1 6^0 \mid 4^2).
\end{equation*}
It is clear that the action of $G_n$ on $\FFF_{n,k}$ preserves the subset $\OP_{n,k}$ of $r$-colored
ordered set partitions.

We extend the definition of $\coinv$ to $G_n$-faces as follows.
There is a natural map
\begin{equation}
\pi: \FFF_{n,k} \rightarrow \bigcup_{z = 0}^{n-k} \OP_{n-z,k}
\end{equation}
which removes the zero block $Z$ of a $G_n$-face
 (if present), and then maps the letters in $[n] - Z$ onto $\{1, 2, \dots, n - |Z| \}$ via an order-preserving
bijection while preserving colors.  For example, we have
\begin{equation*}
\pi:  (2 5  \mid  1^1 3^2 6^2   \mid 4^1)  \mapsto (1^1 2^2 4^2   \mid 3^1).
\end{equation*}
If $\sigma$ is a $G_n$-face whose zero block has size $z$, we define $\coinv(\sigma)$ by
\begin{equation}
\coinv(\sigma) := krz + \coinv(\pi(\sigma)).
\end{equation}
In the $r = 3$ example above, we have
\begin{equation*}
\coinv(2 5  \mid  1^1 3^2 6^2   \mid 4^1) = 2 \cdot 3 \cdot 2 + \coinv(1^1 2^2 4^2   \mid 3^1) = 12 + 8 = 20.
\end{equation*}

\subsection{Symmetric functions}
For $n \geq 0$,
a {\em (weak) composition of $n$} is a sequence $\alpha = (\alpha_1, \dots, \alpha_k)$ 
of nonnegative integers with $\alpha_1 + \cdots + \alpha_k = n$.
We write $\alpha \models n$ or $|\alpha| = n$ to indicate that $\alpha$ is a composition of $n$.

A {\em partition of $n$} is a composition $\lambda$ of $n$ whose parts are positive and weakly 
decreasing.  We write $\lambda \vdash n$ to indicate that $\lambda$ is a partition of $n$.
If $\lambda$ and $\mu$ are partitions (of any size) we say that $\lambda$ {\em dominates} $\mu$ and 
write $\lambda \geq_{dom} \mu$ if
$\lambda_1 + \cdots + \lambda_i \geq \mu_1 + \cdots  + \mu_i$ for all $i \geq 1$.

The {\em Ferrers diagram} of  a partition $\lambda$ 
(in English notation) consists of $\lambda_i$ left-justified boxes in row $i$.
The Ferrers diagram of $(4,2,2) \vdash 8$ is shown below.
The {\em conjugate} $\lambda'$ of a partition $\lambda$ is obtained by reflecting the Ferrers diagram across
its main diagonal.  For example, we have $(4,2,2)' = (3,3,1,1)$.
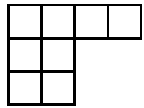
\begin{small}
\begin{center}
\begin{Young}
 & & & \cr
 &  \cr 
 & \cr
\end{Young}
\end{center}
\end{small}

For an infinite sequence of variables $\yy = (y_1, y_2, \dots )$, let $\Lambda(\yy)$ denote the ring of symmetric 
functions in the variable set $\yy$ with coefficients in the field $\QQ(q)$. 
The ring $\Lambda(\yy) = \bigoplus_{n \geq 0} \Lambda(\yy)_n$ is graded by
degree.  The degree $n$ piece $\Lambda(\yy)_n$ has vector space dimension equal to the number of partitions of $n$.

For a partition $\lambda$, let
\begin{center}
$\begin{array}{ccccc}
m_{\lambda}(\yy),  & e_{\lambda}(\yy), & h_{\lambda}(\yy), & s_{\lambda}(\yy)
\end{array}$
\end{center}
be the corresponding {\em monomial, 
elementary, (complete) homogeneous,} and {\em Schur} symmetric functions.
As $\lambda$ varies over the collection of all partitions, these symmetric functions give four different bases
for $\Lambda(\yy)$.
Given any composition $\beta$ whose nonincreasing rearrangement is the partition $\lambda$,
we extend this notation by setting $e_{\beta}(\yy) := e_{\lambda}(\yy)$ and
$h_{\beta}(\yy) := h_{\lambda}(\yy)$.

Let $\omega: \Lambda(\yy) \rightarrow \Lambda(\yy)$ be the linear map
which sends 
$s_{\lambda}(\yy)$ to $s_{\lambda'}(\yy)$ for all partitions $\lambda$.  
The map $\omega$ is an involution and a ring automorphism.  For any partition $\lambda$,
we have $\omega(e_{\lambda}(\yy)) = h_{\lambda}(\yy)$ and
$\omega(h_{\lambda}(\yy)) = e_{\lambda}(\yy)$.

We let $\langle \cdot, \cdot \rangle$ denote the {\em Hall inner product} on $\Lambda(\yy)$.  This can be 
defined by either of the rules $\langle s_{\lambda}(\yy), s_{\mu}(\yy) \rangle = \delta_{\lambda,\mu}$
or $\langle h_{\lambda}(\yy), m_{\mu}(\yy) \rangle = \delta_{\lambda, \mu}$ for all partitions $\lambda, \mu$.
If $F(\yy) \in \Lambda(\yy)$ is any symmetric function, let $F(\yy)^{\perp}$ be the linear operator on
$\Lambda(\yy)$ which is adjoint to the operation of multiplication by $F(\yy)$.  That is, we have
\begin{equation}
\langle F(\yy)^{\perp} G(\yy), H(\yy) \rangle = \langle G(\yy), F(\yy) H(\yy) \rangle
\end{equation}
for all symmetric functions $G(\yy), H(\yy) \in \Lambda(\yy)$.

The representation theory of $G_n$ is analogous to that of  $\symm_n$,
but involves $r$-tuples of objects.  
Given any  $r$-tuple $\bm{o} = (o^{(1)}, o^{(2)}, \dots, o^{(r-1)}, o^{(r)})$ of objects, we define the {\em dual}
$\bm{o^*}$ to be the $r$-tuple 
\begin{equation}
\bm{o^*} := (o^{(r-1)}, \dots, o^{(2)}, o^{(1)}, o^{(r)})
\end{equation}
obtained by reversing the first $r-1$ terms in the sequence $\bm{o}$. 
At the algebraic level, the operator $\bm{o} \mapsto \bm{o^*}$ corresponds to the 
entrywise action of
 complex conjugation
on matrices in $G_n$ (which is trivial when $r = 1$ or $r = 2$).
If $1 \leq i \leq r$, we define the {\em dual} $i^*$ of $i$ by the rule
\begin{equation}
i^* = \begin{cases}
r-i & 1 \leq i \leq r-1 \\
r & i = r.
\end{cases}
\end{equation}
We therefore have
\begin{equation}
\bm{o^*} = (o^{(1^*)}, \dots, o^{(r^*)}) \text{ if } \bm{o} = (o^{(1)}, \dots, o^{(r)}).
\end{equation}

For a positive integer $n$, an {\em $r$-composition} 
$\bm{\alpha}$ of $n$ is an $r$-tuple of compositions 
$\bm{\alpha} = (\alpha^{(1)}, \dots, \alpha^{(r)})$ which satisfies 
$|\bm{\alpha}| := |\alpha^{(1)}| + \cdots + |\alpha^{(r)}| = n$.
We write $\bm{\alpha} \models_r n$ to indicate that $\bm{\alpha}$ is an $r$-composition of $n$.

Similarly, an {\em $r$-partition}
$\blambda = (\lambda^{(1)}, \dots, \lambda^{(r)})$ of $n$ is an $r$-tuple of partitions with
$|\blambda| := |\lambda^{(1)}| + \cdots + |\lambda^{(r)}| = n$.  
We write $\blambda \vdash_r n$ to mean that $\blambda$ is an $r$-partition of $n$.
The {\em conjugate} of an $r$-partition $\blambda = (\lambda^{(1)}, \dots, \lambda^{(r)})$ 
is defined componentwise;
$\bm{\lambda'} := (\lambda^{(1)'}, \dots, \lambda^{(r)'})$.

The {\em Ferrers diagram} of an $r$-partition $\bm{\lambda} = (\lambda^{(1)}, \dots, \lambda^{(r)})$ 
is  the $r$-tuple of Ferrers diagrams of its constituent partitions.  The Ferrers diagram of the 
$3$-partition $((3,2), \varnothing, (2,2)) \vdash_3 9$ is shown below.
\begin{center}
\begin{small}
\begin{Young}
 & & \cr
  & \cr
\end{Young},  \, \,
\begin{large}$\varnothing$\end{large},  \, \,
\begin{Young}
 &  \cr
&  \cr
\end{Young}
\end{small}
\end{center}

Let $\blambda = (\lambda^{(1)}, \dots, \lambda^{(r)}) \vdash_r n$ be an $r$-partition of $n$.
A {\em semistandard tableau $\bT$ of shape $\blambda$} is a tuple
$\bT = (T^{(1)}, \dots, T^{(r)})$, where $T^{(i)}$ is a filling of the boxes of $\lambda^{(i)}$ with positive integers
which increase weakly across rows and strictly down columns.  
A semistandard tableau $\bT$ of shape $\blambda$ is {\em standard} if the entries 
$1, 2, \dots, n$ all appear precisely once in $\bT$.
Let $\SYT^r(n)$ denote the collection of all possible standard tableaux with $r$ components and $n$
boxes.

For example, let $\blambda= ((3,2), \varnothing, (2,2)) \vdash_3 9$. 
A semistandard tableau $\bT = (T^{(1)}, T^{(2)}, T^{(3)})$ of shape $\blambda$ is 
\begin{center}
\begin{small}
\begin{Young}
 1 & 3 & 3\cr
 3 & 4 \cr
\end{Young},  \, \,
\begin{large}$\varnothing$\end{large},  \, \,
\begin{Young}
1 & 3  \cr
4 & 4  \cr
\end{Young}
\end{small}.
\end{center}
A  standard tableau of shape $\blambda$ is
\begin{center}
\begin{small}
\begin{Young}
 3 & 6 & 9\cr
 5 & 7 \cr
\end{Young},  \, \,
\begin{large}$\varnothing$\end{large},  \, \,
\begin{Young}
1 & 4  \cr
2 &8  \cr
\end{Young}
\end{small}.
\end{center}

Let $\bm{T} = (T^{(1)}, \dots, T^{(r})) \in \SYT^r(n)$ 
be a standard tableau with $n$ boxes.  A letter $1 \leq i \leq n-1$ is called a 
{\em descent} of $\bm{T}$ if 
\begin{itemize}
\item the letters $i$ and $i+1$ appear in the same component $T^{(j)}$ of $\bm{T}$, and
 $i+1$ appears in a row below $i$
in $T^{(j)}$, or
\item the letter $i+1$ appears in a component of $\bm{T} = (T^{(1)}, \dots, T^{(r)})$ strictly to the right of the component 
containing $i$.
\end{itemize}
We let $\Des(\bm{T}) := \{ 1 \leq i \leq n \,:\, \text{$i$ is a descent of $\bm{T}$} \}$  denote the collection of all
descents of $\bm{T}$ and let $\des(\bm{T}) := | \Des(\bm{T}) |$ denote the number of descents of $T$.
The {\em major index} of $\bm{T}$ is
\begin{equation}
\maj(\bm{T}) := r \cdot \sum_{i \in \Des(\bm{T})} i +  \sum_{j = 1}^r (j-1) \cdot |T^{(j)}|,
\end{equation}
where $|T^{(j)}|$ is the number of boxes in the component $T^{(j)}$.
For example, if $\bm{T} = (T^{(1)}, T^{(2)}, T^{(3)})$ is the standard tableau above, then
$\Des(\bm{T}) = \{1,3,6,7\}, \des(\bm{T}) = 4$, and
\begin{equation*}
\maj(\bm{T}) = 3 \cdot (1 + 3 + 6 + 7) + (0 \cdot 5 + 1 \cdot 0 + 2 \cdot 4) = 59.
\end{equation*}

For $1 \leq i \leq r$, let $\xx^{(i)} = (x_1^{(i)}, x_2^{(i)}, \dots )$ be an infinite list of variables and let
$\Lambda(\xx^{(i)})$ be the ring of symmetric functions in the variables $\xx^{(i)}$ with coefficients in $\QQ(q)$.
We use $\xx$ to denote the union of the $r$ variable sets $\xx^{(1)}, \dots, \xx^{(r)}$.
Let $\Lambda^r(\xx)$ be the  tensor product
\begin{equation*}
\Lambda^r(\xx) =  \Lambda(\xx^{(1)}) \otimes \cdots \otimes \Lambda(\xx^{(r)}).
\end{equation*}
We can think of $\Lambda^r(\xx)$ as the ring of formal power series in $\QQ(q)[[\xx]]$ which are symmetric 
in the variable sets $\xx^{(1)}, \dots, \xx^{(1)}$ separately.

The algebra $\Lambda^r(\xx)$ is spanned by generating tensors of the form
\begin{equation*}
F_1(\xx^{(1)}) \cdot \ldots \cdot F_r(\xx^{(r)}) := 
F_1(\xx^{(1)}) \otimes \cdots \otimes F_r(\xx^{(r)}),
\end{equation*}
where $F_i(\xx^{(i)}) \in \Lambda(\xx^{(i)})$ is a symmetric function in the variables $\xx^{(i)}$.
The algebra $\Lambda^r(\xx)$
is graded via
\begin{equation*}
\deg(F_1(\xx^{(1)}) \cdot  \ldots \cdot F_{r}(\xx^{(r)})) :=
\deg(F_1(\xx^{(1)})) + \cdots + \deg(F_{r}(\xx^{(r)})),
\end{equation*}
where the $F_i(\xx^{(i)})$ are homogeneous.

The standard bases of $\Lambda^r(\xx)$ are obtained from those of 
$\Lambda(\xx^{(1)}), \dots, \Lambda(\xx^{(r)})$ by multiplication.  More precisely,
let $\bm{\lambda} = (\lambda^{(1)}, \dots, \lambda^{(r)})$ be an $r$-partition.
We define elements 
\begin{equation*}
\bm{m_{\lambda}}(\xx), \bm{e_{\lambda}}(\xx), \bm{h_{\lambda}}(\xx), \bm{s_{\lambda}}(\xx) 
\in \Lambda^r(\xx)
\end{equation*}
 by
\begin{center}
$\begin{array}{cc}
 \bm{m_{\lambda}}(\xx) := m_{\lambda^{(1)}}(\xx^{(1)}) \cdots m_{\lambda^{(r)}}(\xx^{(r)}), & 
 \bm{e_{\lambda}}(\xx) := e_{\lambda^{(1)}}(\xx^{(1)}) \cdots e_{\lambda^{(r)}}(\xx^{(r)}),  \\
 \bm{h_{\lambda}}(\xx) := h_{\lambda^{(1)}}(\xx^{(1)}) \cdots h_{\lambda^{(r)}}(\xx^{(r)}),  
 & \bm{s_{\lambda}}(\xx) := s_{\lambda^{(1)}}(\xx^{(1)}) \cdots s_{\lambda^{(r)}}(\xx^{(r)}).
\end{array}$
\end{center}
As $\blambda$ varies over the collection of all $r$-partitions, any of the sets 
$\{ \bm{m_{\lambda}}(\xx) \}, \{ \bm{e_{\lambda}}(\xx) \}, \{ \bm{h_{\lambda}}(\xx) \},$ or
$\{ \bm{s_{\lambda}}(\xx) \}$ forms a basis for $\Lambda^r(\xx)$.
If $\bbeta = (\beta^{(1)}, \dots, \beta^{(r)})$ is an $r$-composition, we extend this notation by setting
\begin{center}
$\begin{array}{cc}
\bm{e_{\beta}}(\xx) := e_{\beta^(1)}(\xx^{(1)}) \cdots e_{\beta^{(r)}}(\xx^{(r)}), &
\bm{h_{\beta}}(\xx) := h_{\beta^(1)}(\xx^{(1)}) \cdots h_{\beta^{(r)}}(\xx^{(r)}). 
\end{array}$
\end{center}

The Schur functions $\bm{s_{\lambda}}(\xx)$  admit the following combinatorial
description.  If $\bT = (T^{(1)}, \dots, T^{(r)})$ is a semistandard tableau with $r$ components, let 
$\xx^{\bT}$ be the monomial in the variable set $\xx$ where the exponent of $x^{(i)}_j$ equals the multiplicity of $j$
in the tableau $T^{(i)}$.  
For example, if $r = 3$ and $\bT = (T^{(1)}, T^{(2)}, T^{(3)})$ is as above, we have
\begin{equation*}
\xx^{\bT} = (x^{(1)}_1)^1 (x^{(1)}_3)^3 (x^{(1)}_4)^1 (x^{(3)}_1)^1 (x^{(3)}_3)^1 (x^{(3)}_4)^2.
\end{equation*}
Similarly, if $w$ is any word in the $r$-colored positive integers $\AAA_r$, let $\xx^w$ be the monomial in $\xx$
where the exponent of $x^{(i)}_j$ equals the multiplicity of $j^{i-1}$ in the word $w$.
Also, if
$\bbeta = (\beta^{(1)}, \dots, \beta^{(r)})$ is an $r$-composition, define the monomial
$\xx^{\bbeta}$ by
\begin{equation}
\xx^{\bbeta} := (x^{(1)}_1)^{\beta^{(1)}_1} (x^{(1)}_2)^{\beta^{(1)}_2} \cdots (x^{(2)}_1)^{\beta^{(2)}_1} 
(x^{(2)}_2)^{\beta^{(2)}_2} \cdots 
\end{equation}
Given an $r$-partition $\blambda \vdash_r n$, we have
\begin{equation}
\bm{s_{\lambda}}(\xx) = \sum_{\bT} \xx^{\bT},
\end{equation}
where the sum is over all semistandard tableaux $\bT$ of shape $\blambda$.

The Hall inner product $\langle \cdot, \cdot \rangle$ extends to $\Lambda^r(\xx)$ by the rule
\begin{equation}
\langle \bm{s_{\lambda}}(\xx), \bm{s_{\mu^*}}(\xx) \rangle =
\langle \bm{h_{\lambda}}(\xx), \bm{m_{\mu^*}}(\xx) \rangle = \delta_{\blambda, \bm{\mu}}
\end{equation}
for all $r$-partitions $\blambda$ and $\bm{\mu}$. 
The presence of duals in this definition comes from the nontriviality of complex conjugation on 
$G_n$ for $r > 2$.

The involution $\omega$ is defined on $\Lambda^r(\xx) = \Lambda(\xx^{(1)}) \otimes \cdots \otimes \Lambda(\xx^{(r)})$
by applying $\omega$ in each component separately.  
The map $\omega$ is  an isometry of the inner product $\langle \cdot, \cdot \rangle$.

If $\bm{F(x)} \in \Lambda^r(\xx)$, 
we let $\bm{F(x)}^{\perp}$ be the operator on $\Lambda^r(\xx)$ which is adjoint to multiplication by $\bm{F(x)}$ under the 
inner product $\langle \cdot, \cdot \rangle$.  In particular, if $j \geq 1$ and if $1 \leq i \leq r$, we have 
$h_j(\xx^{(i)}), e_j(\xx^{(i)}) \in \Lambda^r(\xx)$, so that 
$h_j(\xx^{(i)})^{\perp}$ and $e_j(\xx^{(i)})^{\perp}$ make sense as linear operators on $\Lambda^r(\xx)$.
These operators (and their `dual' versions
$h_j(\xx^{(i^*)})^{\perp}$ and $e_j(\xx^{(i^*)})^{\perp}$)
 will play a key role in this paper.

\subsection{Representations of $G_n$}  
In his thesis, Specht \cite{Specht} described the irreducible representations of $G_n$.
We recall his construction.

Given a matrix $g \in G_n$, 
define numbers $\chi(g)$ and $\sign(g)$ by 
\begin{align}
\chi(g) &:= \text{product of the nonzero entries in $g$}, \\
\sign(g) &:= \text{determinant of the permutation matrix underlying $g$}.
\end{align}
In particular, the number $\chi(g)$ is an $r^{th}$ root of unity and $\sign(g) = \pm 1$.  Both of the functions 
$\chi$ and $\sign$ are linear characters of $G_n$.  In other words, we have
$\chi(gh) = \chi(g) \chi(h)$ and $\sign(g h) = \sign(g) \sign(h)$ for all $g, h \in G_n$.

It is well known that the irreducible 
complex representations of the 
symmetric group $\symm_n$ are indexed by partitions $\lambda \vdash n$.  Given $\lambda \vdash n$,
let $S^{\lambda}$ be the corresponding irreducible $\symm_n$-module.
For example, we have that $S^{(n)}$ is the trivial representation of $\symm_n$ and $S^{(1^n)}$ is the sign 
representation of $\symm_n$.

Let $V$ be a $G$-module and let $U$ be an $\symm_n$-module.  We build a 
$G_n$-module $V \wr U$ by letting $V \wr U = (V)^{\otimes n} \otimes U$ as a vector space and defining
the action of $G_n$ by
\begin{equation}
\mathrm{diag}(g_1, \dots, g_n).(v_1 \otimes \cdots \otimes v_n \otimes u) := 
(g_1.v_1) \otimes \cdots \otimes (g_n.v_n) \otimes u,
\end{equation}
for all diagonal matrices $\mathrm{diag}(g_1, \dots, g_n) \in G_n$, and
\begin{equation}
\pi.(v_1 \otimes \cdots \otimes v_n \otimes u) := v_{\pi^{-1}_1} \otimes \cdots \otimes v_{\pi^{-1}_n} \otimes (\pi.u),
\end{equation}
for all $\pi \in \symm_n \subseteq G_n$.  If $V$ is an irreducible $G$-module and $U$ 
is an irreducible $\symm_n$-module, then $V \wr U$ is an irreducible $G_n$-module, but not all of the
irreducible $G_n$-modules arise in this way.

For any  composition $\alpha = (\alpha_1, \dots , \alpha_r) \models n$ with $r$ parts,
 the parabolic subgroup of 
block diagonal matrices in $G_n$ with block sizes $\alpha_1, \dots, \alpha_r$ 
gives an inclusion 
\begin{equation}
G_{\alpha} :=
G_{\alpha_1} \times \cdots \times G_{\alpha_r} \subseteq G_n.
\end{equation}
 If $W_i$ is a $G_{\alpha_i}$-module for $1 \leq i \leq r$, the tensor product 
 $W_1 \otimes \cdots \otimes W_r$ is a $G_{\alpha}$-module and
the induction $\Ind_{G_{\alpha}}^{G_n}(W_1 \otimes \cdots \otimes W_r)$ 
 is a $G_n$-module.

We index the irreducible representations of the cyclic group
$G = \ZZ_r = \langle \zeta \rangle$ in the following slightly nonstandard way.
For $1 \leq i \leq r$, 
let $\rho_i: G \rightarrow GL_1(\CC) = \CC^{\times}$ be the homomorphism
\begin{equation}
\rho_i: \zeta \mapsto \zeta^{-i}.
\end{equation}
and let $V_i$ be the vector space $\CC$ with $G$-module structure given by $\rho_i$.
In particular, we have that $V_r$ is the trivial representation of $G$ and
$V_1, V_2, \dots, V_{r-1}$ are the nontrivial irreducible representations of $G$.

The irreducible modules for $G_n$ are indexed by $r$-partitions of $n$.
If $\bm{\lambda} = (\lambda^{(1)}, \dots, \lambda^{(r)}) \vdash_r n$ is an $r$-partition of $n$, let 
$\alpha = (\alpha_1, \dots, \alpha_r) \models n$ be the composition whose parts are $\alpha_i := |\lambda^{(i)}|$.
Define 
$\bm{S^{\lambda}}$ to be the 
$G_n$-module given by
\begin{equation}
\bm{S^{\lambda}} := \Ind_{G_{\alpha}}^{G_n}
((V_1 \wr S^{\lambda^{(1)}}) \otimes \cdots \otimes (V_r \wr S^{\lambda^{(r)}})).
\end{equation}
Specht proved that the set $\{ \bm{S^{\lambda}} \,:\, \bm{\lambda} \vdash_r n \}$ forms a complete set of 
nonisomorphic irreducible representations of $G_n$.

\begin{example}
For any $1 \leq i \leq r$, both of the functions 
\begin{equation}
\begin{cases}
\chi^i: g \mapsto (\chi(g))^i \\
\sign \cdot \chi^i: g \mapsto \sign(g) \cdot (\chi(g))^i
\end{cases}
\end{equation}
on $G_n$ are linear characters.
We leave it for the reader to check that under the above classification we have
\begin{center}
$\begin{array}{cccc}
\chi^1 \leftrightarrow ((n), \varnothing, \dots, \varnothing), & & 
\sign \cdot \chi^1 \leftrightarrow ((1^n), \varnothing \dots, \varnothing), \\
\chi^2 \leftrightarrow (\varnothing, (n), \dots,  \varnothing), & & 
\sign \cdot \chi^2 \leftrightarrow (\varnothing, (1^n), \dots, \varnothing),  \\ 
\vdots & & \vdots \\
\chi^r \leftrightarrow (\varnothing,  \varnothing,  \dots, (n)), & &
\sign \cdot \chi^r \leftrightarrow (\varnothing, \varnothing, \dots, (1^n)). 
\end{array}$
\end{center}
Since $\chi^r$ is the trivial character of $G_n$, the trivial representation
therefore corresponds to the $r$-partition $(\varnothing, \dots, \varnothing, (n))$.
\end{example}

Let $V$ be a finite-dimensional $G_n$-module.  There exist unique integers $m_{\bm{\lambda}}$
such that 
\begin{equation*}
V \cong \bigoplus_{\bm{\lambda} \vdash_r n} (\bm{S^{\lambda}})^{m_{\bm{\lambda}}}.
\end{equation*}
The {\em Frobenius character}  $\Frob(V) \in \Lambda^r(\xx)$ of $V$ is given by
\begin{equation}
\Frob(V) := \sum_{\bm{\lambda} \vdash_r n} m_{\bm{\lambda}} \bm{s_{\lambda}}(\xx).
\end{equation}
In particular, the multiplicity $m_{\bm{\lambda}}$ of $\bm{S^{\lambda}}$ in $V$ is 
$\langle \Frob(V), \bm{s_{\lambda^*}}(\xx) \rangle$.

More generally, if $V = \oplus_{d \geq 0} V_d$ is a graded $G_n$-module with
each $V_d$ finite-dimensional, the 
{\em graded Frobenius character}  $\grFrob(V;q) \in \Lambda^r(\xx)[[q]]$ of $V$ is
\begin{equation}
\grFrob(V;q) := \sum_{d \geq 0} \Frob(V_d) \cdot q^d.
\end{equation}
Also recall that 
the {\em Hilbert series} $\Hilb(V;q)$ of $V$ is 
\begin{equation}
\Hilb(V;q) := \sum_{d \geq 0} \dim(V_d) \cdot q^d.
\end{equation}

The Frobenius character is compatible with induction product in the following way.  
Let $V$ be an $G_n$-module and let $W$ be a $G_m$ module.
The tensor product $V \otimes W$ is a $G_{(n,m)}$-module, so that 
$\Ind_{G_{(n,m)}}^{G_{n+m}} (V \otimes W)$ is a
$G_{n+m}$-module.
We have
\begin{equation}
\Frob(\Ind_{G_{(n,m)}}^{G_{n+m}} (V \otimes W)) = 
\Frob(V) \cdot \Frob(W),
\end{equation}
where the multiplication on the right-hand side takes place within $\Lambda^r(\xx)$.

\subsection{Gr\"obner theory}
A total order $<$ on the monomials in $\CC[\xx_n]$ is called a {\em monomial order} if 
\begin{itemize}
\item  $1 \leq m$ for every monomial $m \in \CC[\xx_n]$, and
\item  $m \leq m'$ implies $m \cdot m'' \leq m' \cdot m''$ for all monomials $m, m', m'' \in \CC[\xx_n]$.
\end{itemize}
In this paper we will only use the {\em lexicographic} monomial order defined by
$x_1^{a_1} \cdots x_n^{a_n} < x_1^{b_1} \cdots x_n^{b_n}$ if there exists $1 \leq i \leq n$ such that
$a_1 = b_1, \dots, a_{i-1} = b_{i-1}$, and $a_i < b_i$.

If $f \in \CC[\xx_n]$ is a nonzero polynomial and $<$ is a monomial order, let $\initial_<(f)$ be the leading term of 
$f$ with respect to the order $<$.  If $I \subseteq \CC[\xx_n]$ is an ideal, the corresponding {\em initial ideal}
$\initial_<(I) \subseteq \CC[\xx_n]$ is the monomial ideal in $\CC[\xx_n]$ generated by the leading terms
of every nonzero polynomial in $I$:
\begin{equation}
\initial_<(I) := \langle \initial_<(f) \,:\, f \in I - \{0\} \rangle.
\end{equation}
The collection of monomials $m \in \CC[\xx_n]$ which are not contained in $\initial_<(I)$, namely
\begin{equation}
\{ \text{monomials $m \in \CC[\xx_n]$} \,:\, \initial(f) \nmid m \text{ for all $f \in I - \{0\}$} \}
\end{equation}
descends to a vector space basis for the quotient $\CC[\xx_n]/I$.  This is called the 
{\em standard monomial basis}.

A finite subset $B = \{g_1, \dots, g_m\} \subseteq I$ of nonzero polynomials in $I$ is called a {\em Gr\"obner basis}
of $I$ if $\initial_<(I) = \langle \initial_<(g_1), \dots, \initial_<(g_m) \rangle$.  A Gr\"obner basis $B$ 
is called {\em reduced} if 
\begin{itemize}
\item  the leading coefficient of $g_i$ is $1$ for all $1 \leq i \leq m$, and
\item  for $i \neq j$, the monomial $\initial(g_i)$ does not divide any of the terms appearing in $g_j$.
\end{itemize}
After fixing a monomial order, every ideal $I \subseteq \CC[\xx_n]$ has a unique reduced Gr\"obner basis.

\section{Polynomial identities}
\label{Polynomial}

In this section we prove a family of polynomial and symmetric function identities which 
will be useful in our analysis of the rings $R_{n,k}$ and $S_{n,k}$.  
The first of these identities is the $G_n$-analog of \cite[Lem. 3.1]{HRS}.

\begin{lemma}
\label{alternating-sum-lemma}
Let $k \leq n$,
let $\alpha_1, \dots, \alpha_k \in \CC$ be distinct  complex numbers, and let $\beta_1, \dots, \beta_n \in \CC$
be complex numbers with the property that $\{\alpha_1, \dots, \alpha_k\} \subseteq \{\beta_1^r, \dots, \beta_n^r \}$.
For any $n-k+1 \leq s \leq n$ we have
\begin{equation}
\sum_{j = 0}^{s} (-1)^{j} e_{s-j}(\beta_1^r, \dots, \beta_n^r) h_j(\alpha_1, \dots, \alpha_k) = 0.
\end{equation}
\end{lemma}

\begin{proof}
The left-hand side is the coefficient of $t^s$ in the power series
\begin{equation}
\frac{\prod_{i = 1}^n (1 + t \beta_i^r)}{\prod_{i = 1}^k (1 + t \alpha_i)}.
\end{equation}
By assumption, every term in the denominator cancels with a distinct term in the numerator, so that this expression
is a polynomial in $t$ of degree $n-k$.  Since $s > n-k$, the coefficient of $t^s$ in this polynomial is $0$.
\end{proof}

In practice, our applications of Lemma~\ref{alternating-sum-lemma} will always involve one of the two
situations $\{\beta_1^r, \dots, \beta_n^r\} = \{\alpha_1, \dots, \alpha_k\}$ or 
$\{\beta_1^r, \dots, \beta_n^r\} = \{\alpha_1, \dots, \alpha_k, 0 \}$.

Let $\gamma = (\gamma_1, \dots, \gamma_n) \models n$ be a  composition with $n$ parts.
The {\em Demazure character} $\kappa_{\gamma}(\xx_n) \in \CC[\xx_n]$ is defined recursively
as follows.  If $\gamma_1 \geq \cdots \geq \gamma_n$, we let 
$\kappa_{\gamma}(\xx_n)$ be the monomial
\begin{equation}
\kappa_{\gamma}(\xx_n) = x_1^{\gamma_1} \cdots x_n^{\gamma_n}.
\end{equation}
In general, if $\gamma_i < \gamma_{i+1}$, we let
\begin{equation}
\kappa_{\gamma}(\xx_n) = \frac{ x_i (\kappa_{\gamma'}(\xx_n)) - x_{i+1} (s_i \cdot \kappa_{\gamma'}(\xx_n))}{x_i - x_{i+1}},
\end{equation}
where $\gamma' = (\gamma_1, \dots, \gamma_{i+1}, \gamma_i, \dots, \gamma_n)$ is the 
composition obtained by interchanging the $i^{th}$ and $(i+1)^{st}$ parts of $\gamma$
and $s_i \cdot \kappa_{\gamma'}(\xx_n)$ is the polynomial $\kappa_{\gamma'}(\xx_n)$ with $x_i$
and $x_{i+1}$ interchanged.
It can be shown that this recursion gives a well defined collection of polynomials 
$\{ \kappa_{\gamma}(\xx_n) \}$ indexed by compositions $\gamma$ with $n$ parts.
This set forms a basis for the polynomial ring $\CC[\xx_n]$.

Demazure characters played a key role  in \cite{HRS};
they will be equally important here.
In order to state the $G_n$-analogs of the lemmata from \cite{HRS} that we will need, we must 
introduce some notation.

\begin{defn}
Let $S = \{s_1 < s_2 < \cdots < s_m\} \subseteq [n]$.  The {\em skip monomial} $\xx(S)$ in $\CC[\xx_n]$ is 
\begin{equation*}
\xx(S) := x_{s_1}^{s_1} x_{s_2}^{s_2 - 1} \cdots x_{s_m}^{s_m - m + 1}.
\end{equation*}
The {\em skip composition} $\gamma(S) = (\gamma_1, \dots, \gamma_n)$ is the length $n$ composition defined by
\begin{equation*}
\gamma_i = \begin{cases}
0 & i \notin S \\
s_j - j + 1 & i = s_j \in S.
\end{cases} 
\end{equation*}
We also let $\overline{\gamma(S)} := (\gamma_n, \dots, \gamma_1)$ be the reverse of the skip composition $\gamma(S)$.
\end{defn}

For example, if $n = 8$ and $S = \{2,3,5,8\}$, then $\gamma(S) = (0,2,2,0,3,0,0,5)$ and
$\xx(S) = x_2^2 x_3^2 x_5^3 x_8^5$.
In general, we have that $\gamma(S)$ is the exponent vector of $\xx(S)$.
We will be interested in the $r^{th}$ powers $\xx(S)^r$ of  skip monomials in this paper.

Skip monomials are related to Demazure characters as follows.
For any polynomial $f(\xx_n) = f(x_1, \dots, x_n) \in \CC[\xx_n]$, let
$f(\xx_n^r) = f(x_1^r, \dots, x_n^r)$ and $\overline{f(\xx_n^r)} = f(x_n^r, \dots, x_1^r)$.
The following result is immediate from 
\cite[Lem. 3.5]{HRS} after the change of variables
$(x_1, \dots, x_n) \mapsto (x_1^r, \dots, x_n^r)$.

\begin{lemma}
\label{demazure-initial-term}
Let $n \geq k$ and let $S \subseteq [n]$ satisfy $|S| = n-k+1$.  Let $<$ be lexicographic order.  We have
\begin{equation}
\initial_<(\overline{\kappa_{\overline{\gamma(S}}(\xx_n^r)}) = \xx(S)^r.
\end{equation}
Moreover, for any $1 \leq i \leq n$ we have
\begin{equation}
x_i^{r \cdot (\max(S)-n+k+1)} \nmid m
\end{equation}
for any monomial $m$ appearing in $\overline{\kappa_{\overline{\gamma(S)}}(\xx_n^r)}$.  Finally, if $T \subseteq [n]$ 
satisfies $|T| = n-k+1$ and $T \neq S$, then $\xx(S)^r \nmid m$ for any monomial $m$ appearing in 
$\overline{\kappa_{\overline{\gamma(T)}}(\xx_n^r)}$.
\end{lemma}

We also record the fact, which follows immediately from \cite{HRS}, that the polynomials
$\kappa_{\gamma(S)^*}(\xx_n^{r,*})$ appearing in Lemma~\ref{demazure-initial-term}
are contained in the ideals $I_{n,k}$ and $J_{n,k}$.
The following result follows from \cite[Eqn. 3.4]{HRS} after the change of variables
$(x_1, \dots, x_n) \mapsto (x_1^r, \dots, x_n^r)$.

\begin{lemma}
\label{demazures-in-ideal}
Let $n \geq k$ and let $S \subseteq [n]$ satisfy $|S| = n-k+1$.  The polynomial
$\overline{\kappa_{\overline{\gamma(S)}}(\xx_n^r)}$ is contained in the ideal
\begin{equation}
\langle e_n(\xx_n^r), e_{n-1}(\xx_n^r), \dots, e_{n-k+1}(\xx_n^r) \rangle \subseteq \CC[\xx_n].
\end{equation}
In particular, we have $\overline{\kappa_{\overline{\gamma(S)}}(\xx_n^r)} \in I_{n,k}$ and 
$\overline{\kappa_{\overline{\gamma(S)}}(\xx_n^{r})} \in J_{n,k}$.
\end{lemma}

We define two formal power series in the infinite variable set 
$\xx = (\xx^{(1)}, \dots, \xx^{(r)})$ using the $\coinv$ and $\comaj$ statistics on 
$r$-colored ordered multiset partitions.
If $\mu$ is an $r$-colored ordered multiset partition, let $\xx^{\mu}$ be the monomial
in the variable set $\xx$ where the exponent of $x_j^{(i)}$ is the number of occurrences 
of $j^{i-1}$ in $\mu$.

\begin{defn}
\label{m-and-i}
Let $r \geq 1$ and let $k \leq n$ be positive integers.  Define two formal power series in the variable set 
$\xx = (\xx^{(1)}, \dots, \xx^{(r)})$ by
\begin{align}
\bm{M_{n,k}}(\xx;q) &:= \sum_{\mu} q^{\maj(\mu)} \xx^{\mu}, \\
\bm{I_{n,k}}(\xx;q) &:= \sum_{\mu} q^{\coinv(\mu)} \xx^{\mu},
\end{align}
where the sum is over all $r$-colored ordered multiset partitions $\mu$ of size $n$ with $k$ blocks.
\end{defn}

The next result establishes that
the formal power series $\bm{M_{n,k}}(\xx;q), \bm{I_{n,k}}(\xx;q)$ in Definition~\ref{m-and-i}
both contained in the ring $\Lambda^r(\xx)$ and are related to each other by $q$-reversal.

\begin{lemma}
\label{m-equals-i}
Both of the formal power series $\bm{M_{n,k}}(\xx;q)$ and $\bm{I_{n,k}}(\xx;q)$ 
lie in the ring $\Lambda^r(\xx)$.  Moreover,
we have
$\bm{M_{n,k}}(\xx;q) = \rev_q (\bm{I_{n,k}}(\xx;q))$. 
\end{lemma}

\begin{proof}
The truth of this statement for $r = 1$ (when $\Lambda^r(\xx)$ is the usual
ring of symmetric functions) follows from the work of Wilson \cite{WMultiset}.  To deduce this statement for
general $r \geq 1$, consider a new countably infinite set of variables
\begin{equation}
\zz = \{z_{i,j} \,:\, j \in \ZZ_{> 0}, 1 \leq i \leq r \}.
\end{equation}
The association $z_{i,j} \leftrightarrow x_j^{(i)}$ gives a bijection with our collection of variables
$\xx = (\xx^{(1)}, \dots, \xx^{(r)})$.  The idea is to reinterpret $\bm{M_{n,k}}(\xx;q)$ and $\bm{I_{n,k}}(\xx;q)$ in terms of the 
new variable set $\zz$, and then apply the equality and symmetry known in the case $r = 1$.

To achieve the program of the preceding paragraph, we introduce the following notation.
Let $\bm{M_{n,k}^1}(\zz;q^r)$ be the formal power series 
\begin{equation}
\bm{M_{n,k}^1}(\zz;q^r) := \sum_{\mu} q^{r \cdot \maj(\mu)} \zz^{\mu},
\end{equation}
where the sum is over all ordered multiset partitions $\mu$ of size $n$ with $k$ blocks 
on the countably infinite alphabet
\begin{equation*}
1^{r-1} < 2^{r-1} < \cdots < 1^{r-2} < 2^{r-1} < \cdots < 1^0 < 2^0 < \cdots 
\end{equation*}
and we compute $\maj(\mu)$ as in the  $r = 1$ case (i.e., ignoring contributions to $\maj$ coming from colors,
and not multiplying descents by  $r$).

Similarly, let $\bm{I^1_{n,k}}(\zz;q^r)$ be the formal power series
\begin{equation}
\bm{I^1_{n,k}}(\zz;q) := \sum_{\mu} q^{r \cdot \coinv(\mu)} \xx^{\mu},
\end{equation}
where the sum is over all ordered multiset partitions $\mu$ of size $n$ with $k$ blocks 
on the countably infinite alphabet
\begin{equation*}
1^{r-1} \prec \cdots \prec 1^0 \prec 2^{r-1} \prec \cdots \prec 2^0 \prec \cdots
\end{equation*}
and we define $\coinv(\mu)$ as in the $r = 1$ case (i.e., ignoring the contribution to
$\coinv$ coming from colors, and not multiplying the number of coinversion pairs by $r$).

It follows from the definition of $\bm{M_{n,k}}(\xx;q)$ that
\begin{equation}
\label{maj-relation}
\bm{M_{n,k}}(\xx;q) =
\bm{M_{n,k}^1}(\zz;q^r) |_{z_{i,j} = q^{i-1} \cdot x_j^{(i)}}.
\end{equation}
This expression for $\bm{M_{n,k}}(\xx;q)$, together with the fact
that $\bm{M_{n,k}^1}(\zz;q^r)$ is symmetric in the $\zz$ variables, proves that 
$\bm{M_{n,k}}(\xx;q) \in \Lambda^r(\xx)$.
Similarly, we have
\begin{equation}
\label{inv-relation}
\bm{I_{n,k}}(\xx;q) =
\bm{I_{n,k}^1}(\zz;q^r) |_{z_{i,j} = q^{r-i} \cdot x_j^{(i)}},
\end{equation}
so that $\bm{I_{n,k}}(\xx;q) \in \Lambda^r(\xx)$.

Applying the lemma in the case  $r = 1$, we have
\begin{align}
\bm{M_{n,k}}(\xx;q) 
&= \bm{M_{n,k}^1}(z_{1,r}, z_{2,r}, \dots, z_{1,r-1}, z_{2,r-1}, \dots , z_{1,1}, z_{2,1}, \dots ;q^r) 
|_{z_{i,j} = q^{i-1} \cdot x_j^{(i)}} 
\\
&= \bm{M_{n,k}^1}(z_{1,r}, \dots, z_{1,1}, z_{2,r}, \dots, z_{2,1}, \dots; q^r) |_{z_{i,j} = q^{i-1} \cdot x_j^{(i)}} \\
&=  \rev_q \left[ \bm{I_{n,k}^1}(z_{1,r}, \dots, z_{1,1}, z_{2,r}, \dots, z_{2,1}, \dots; q^r)  
\right]|_{z_{i,j} = q^{i-1} \cdot x_j^{(i)}} \\
&= \rev_q \left[ \bm{I_{n,k}^1}(z_{1,r}, \dots, z_{1,1}, z_{2,r}, \dots, z_{2,1}, \dots; q^r)|_{z_{i,j} = q^{r-i} \cdot x_j^{(i)}}
\right] \\
&= \rev_q(\bm{I_{n,k}}(\xx;q)).
\end{align}
The first equality is Equation~\ref{maj-relation}, the second equality 
uses the fact that $\bm{M_{n,k}^1}(\zz;q)$ is symmetric in the $\zz$ variables,  the third equality uses the fact that 
$\bm{M_{n,k}^1}(\zz;q) = \rev_q(\bm{I_{n,k}^1}(\zz;q))$,
 the fourth equality interchanges evaluation and $q$-reversal, and the final equality 
 is Equation~\ref{inv-relation}.
\end{proof}

The power series in Lemma~\ref{m-equals-i} will be (up to minor transformations) the graded Frobenius character
of the ring $S_{n,k}$.  
We give this character-to-be a name.

\begin{defn}
Let $r \geq 1$ and let $k \leq n$ be positive integers.  Let $\bm{D_{n,k}}(\xx;q) \in \Lambda^r(\xx)$ be the common 
ring element
\begin{equation}
\bm{D_{n,k}}(\xx;q) := (\rev_q \circ \omega) \bm{M_{n,k}}(\xx;q) =  \omega \bm{I_{n,k}}(\xx;q).
\end{equation}
\end{defn}

As a Frobenius character, the ring element $\bm{D_{n,k}}(\xx;q) \in \Lambda^r(\xx)$ must expand positively in the Schur
basis  $\{ \bm{s_{\lambda}}(\xx) \,:\, \blambda \vdash_r n \}$.  The $\maj$ formulation of $\bm{D_{n,k}}(\xx;q)$ 
is well suited to proving this fact directly, as well as giving the Schur expansion of $\bm{D_{n,k}}(\xx;q)$.
The following proposition is a colored version of a result of Wilson \cite[Thm. 5.0.1]{WMultiset}.

\begin{proposition}
\label{d-schur-expansion}
Let $r \geq 1$ and let $k \leq n$ be positive integers.  We have the Schur expansion
\begin{equation}
\bm{D_{n,k}}(\xx;q) = \rev_q \left[\sum_{\bT \in \SYT^r(n)} q^{\maj(\bT) + r {n-k \choose 2} - r  (n-k)  \des(\bT)} 
{\des(\bT) \brack n-k}_{q^r} \bm{s_{\shape(\bT)'}}(\xx) \right].
\end{equation}
\end{proposition}

\begin{proof}
Consider the collection $\WWW_n$ of all length $n$ words $w = w_1 \dots w_n$ in the alphabet of
$r$-colored positive integers.  
For any word $w \in \WWW_n$, the (colored version of the) {\em RSK correspondence} gives a pair of
$r$-tableaux $(\bm{U}, \bT)$ of the same shape, with $\bm{U}$ semistandard and $\bT$ standard.  
For example, if $r = 3$ and $w = 2^0 1^1 4^1 2^2 1^0 2^0 2^1 1^2 \in \WWW_8$ then 
$w \mapsto (\bm{U}, \bT)$ where
\begin{small}
\begin{equation*}
\bm{U} = \, 
\begin{Young}
1 & 2 \\
2
\end{Young}, \hspace{0.1in}
\begin{Young}
1 & 2 \\
4
\end{Young}, \hspace{0.1in}
\begin{Young}
1 \\ 2
\end{Young} \hspace{0.3in} \bT = \,
\begin{Young}
1 & 6 \\ 5
\end{Young}, \hspace{0.1in}
\begin{Young}
2 & 3 \\ 7 \end{Young}, \hspace{0.1in}
\begin{Young}
4 \\ 8
\end{Young} \, .
\end{equation*} 
\end{small}
The RSK map gives a bijection
\begin{equation}
\WWW_n \xrightarrow{\sim} \left\{ (\bm{U}, \bT) \,:\, 
\begin{array}{c}
\text{$\bm{U}$ a semistandard $r$-tableau with $n$ boxes,} \\
\text{$\bT$ a standard $r$-tableau with $n$ boxes,} \\
\text{$\shape(\bm{U}) = \shape(\bT)$}
\end{array}  \right\}.
\end{equation}
If $w \mapsto (\bm{U}, \bT)$, then $\Des(w) = \Des(\bT)$ so that $\maj(w) = \maj(\bT)$.

For any word $w \in \WWW_n$, we can generate a collection of ${\des(w) \choose n-k}$
$r$-colored ordered multiset partitions $\mu$ as follows. 
 Among the $\des(w)$ descents of $w$, choose $n-k$ of them to star, yielding a 
pair $(w, S)$ where $S \subseteq \Des(w)$ satisfies $|S| = n-k$.  We may identify $(w, S)$ with an 
$r$-colored ordered multiset partition $\mu$.

The above paragraph
implies that
\begin{equation}
\label{first-m-equation}
\bm{M_{n,k}}(\xx;q) = \sum_{w \in \WWW_n^r} q^{\maj(w) + r {n-k \choose 2} - r(n-k)\des(w)} 
{\des(w) \brack n-k}_{q^r} \xx^w,
\end{equation}
where the factor $q^{r {n-k \choose 2} - r(n-k)\des(w)} {\des(w) \brack n-k}_{q^r}$
is generated by the ways in which $n-k$  stars can be placed 
among the $\des(w)$ descents of $w$.

Applying RSK to the right-hand side of Equation~\ref{first-m-equation}, we deduce that
\begin{equation}
\bm{M_{n,k}}(\xx;q) = \sum_{\bT \in \SYT^r(n)} q^{\maj(\bT) - r {n-k \choose 2} + r(n-k)\des(\bT)} 
{ \des(\bT) \brack n-k}_{q^{r}} \bm{s_{\shape(\bT)}}(\xx).
\end{equation}
Since $\bm{D_{n,k}}(\xx;q) = (\rev_q \circ \omega) \bm{M_{n,k}}(\xx;q)$, we are done.
\end{proof}

Our basic tool for proving that $\bm{D_{n,k}}(\xx;q) = \grFrob(S_{n,k};q)$ will be the following lemma,
which is a colored version of \cite[Lem. 3.6]{HRS}.

\begin{lemma}
\label{e-perp-lemma} 
Let $\bm{F(\xx)}, \bm{G(\xx)} \in \Lambda^r(\xx)$ have equal constant terms.  Then 
$\bm{F(\xx)} = \bm{G(\xx)}$ if and only if
$e_j(\xx^{(i^*)})^{\perp} \bm{F(\xx)} = e_j(\xx^{(i^*)})^{\perp} \bm{G(\xx)}$ for all $j \geq 1$ and $1 \leq i \leq r$.
\end{lemma}

\begin{proof}
The forward direction is obvious.  For the reverse direction, let $\blambda$ be any $r$-partition, let 
$j \geq 1$, and let $1 \leq i \leq r$.  We have
\begin{align}
\langle \bm{F(\xx)}, e_j(\xx^{(i^*)}) \bm{e_{\blambda}}(\xx) \rangle &= 
\langle e_j(\xx^{(i^*)})^{\perp} \bm{F(\xx)}, \bm{e_{\blambda}(\xx)} \rangle \\
&= \langle e_j(\xx^{(i^*)})^{\perp} \bm{G(\xx)}, \bm{e_{\blambda}(\xx)} \rangle \\
&= \langle \bm{G(\xx)}, e_j(\xx^{(i^*)}) \bm{e_{\blambda}}(\xx) \rangle.
\end{align}
Since $\langle \bm{F(\xx)}, \bm{e_{\bm{\varnothing}}}(\xx) \rangle = 
\langle \bm{G(\xx)}, \bm{e_{\bm{\varnothing}}}(\xx) \rangle$
by assumption (where $\bm{\varnothing} = (\varnothing, \dots, \varnothing)$ is the empty $r$-partition),
this chain of equalities implies that
$\langle \bm{F(\xx)}, \bm{e_{\blambda}}(\xx) \rangle = 
\langle \bm{G(\xx)}, \bm{e_{\blambda}(\xx)} \rangle$ for any $r$-partition
$\blambda$.  We conclude that $\bm{F(\xx)} = \bm{G(\xx)}$.
\end{proof}

We will show that $\bm{D_{n,k}}(\xx;q)$ and $\grFrob(S_{n,k};q)$ satisfy the conditions of 
Lemma~\ref{e-perp-lemma} by showing that their images under $e_j(\xx^{(i^*)})^{\perp}$
satisfy the same recursion.
The $\coinv$ formulation of $\bm{D_{n,k}}(\xx;q)$ is best suited to calculating 
$e_j(\xx^{(i^*)})^{\perp}$.  The following lemma is a colored version of \cite[Lem. 3.7]{HRS}.

\begin{lemma}
\label{d-under-e-perp}
Let $r \geq 1$ and let $k \leq n$ be positive integers.  Let $1 \leq i \leq r$ and let $j \geq 1$.  
We have 
\begin{equation}
e_j(\xx^{(i^*)})^{\perp} \bm{D_{n,k}}(\xx;q) = q^{j \cdot (r-i) + r \cdot {j \choose 2}} {k \brack j}_{q^r}
\sum_{m = \max(1,k-j)}^{\min(k,n-j)} q^{r \cdot (k-m) \cdot (n-j-m)} {j \brack k-m}_{q^r} \bm{D_{n-j,m}}(\xx;q).
\end{equation}
\end{lemma}

\begin{proof}
Applying $\omega$ to both sides of the purported identity, it suffices to prove 
\begin{equation}
\label{h-equation}
h_j(\xx^{(i^*)})^{\perp} \bm{I_{n,k}}(\xx;q) = 
q^{j \cdot (r-i) + r \cdot {j \choose 2}} {k \brack j}_{q^r}
\sum_{m = \max(1,k-j)}^{\min(k,n-j)} q^{r \cdot (k-m) \cdot (n-j-m)} {j \brack k-m}_{q^r} \bm{I_{n-j,m}}(\xx;q).
\end{equation}

Since the bases $\{ \bm{h_{\blambda}}(\xx) \}$ and $\{ \bm{m_{\blambda^*}}(\xx) \}$ are dual bases
for $\Lambda^r(\xx)$ under the Hall inner product, for any  $\bm{F(\xx)} \in \Lambda^r(\xx)$
and any $r$-composition $\bbeta$, we have
\begin{equation}
\label{coefficient-extraction}
\langle \bm{F(\xx)}, \bm{h_{\bbeta^*}}(\xx) \rangle = \text{coefficient of $\xx^{\bbeta}$ in $\bm{F}(\xx)$}.
\end{equation}
Equation~\ref{coefficient-extraction} is our tool for proving Equation~\ref{h-equation}.

Let $\bbeta = (\beta^{(1)}, \dots, \beta^{(r)})$ be an $r$-composition and consider the inner product
\begin{equation}
\label{h-inner-product}
\langle h_j(\xx^{(i^*)})^{\perp} \bm{I_{n,k}}(\xx;q), \bm{h_{\bbeta^{*}}}(\xx) \rangle = 
\langle \bm{I_{n,k}}(\xx;q), h_j(\xx^{(i^*)}) \bm{h_{\bbeta^*}}(\xx) \rangle.
\end{equation}
We may write $h_j(\xx^{(i^*)}) \bm{h_{\bbeta^*}}(\xx) = \bm{h_{\bm{\widehat{\beta}}^*}}(\xx)$, where
 
\begin{itemize}
\item
$\bm{\widehat{\beta}} = (\beta^{(1)}, \dots, \widehat{\beta}^{(i)}, \dots, \beta^{(r)})$ is an $r$-composition which agrees with 
$\bbeta$ in every component except for $i$, and 
\item
$\widehat{\beta}^{(i)} = (\beta^{(i)}_1, \beta^{(i)}_2, \dots, 0 ,\dots, 0, j)$,
where the composition $\widehat{\beta}^{(i)}$ has $N$ parts for some positive
integer $N$ larger than the number of parts in any of $\beta^{(1)}, \dots, \beta^{(r)}$.
\end{itemize}
By Equation~\ref{coefficient-extraction}, we can interpret 
$\langle \bm{I_{n,k}}(\xx), h_j(\xx^{(i^*)}) \bm{h_{\bbeta^*}}(\xx) \rangle = 
\langle \bm{I_{n,k}}(\xx), \bm{h_{\bm{\widehat{\beta}}^*}}(\xx) \rangle$
combinatorially.

For any $r$-composition $\bm{\alpha} = (\alpha^{(1)}, \dots, \alpha^{(r)})$, 
let $\OP_{\bm{\alpha},k}$ be the collection of $r$-colored ordered multiset partitions with $k$ blocks which contain
$\alpha^{(i)}_j$ copies of the letter $j^{i-1}$.  Equation~\ref{coefficient-extraction} implies
\begin{equation}
\label{combinatorial-coefficient-extraction}
\langle \bm{I_{n,k}}(\xx), \bm{h_{\bm{\widehat{\beta}}^*}}(\xx) \rangle = 
\sum_{\mu \in \OP_{\bm{\widehat{\beta}},k}} q^{\coinv(\mu)}.
\end{equation}

Let us analyze the right-hand side of Equation~\ref{combinatorial-coefficient-extraction}.
A typical element $\mu \in \OP_{\bm{\widehat{\beta}},k}$ contains $j$ copies of the {\em big letter} $N^{i-1}$, together
with various other {\em small letters}.
Recall that the statistic $\coinv$ is defined using the order $\prec$, which prioritizes letter value over color.
 Our choice of $N$ guarantees that every small letter is $\prec N^{i-1}$.
We have a map
\begin{equation}
\varphi: \OP_{\bm{\widehat{\beta}},k} \rightarrow \bigcup_{m = \max(1,k-j)}^{\min(k,n-j)} \OP_{\bbeta,m},
\end{equation}
where $\varphi(\mu)$ is the $r$-colored ordered multiset partition obtained by erasing all $j$ 
of the big letters $N^{i-1}$
in $\mu$ (together with any singleton blocks $\{N^{i-1}\}$).  Let us analyze the effect of $\varphi$ on $\coinv$.

Fix $m$ in the range $\max(1,k-j) \leq m \leq \min(k,n-j)$ and let $\mu \in \OP_{\bbeta,m}$.
Then any $\mu' \in \varphi^{-1}(\mu)$ is obtained by adding $j$ copies of the big letter $N^{i-1}$
to $\mu$, precisely $k-m$ of which must be added in singleton blocks.
We calculate $\sum_{\mu' \in \varphi^{-1}(\mu)} q^{\coinv(\mu')}$ in terms of $\coinv(\mu)$ as follows.

Following the notation of the proof of \cite[Lem. 3.7]{HRS},
let us call a big letter $N^{i-1}$ {\em minb} if it is $\prec$-minimal in its block and {\em nminb} if it 
is not $\prec$-minimal in its block.  Similarly, let us call a small letter {\em mins} or {\em nmins} depending
on whether it is minimal in its block.  The contributions to $\sum_{\mu' \in \varphi^{-1}(\mu)} q^{\coinv(\mu')}$
coming from big letters are as follows.
\begin{itemize}
\item  The $j$ big letters $N^{i-1}$ give a complementary color contribution of $j \cdot (r-i)$ to $\coinv$.
\item  Each of the $minb$ letters forms a coinversion pair with every $nmins$ letter.  Since there are $k-m$
$minb$ letters and $n-j-m$ $nmins$ letters, this contributes $r(k-m)(n-j-m)$ to $\coinv$.
\item  Each of the $minb$ letters forms a coinversion pair with every $nminb$ letter (for a total of 
$(k-m)(j-k+m)$ coinversion pairs) as well each $minb$ letter to its left (for a total of ${k-m \choose 2}$ coinversion pairs.
This contributes $r \cdot [ (k-m)(j-k+m) + {k-m \choose 2} ]$ to $\coinv$.
\item  Each $minb$ letter forms a coinversion pair with each $mins$ letter to its left.  If we sum over the ${k \choose k-m}$
ways of interleaving the singleton blocks $\{N^{i-1}\}$ within the blocks of $\mu$, this gives rise to a factor of
${k \brack k-m}_{q^r}$.
\item  Each $nminb$ letter forms a coinversion pair with each $mins$ letter to its left.  If we consider the ${m \choose j-k+m}$
ways to augment the $m$ blocks of $\mu$ with a $nminb$ letter, this gives rise to a factor of 
$q^{{j-k+m \choose 2}} {m \brack j-k+m}_{q^r}$.
\end{itemize}

Applying the identity
\begin{equation}
r \cdot \left[(k-m)(j-k-m) + {k-m \choose 2} + {j-k+m \choose 2} \right] = r \cdot {j \choose 2},
\end{equation}
we see that
\begin{align}
\sum_{\mu' \in \varphi^{-1}(\mu)} q^{\coinv(\mu')} &= 
q^{j \cdot (r-i) + r \cdot {j \choose 2} + r \cdot (k-m)(n-j-m)} {k \brack k-m}_{q^r} {m \brack j-k+m}_{q^r} q^{\coinv(\mu)} \\
&= q^{j \cdot (r-i) + r \cdot {j \choose 2} + r \cdot (k-m)(n-j-m)} {k \brack j}_{q^r} {j \brack k-m}_{q^r} q^{\coinv(\mu)}.
\end{align}
If we sum this expression over all $\mu \in \OP_{\bbeta,m}$, and then sum over $m$, we get
\begin{equation}
\label{big-expression-h}
q^{j \cdot (r-i) + r \cdot {j \choose 2}}  {k \brack j}_{q^r} \sum_{m = \max(1,k-j)}^{\min(k,n-j)} 
q^{r \cdot (k-m)(n-j-m)} {j \brack k-m}_{q^r} \sum_{\mu \in \OP_{\bbeta,m}} q^{\coinv(\mu)}.
\end{equation}
However, thanks to Equation~\ref{coefficient-extraction} and the definition of the $\bm{I}$-functions,
the expression (\ref{big-expression-h}) is also equal to 
\begin{equation}
\left\langle q^{j \cdot (r-i) + r \cdot {j \choose 2}} {k \brack j}_{q^r}
\sum_{m = \max(1,k-j)}^{\min(k,n-j)} q^{r \cdot (k-m) \cdot (n-j-m)} {j \brack k-m}_{q^r} \bm{I_{n-j,m}}(\xx;q), 
\bm{h_{\bbeta^*}}(\xx) \right\rangle.
\end{equation}
Since both sides of the equation in the statement of the lemma have the same pairing under $\langle \cdot, \cdot \rangle$
with $\bm{h_{\bbeta^*}}(\xx)$ for any $r$-composition $\bbeta$, we are done.
\end{proof}

\section{Hilbert series and standard monomial basis}
\label{Hilbert}

\subsection{The point sets $Y_{n,k}^r$ and $Z_{n,k}^r$}
In this section we  derive the Hilbert series of $R_{n,k}$ and $S_{n,k}$. 
We  also prove that, as ungraded $G_n$-modules, we have
$R_{n,k} \cong  \CC[\FFF_{n,k}]$ and $S_{n,k} \cong \CC[\OP_{n,k}]$.
 To do this, we will use a general method dating back to Garsia and Procesi \cite{GP}
in the context of the Tanisaki ideal.  
We recall the method, and then apply it to our situation.

For any finite point set $Y \subset \CC^n$, let $\II(Y) \subseteq \CC[\xx_n]$ be the ideal of polynomials which vanish on 
$Y$.  That is, we have
\begin{equation}
\II(Y) := \{ f \in \CC[\xx_n] \,:\, f(\yy) = 0 \text{ for all $\yy \in Y$} \}.
\end{equation}
We can identify the quotient $\CC[\xx_n]/\II(Y)$ with the $\CC$-vector space of functions
$Y \rightarrow \CC$.  In particular
\begin{equation}
\dim (\CC[\xx_n]/\II(Y)) = |Y|.
\end{equation}
If $W \subseteq GL_n(\CC)$ is a finite subgroup and $Y$ is stable under the action of $W$, we  have
\begin{equation}
\CC[\xx_n]/\II(Y) \cong_W \CC[Y]
\end{equation}
as $W$-modules, where we used the fact that the permutation module $Y$ is self-dual.

The ideal $\II(Y)$ is almost never homogeneous.  To get a homogeneous ideal, we proceed as follows.  
If $f \in \CC[\xx_n]$ is any nonzero polynomial of degree $d$, write $f = f_d + f_{d-1} + \cdots + f_0$, where $f_i$ is 
homogeneous of degree $i$.  Define $\tau(f) := f_d$ and define a homogeneous ideal $\TT(Y) \subseteq \CC[\xx_n]$ by
\begin{equation}
\TT(Y) := \langle \tau(f) \,:\, f \in \II(Y) - \{0\} \rangle.
\end{equation}
The passage from $\II(Y)$ to $\TT(Y)$ does not affect the $W$-module structure (or vector space dimension)
of the quotient:
\begin{equation}
\CC[\xx_n]/\TT(Y) \cong_W\CC[\xx_n]/\II(Y) \cong_W \CC[Y].
\end{equation}

Our strategy, whose $r = 1$ avatar was accomplished in \cite{HRS}, is as follows.

\begin{enumerate}
\item  Find finite point sets $Y_{n,k}, Z_{n,k}  \subset \CC^n$  
which are stable under the action of $G_n$
such that there are equivariant bijections $Y_{n,k} \cong \FFF_{n,k}$ and $Z_{n,k} \cong \OP_{n,k}$.
\item  Prove that $I_{n,k} \subseteq \TT(Y_{n,k})$ and
$J_{n,k} \subseteq \TT(Z_{n,k})$ by showing that the generators of the ideals $I_{n,k}, J_{n,k}$ arise as top
degree components of polynomials vanishing on $Y_{n,k}, Z_{n,k}$ (respectively).
\item  Use Gr\"obner theory to prove 
\begin{equation*}
\dim(R_{n,k}) = \dim \left( \CC[\xx_n]/I_{n,k} \right) \leq | \FFF_{n,k} | = 
\dim \left( \CC[\xx_n]/\TT(Y_{n,k}) \right)
\end{equation*} 
and
\begin{equation*}
\dim(S_{n,k}) = \dim \left( \CC[\xx_n]/J_{n,k} \right) \leq | \OP_{n,k} | = 
\dim \left( \CC[\xx_n]/\TT(Z_{n,k}) \right).
\end{equation*}
Step 2 then implies $I_{n,k} = \TT(Y_{n,k})$ and $J_{n,k} = \TT(Z_{n,k})$.
\end{enumerate}

To accomplish Step 1 of this program, we introduce the following point sets.  

\begin{defn}
Fix $k$ distinct positive real numbers $0 < \alpha_1 < \cdots < \alpha_k$.  
Let $Y_{n,k} \subset \CC^n$ be the 
set of points $(y_1, \dots, y_n)$ such that
\begin{itemize}
\item  we have $y_i = 0$ or $y_i \in \{ \zeta^c  \alpha_j \,:\, 0 \leq c \leq r-1, 1 \leq j \leq k \}$ for all $i$, and  
\item  we have $\{\alpha_1, \dots, \alpha_k\} \subseteq \{ |y_1|, \dots, |y_n| \}$.
\end{itemize}
Let $Z_{n,k} \subseteq \CC^n$ be the set of points in $Y_{n,k}$ whose coordinates do not vanish:
\begin{equation*}
Z_{n,k} := \{ (y_1, \dots, y_n) \in Y_{n,k} \,:\, y_i \neq 0 \text{ for all $i$.} \}.
\end{equation*}
\end{defn}

There is a bijection $\varphi: \FFF_{n,k} \rightarrow Y_{n,k}$ given as follows.  Let 
$\sigma = (Z \mid B_1 \mid \cdots \mid B_k) \in \FFF_{n,k}$ be an $G_n$-face of dimension $k$, whose 
zero block $Z$ may be empty.  The point $\varphi(\sigma) = (y_1, \dots, y_n)$ has coordinates given by
\begin{equation}
y_i = 
\begin{cases}
0 & \text{if $i \in Z$,} \\
\zeta^c  \alpha_j & \text{if $i\in B_j$ and $i$ has color $c$.}
\end{cases}
\end{equation}
For example if $r = 3$ then
\begin{equation*}
\varphi:
( 25 \mid 3^0  \mid  1^0 4^2 6^2) \mapsto 
(\zeta^0  \alpha_2, 0, \zeta^0  \alpha_1, \zeta^2  \alpha_2, 0, \zeta^2  \alpha_2).
\end{equation*}
The set $Y_{n,k}$ is closed under the action of $G_n$ and the map $\varphi$ commutes
with the action of $G_n$.  It follows that $Y_{n,k} \cong \FFF_{n,k}$ as 
$G_n$-sets.  Moreover, the bijection $\varphi$ restricts to show that 
$Z_{n,k} \cong \OP_{n,k}$ as $G_n$-sets.
This accomplishes Step 1 of our program.

Step 2 of our program is accomplished by  appropriate modifications of  \cite[Sec. 4]{HRS}.

\begin{lemma}
\label{i-contained-in-t}
We have $I_{n,k} \subseteq \TT(Y_{n,k})$ and $J_{n,k} \subseteq \TT(Z_{n,k})$.
\end{lemma}

\begin{proof}
We will show that every generator of $I_{n,k}$ (resp. $J_{n,k}$) 
is the top degree component of some polynomial in $\II(Y_{n,k})$ (resp. $\II(Z_{n,k})$).

Let $1 \leq  i \leq n$.
It is clear that $x_i (x_i^r - \alpha_1^r) \cdots (x_i^r - \alpha_k^r) \in \II(Y_{n,k})$.  Taking the highest 
component, we have $x_i^{kr+1} \in \TT(Y_{n,k})$.  Similarly, the polynomial
$(x_i^r - \alpha_1^r) \cdots (x_i^r - \alpha_k^r)$ vanishes on $Z_{n,k}$, so that 
$x_i^{kr} \in \TT(Z_{n,k})$.
Lemma~\ref{alternating-sum-lemma} applies to show 
$e_{n-k+1}(\xx_n^r), \dots, e_n(\xx_n^r) \in \TT(Y_{n,k})$ and 
$e_{n-k+1}(\xx_n^r), \dots, e_n(\xx_n^r) \in \TT(Z_{n,k})$.
\end{proof}

\subsection{Skip monomials and initial terms}
Step 3 of our program takes more work.  We begin by isolating certain monomials in the
initial ideals of $I_{n,k}$ and $J_{n,k}$.

\begin{lemma}
\label{skip-leading-terms}
Let $<$ be the lexicographic order on monomials in $\CC[\xx_n]$.  
\begin{itemize}
\item
For any $1 \leq i \leq n$ we have
$x_i^{kr+1} \in \initial_<(I_{n,k})$ and $x_i^{kr} \in \initial_<(J_{n,k})$.
\item
If $S \subseteq [n]$ satisfies $|S| = n-k+1$, we also have
$\xx(S)^r \in \initial_<(I_{n,k})$ and $\xx(S)^r \in \initial_<(J_{n,k})$.
\end{itemize}
\end{lemma}

\begin{proof}
The first claim follows from the fact that $x_i^{kr+1}$ is a generator of $I_{n,k}$ and 
$x_i^{kr}$ is a generator of $J_{n,k}$.  The second claim is a consequence of Lemma~\ref{demazure-initial-term}
and Lemma~\ref{demazures-in-ideal}.
\end{proof}

It will turn out that the monomials given in Lemma~\ref{skip-leading-terms} will suffice to generate
$\initial_<(I_{n,k})$ and $\initial_<(J_{n,k})$.  The next definition gives the family of monomials which
are not divisible by any of the monomials in Lemma~\ref{skip-leading-terms} a name.

\begin{defn}
A monomial $m \in \CC[\xx_n]$ is {\em $(n,k)$-nonskip} if 
\begin{itemize}
\item   $x_i^{kr+1} \nmid m$ for $1 \leq i \leq n$, and
\item  $\xx(S)^r \nmid m$ for all $S \subseteq [n]$ with $|S| = n-k+1$.
\end{itemize}
Let $\MMM_{n,k}$ denote the collection of all $(n,k,r)$-nonskip monomials in $\CC[\xx_n]$.

An $(n,k)$-nonskip monomial $m \in \MMM_{n,k}$ is called {\em strongly $(n,k)$-nonskip} if we have
$x_i^{kr} \nmid m$ for all $1 \leq i \leq n$.  Let $\NNN_{n,k}$ denote the collection of strongly 
$(n,k)$-nonskip monomials.  
\end{defn}

We will describe a bijection $\Psi: \FFF_{n,k} \rightarrow \MMM_{n,k}$ which restricts to a bijection
$\OP_{n,k} \rightarrow \NNN_{n,k}$.
The bijection $\Psi$ will be constructed recursively, so that $\Psi(\sigma)$ will be determined by $\Psi(\overline{\sigma})$,
where $\overline{\sigma}$ is the $G_{n-1}$-face obtained from $\sigma$ by deleting the largest letter $n$.
The recursive procedure which gives the derivation $\Psi(\overline{\sigma}) \mapsto \Psi(\sigma)$ will rely
on the following lemmata involving skip monomials.  The first of these is an extension of 
\cite[Lem. 4.5]{HRS}.

\begin{lemma}
\label{skip-monomial-union}
Let $m \in \CC[\xx_n]$ be a monomial and let $S, T \subseteq [n]$ be subsets.  If $\xx(S)^r \mid m$
and $\xx(T)^r \mid m$, then $\xx(S \cup T)^r \mid m$.  
\end{lemma}

\begin{proof}
Given $i \in S$, it follows from the definition of skip monomials that the exponent of $x_i$ in $\xx(S \cup T)^r$
is $\leq$ the exponent of $x_i$ in $\xx(S)^r$.  A similar observation holds for $i \in T$.  The claimed divisibility follows.
\end{proof}

The following result is an immediate consequence of Lemma~\ref{skip-monomial-union}; it extends
\cite[Lem. 4.6]{HRS}.

\begin{lemma}
\label{skip-monomial-unique}
Let $m \in \CC[\xx_n]$ be a monomial and let $\ell$ be the largest integer such that there exists a subset
$S \subseteq [n]$ with $|S| = \ell$ and $\xx(S)^r \mid m$.  Then there exists {\em a unique} subset $S \subseteq [n]$
with $|S| = \ell$ and $\xx(S)^r \mid m$.
\end{lemma}

\begin{proof}
If there were two such sets $S, S'$ then by Lemma~\ref{skip-monomial-union} we would have
$\xx(S \cup S')^r \mid m$, contradicting the definition of $\ell$.
\end{proof}

Given any subset $S \subseteq [n]$, let $\mm(S) := \prod_{i \in S} x_i$ be the 
corresponding squarefree monomial.  
For example, we have $\mm(245) = x_2 x_4 x_5$.
We have the following lemma involving the $r^{th}$ power 
$\mm(S)^r$ of $\mm(S)$.  This is the extension of \cite[Lem. 4.7]{HRS}.

\begin{lemma}
\label{skip-monomial-multiply}
Let $m \in \MMM_{n,k}$ be an $(n,k)$-nonskip monomial.  There exists a unique set $S \subseteq [n]$ with
$|S| = n-k$ such that
\begin{enumerate}
\item  $\xx(S)^r \mid ( \mm(S)^r \cdot m)$, and
\item  $\xx(U)^r \nmid ( \mm(S)^r \cdot m)$ for all $U \subseteq [n]$ with $|U| = n-k+1$.
\end{enumerate} 
\end{lemma}

\begin{proof}
We begin with uniqueness.  Suppose $S = \{s_1 < \cdots < s_{n-k} \}$ and $T = \{t_1 < \cdots < t_{n-k} \}$ were
two such sets.  Let $\ell$ be such that $s_1 = t_1, \dots, s_{\ell-1} = t_{\ell-1}$, and $s_{\ell} \neq t_{\ell}$;
without loss of generality we have $s_{\ell} < t_{\ell}$.
Define a new set $U$ by $U := \{s_1 < \cdots < s_{\ell} < t_{\ell} < t_{\ell + 1} < \cdots < t_{n-k} \}$, so that 
$|U| = n-k+1$.
Since $\xx(S)^r \mid (\mm(S)^r \cdot m)$ and $\xx(T)^r \mid (\mm(T)^r \cdot m)$, we have
$\xx(U)^r \mid (\mm(S)^r \cdot m)$, which is a contradiction.

To prove existence, consider the following collection $\CCC$ of subsets of $[n]$:
\begin{equation}
\CCC := \{ S \subseteq [n] \,:\, |S| = n-k \text{ and } \xx(S)^r \mid (\mm(S)^r \cdot m) \}.
\end{equation}
The collection $\CCC$ is nonempty; indeed, we have $\{1, 2, \dots, n-k\} \in \CCC$.  Let $S_0 \in \CCC$ 
be the lexicographically {\em final} set in $\CCC$;  we argue that $\mm(S_0)^r \cdot m$ satisfies 
Condition 2 of the statement of the lemma, thus finishing the proof.

Let $U \subseteq [n]$ have size $|U| = n-k+1$ and suppose $\xx(U)^r \mid (\mm(S_0)^r \cdot m)$.  
If there
were an element $u \in U$ with $u < \min(S_0)$, then we would have $\xx(S_0 \cup \{u\})^r \mid m$,
which contradicts the assumption $m \in \MMM_{n,k}$.  Since $|U| > |S_0|$, there exists an element 
$u_0 \in U - S_0$ with $u_0 > \min(S_0)$.  Write the union $S_0 \cup \{u_0\}$ as
\begin{equation}
S_0 \cup \{u_0\} = \{s_1 < \cdots < s_j < u_0 < s_{j+1} < \cdots < s_{n-k} \},
\end{equation}
where $j \geq 1$.  Define a new set $S_0'$ by
\begin{equation}
S_0' := \{s_1 < \cdots < s_{j-1} < u_0 < s_{j+1} < \cdots < s_{n-k} \}.
\end{equation}
Then $S_0'$ comes after $S_0$ in lexicographic order but we have $S_0' \in \CCC$, contradicting our choice
of $S_0$.
\end{proof}

To see how Lemma~\ref{skip-monomial-multiply} works, consider the case
$(n,k,r) = (5,2,3)$ and $m = x_1^2 x_2^6 x_3^3 x_4^3 x_5^6 \in \MMM_{5,2}$.
The collection $\CCC$ of sets
\begin{equation*}
\CCC = \{ S \subseteq [5] \,:\, |S| = 3 \text{ and } \xx(S)^3 \mid (\mm(S)^3 \cdot m) \}
\end{equation*}
is given by
\begin{equation*}
\CCC = \{123,  124, 125, 234, 235 \}.
\end{equation*}
However, we have
\begin{align*}
&\xx(1234)^3 \mid (\mm(123)^3 \cdot m),  
&\xx(1234)^3 \mid (\mm(124)^3 \cdot m), \\
&\xx(1235)^3 \mid (\mm(125)^3 \cdot m), 
&\xx(2345)^3 \mid (\mm(234)^3 \cdot m).
\end{align*}
On the other hand, if $S \subseteq [5]$ and $|S| = 4$, then $\xx(S)^3 \nmid (\mm(235)^3 \cdot m)$.
Observe that $235$ is the lexicographically final set in $\CCC$.

\subsection{The bijection $\Psi$}
We describe a bijection $\Psi: \FFF_{n,k} \rightarrow \MMM_{n,k}$ which restricts
to a bijection $\OP_{n,k} \rightarrow \NNN_{n,k}$ with the property that 
$\coinv(\sigma) = \deg(\Psi(\sigma))$ for any $G_n$-face $\sigma \in \FFF_{n,k}$.  
The construction of $\Psi$ will be recursive in the parameter $n$.

If $n = 1$ and $k = 1$,
the relation $\coinv(\sigma) = \deg(\Psi(\sigma))$ determines the bijection $\Psi$ uniquely.  Explicitly,
the map $\Psi: \FFF_{1,1} \rightarrow \MMM_{1,1}$ is defined by 
\begin{equation}
\Psi:  (1^c) \mapsto x_1^{r-c-1}, 
\end{equation}
for any color $0 \leq c \leq r-1$.

If $n = 1$ and $k = 0$ then $\FFF_{1,0}$ consists of the sole face $(1)$.  
On the other hand, the collection $\MMM_{1,0}$ of nonskip monomials 
consists of the sole monomial $1$.
We are  forced to define 
\begin{equation}
\Psi: (1) \mapsto 1.
\end{equation}

The combinatorial recursion on which $\Psi$ is based is as follows.
Let $\sigma = (B_1 \mid \cdots \mid B_{\ell}) \in \FFF_{n,k}$ 
be an $G_n$-face of dimension $k$, so that $\ell = k+1$ or $\ell = k$ according to whether $\sigma$
has a zero block.
Suppose we wish to build a larger face
by inserting $n+1$ into $\sigma$.  There are three ways in which this can be done.
\begin{enumerate}
\item  We could perform a {\em star insertion} by inserting $n+1$ into one of the
nonzero blocks $B_{\ell - j}$ of $\sigma$ for $1 \leq j \leq k$
also assigning a color $c$ to $n+1$.  The resulting $G_n$-face would be
$(B_1 \mid \cdots \mid B_{\ell - j} \cup \{(n+1)^c\} \mid \cdots \mid B_{\ell})$.  This leaves
the dimension $k$ unchanged and increases $\coinv$
by $r \cdot (k - j) + (r - c - 1)$.  

For example, if $r = 2$ and  $\sigma = (3 \mid 2^1 4^0 \mid 1^1) \in \FFF_{4,2}$, the possible star insertions of 
$5$ and their effects on $\coinv$ are
\begin{center}
$\begin{array}{cccc}
(3 \mid 2^1 4^0 5^1 \mid 1^1) & (3 \mid 2^1 4^0 5^0 \mid 1^1 ) & (3 \mid 2^1 4^0  \mid 1^1 5^1) & 
(3 \mid 2^1 4^0 \mid 1^1 5^0) \\
\coinv + 0 & \coinv + 1 & \coinv + 2 & \coinv + 3.
\end{array}$
\end{center}
\item  We could perform a {\em zero insertion} by inserting $n+1$ into the zero block of $\sigma$ (or by creating
a new zero block whose sole element is $n+1$).  This leaves the dimension $k$ unchanged and increases 
$\coinv$ by $kr$.

For example, if $r = 2$ and  $\sigma = (3 \mid 2^1 4^0 \mid 1^1) \in \FFF_{4,2}$, the zero insertion of $5$ would
yield $(35 \mid 2^1 4^0 \mid 1^1)$, adding $4$ to $\coinv$.
\item  We could perform a {\em bar insertion} by inserting $n+1$  into a new singleton nonzero block of $\sigma$ just
after the block $B_{\ell - j}$ for some $0 \leq j \leq k$,
also assigning a color $c$ to $n+1$.   The resulting $G_n$-face would be 
$(B_1 \mid \cdots \mid B_{\ell - j} 
\mid (n+1)^c  \mid B_{\ell - j + 1} \mid \cdots \mid B_{\ell})$.  This increases the dimension $k$ by one
and increases $\coinv$ by $r \cdot (n-k) + r \cdot (k-j) + (r-c-1)$.

For example, if $r = 2$ and  $\sigma = (3 \mid 2^1 4^0 \mid 1^1) \in \FFF_{4,2}$, the possible bar insertions of 
$5$ and their effects on $\coinv$ are
\begin{center}
$\begin{array}{ccc}
(3 \mid 5^1 \mid 2^1 4^0 \mid 1^1) & 
(3 \mid 5^0 \mid 2^1 4^0 \mid 1^1) &
(3 \mid 2^1 4^0 \mid 5^1 \mid 1^1)   \\
\coinv + 4 & \coinv + 5 & \coinv + 6  \\ \\
(3 \mid 2^1 4^0 \mid 5^0 \mid 1^1) & 
(3 \mid 2^1 4^0 \mid 1^1 \mid 5^1) & 
(3 \mid 2^1 4^0 \mid 1^1 \mid 5^0)  \\
\coinv + 7 & \coinv + 8 & \coinv + 9.
\end{array}$
\end{center}
\end{enumerate}
The names of these three kinds of insertions come from our combinatorial models for $G_n$-faces; a star insertion adds
a star to the star model of  
$\sigma$, a zero insertion adds an element to the zero block of $\sigma$, and a bar insertion adds
a bar to the bar model of $\sigma$.

Let $\sigma = (B_1 \mid \cdots \mid B_{\ell}) \in \FFF_{n,k}$ 
be an $G_n$-face of dimension $k$ and let $\overline{\sigma}$ be the $G_{n-1}$-face
obtained by deleting $n$ from $\sigma$.  Then $\overline{\sigma} \in \FFF_{n-1,k}$ if $\sigma$ arises from 
$\overline{\sigma}$ by a star or zero insertion and $\overline{\sigma} \in \FFF_{n-1,k-1}$ if $\sigma$
arises from $\overline{\sigma}$ from a bar insertion.
Assume inductively that the monomial $\Psi(\overline{\sigma})$ has been defined, and that 
this monomial lies in $\MMM_{n-1,k}$ or $\MMM_{n-1,k-1}$ according to whether $\overline{\sigma}$ lies
in $\FFF_{n-1,k}$ or $\FFF_{n-1,k-1}$.
We define $\Psi(\sigma)$ by the
rule
\begin{equation}
\Psi(\sigma) := \begin{cases}
\Psi(\overline{\sigma}) \cdot x_n^{r \cdot (k-j-1) + (r-c-1)} & 
\text{if $n^c \in B_{\ell - j}$ and $B_{\ell - j}$ is a nonzero nonsingleton,} \\
\Psi(\overline{\sigma}) \cdot x_n^{kr} & \text{if $n$ lies in the zero block of $\sigma$,} \\
\Psi(\overline{\sigma}) \cdot \mm(S)^r \cdot x_n^{r \cdot (k-j-1) + (r-c-1)} &  
\text{if $B_{\ell - j} = \{n^c\}$ is a nonzero singleton,}
\end{cases}
\end{equation}
where in the third branch $S \subseteq [n-1]$ is the unique subset of size $|S| = n-k$ guaranteed by 
Lemma~\ref{skip-monomial-multiply} applied to $m = \Psi(\overline{\sigma}) \in \MMM_{n-1,k-1}$.

\begin{example}
Let $(n,k,r) = (8,3,3)$
and consider the face $\sigma = (2 5 \mid 1^0 7^0 8^1 \mid 6^1 \mid 3^2 4^2) \in \FFF_{8,3}$.
In order to calculate $\Psi(\sigma) \in \MMM_{8,3}$, we refer to the following table.
Here `type' refers to the type of insertion (star, zero, or bar) of $n$ at each stage.

\begin{center}
\begin{tabular}{l | l | l | l | l | l}
$\sigma$ & $n$ & $k$ & type & $S$ & $\Psi(\sigma)$ \\ \hline
$(1^0)$ & $1$ & $1$ & & & $x_1^2$ \\
$(2 \mid 1^0)$ & $2$ & $1$ & zero & & $x_1^2 x_2^3$ \\
$(2 \mid 1^0 \mid 3^2)$ & $3$ & $2$ & bar & $2$ & $x_1^2 x_2^3 \cdot \mm(2)^3 \cdot x_3^3 = x_1^2 x_2^6 x_3^3$ \\
$(2 \mid 1^0 \mid 3^2 4^2)$ & $4$ & $2$ & star & & $x_1^2 x_2^6 x_3^3 x_4^3$ \\
$(25 \mid 1^0 \mid 3^2 4^2)$ & $5$ & $2$ & zero & & $x_1^2 x_2^6 x_3^3 x_4^3 x_5^6$ \\
$(25 \mid 1^0 \mid 6^1 \mid 3^2 4^2)$ & $6$ & $3$ & bar & $235$ & 
$x_1^2 x_2^6 x_3^3 x_4^3 x_5^6 \cdot \mm(235)^3 \cdot x_6^4 = x_1^2 x_2^9 x_3^6 x_4^3 x_5^9 x_6^4$ \\
$(25 \mid 1^0 7^0 \mid 6^1 \mid 3^2 4^2)$ & $7$ & $3$ & star &  & 
$x_1^2 x_2^9 x_3^6 x_4^3 x_5^9 x_6^4 x_7^2$ \\
$(25 \mid 1^0 7^0 8^1 \mid 6^1 \mid 3^2 4^2)$ & $8$ & $3$ & star &  & 
$x_1^2 x_2^9 x_3^6 x_4^3 x_5^9 x_6^4 x_7^2 x_8^1$ \\
\end{tabular}
\end{center}

We conclude that 
\begin{equation*} \Psi(\sigma) = 
\Psi(2 5 \mid 1^0 7^0 8^1 \mid 6^1 \mid 3^2 4^2) = x_1^2 x_2^9 x_3^6 x_4^3 x_5^9 x_6^4 x_7^2 x_8^1 
\in \MMM_{8,3}.  
\end{equation*}
Observe that the zero block of $\sigma$ is $\{2,5\}$, and that $x_2$ and $x_5$ are the variables in $\Psi(\sigma)$
with exponent $k  r = 3 \cdot 3 = 9$.  
\end{example}

The next result is the extension of \cite[Thm. 4.9]{HRS} to $r \geq 2$.
The proof has the same basic structure, but one must account for the presence of zero blocks.

\begin{proposition}
\label{psi-is-bijection}
The map $\Psi: \FFF_{n,k} \rightarrow \MMM_{n,k}$ is a bijection which restricts to a bijection
$\OP_{n,k} \rightarrow \NNN_{n,k}$.  Moreover, for any $\sigma \in \FFF_{n,k}$ we have
\begin{equation}
\coinv(\sigma) = \deg(\Psi(\sigma)).
\end{equation}
Finally, if $\sigma \in \FFF_{n,k}$ has a zero block $Z$, then 
\begin{equation}
Z = \{1 \leq i \leq n \,:\, \text{the exponent of $x_i$ in $\Psi(\sigma)$ is $kr$} \}.
\end{equation}
\end{proposition}

\begin{proof}
We need to show that $\Psi$ is a well-defined function $\FFF_{n,k} \rightarrow \MMM_{n,k}$.  To do this, we induct
on $n$ (with the base case $n = 1$ being clear).  Let $\sigma = (B_1 \mid \cdots \mid B_{\ell}) \in \FFF_{n,k}$ 
and let $\overline{\sigma}$ be the 
$G_{n-1}$-face obtained by removing $n$ from $\sigma$.  Then $\overline{\sigma} \in \FFF_{n-1,k}$
(if the insertion type of $n$ was star or zero) or $\overline{\sigma} \in \FFF_{n-1,k-1}$ (if the insertion type
of $n$ was bar).  We inductively assume that 
$\Psi(\overline{\sigma}) \in \MMM_{n-1,k}$ or $\Psi(\overline{\sigma}) \in \MMM_{n-1,k-1}$ accordingly.

Suppose first that the insertion type of $n$ was star or zero, so that $\Psi(\overline{\sigma}) \in \MMM_{n-1,k}$.
Then we have
\begin{equation}
\Psi(\sigma) = \begin{cases}
\Psi(\overline{\sigma}) \cdot x_n^{r \cdot (k-j-1) + (r-c-1)} & 
\text{if $n^c \in B_{\ell - j}$ and $B_{\ell - j}$ is a nonzero nonsingleton,} \\
\Psi(\overline{\sigma}) \cdot x_n^{kr} & \text{if $n$ lies in the zero block of $\sigma$.} 
\end{cases}
\end{equation}
By induction and the inequalities $0 \leq j \leq k-1$ and $0 \leq c \leq r-1$, 
we know that none of the variable powers $x_1^{kr+1},  \dots, x_n^{kr+1}$ divide $\Psi(\sigma)$.
Let $S \subseteq [n]$ be a subset of size $|S| = n-k+1$.  Since $\Psi(\overline{\sigma}) \in \MMM_{n-1,k}^r$,
we know that $\xx(S - \{\max(S)\})^r \nmid \Psi(\overline{\sigma})$.  This implies that
$\xx(S)^r \nmid \Psi(\sigma)$.  We conclude that $\Psi(\sigma) \in \MMM_{n,k}$.

Now suppose that the insertion type of $n$ was bar, so that $\Psi(\overline{\sigma}) \in \MMM_{n-1,k-1}$.
We have
\begin{equation}
\Psi(\sigma) = \Psi(\overline{\sigma}) \cdot \mm(S)^r \cdot x_n^{r \cdot (k-j-1) + (r - c- 1)},
\end{equation}
where $B_{\ell - j} = \{n^c\}$ and $S \subseteq [n-1]$ is the unique subset of size $|S| = n-k$ guaranteed
by Lemma~\ref{skip-monomial-multiply} applied to the monomial $m = \Psi(\overline{\sigma})$.
Since none of the variable powers $x_1^{(k-1)\cdot r + 1},  \dots, x_{n-1}^{(k-1) \cdot r + 1}$
divide $\Psi(\overline{\sigma})$, we conclude that none of the variable powers 
$x_1^{kr+1}, \dots, x_n^{kr+1}$ divide $\Psi(\sigma)$.  Let $T \subseteq [n]$ satisfy $|T| = n-k+1$.  
If $n \notin T$, Lemma~\ref{skip-monomial-multiply} and induction guarantee that 
$\xx(T)^r \nmid \Psi(\sigma)$.  If $n \in T$, then the power of $x_n$ in the monomial $\xx(T)^r$ is $kr$, so that 
$\xx(T)^r \nmid \Psi(\sigma)$.  We conclude that $\Psi(\sigma) \in \MMM_{n,k}$.  This finishes the proof
that $\Psi: \FFF_{n,k} \rightarrow \MMM_{n,k}$ is well-defined.

The relationship $\coinv(\sigma) = \deg(\Psi(\sigma))$ is clear from the inductive definition of $\Psi$ and 
the previously described effect of insertion on the $\coinv$ statistic.

Let $\sigma \in \FFF_{n,k}$ be an $G_n$-face with zero block $Z$ (where $Z$ could be empty).  We aim to show that
$Z = \{ 1 \leq i \leq n \,:\, \text{the exponent of $x_i$ in $\Psi(\sigma)$ is $kr$} \}$.  To do this, we proceed by induction on 
$n$ (the case $n = 1$ being clear).  As before, let $\overline{\sigma}$ be the face obtained by erasing $n$ from $\sigma$
and let $\overline{Z}$ be the zero block of $\overline{\sigma}$.  We inductively assume that 
\begin{equation}
\overline{Z} = \begin{cases}
\{1 \leq i \leq n-1 \,:\, \text{the exponent of $x_i$ in $\Psi(\overline{\sigma})$ is $kr$} \} & 
\text{if $\overline{\sigma} \in \FFF_{n-1, k}$},  \\
\{1 \leq i \leq n-1 \,:\, \text{the exponent of $x_i$ in $\Psi(\overline{\sigma})$ is $(k-1) \cdot r$} \} & 
\text{if $\overline{\sigma} \in \FFF_{n-1, k-1}$}.
\end{cases}
\end{equation}

Suppose first that $\sigma$ was obtained from $\overline{\sigma}$ by a star insertion, so that 
$\overline{\sigma} \in \FFF_{n-1,k}$ and  $Z = \overline{Z}$.  Since the exponent of $x_n$ in $\Psi(\sigma)$ is 
$< kr$, the desired equality of sets holds in this case.

Next, suppose that $\sigma$ was obtained from $\overline{\sigma}$ by a zero insertion, so that 
$\overline{\sigma} \in \FFF_{n-1,k}$ and $Z = \overline{Z} \cup \{n\}$.  Since the exponent of $x_n$
in $\Psi(\sigma)$ is $kr$, the desired equality of sets holds in this case.

Finally, suppose that $\sigma$ was obtained from $\overline{\sigma}$ by a bar insertion, so that 
$\overline{\sigma} \in \FFF_{n-1,k-1}$ and $Z = \overline{Z}$.  Since the exponent of $x_n$ in $\Psi(\sigma)$ is
$< kr$, by induction we need only argue that $Z \subseteq S$, where $S \subseteq [n-1]$ is the 
unique subset of size $|S| = n-k$ guaranteed by Lemma~\ref{skip-monomial-multiply} applied to 
the monomial $m = \Psi(\overline{\sigma})$.  

If the containment $Z \subseteq S$ failed to hold, let
$z = Z - S$ be arbitrary.  By induction, the exponent of $x_z$ in $\Psi(\overline{\sigma})$ is $(k-1) \cdot r$.
Also, we have the divisibility $\xx(S)^r \mid \Psi(\overline{\sigma}) \cdot \mm(S)^r$.
If since $z \leq n-1$, we have the divisibility $\xx(S \cup \{z\})^r \mid \xx(S)^r \cdot x_z^{(k-1) \cdot r}$, so that
$\xx(S \cup \{z\})^r \mid \Psi(\overline{\sigma}) \cdot \mm(S)^r$, which contradicts Lemma~\ref{skip-monomial-multiply}. 
We conclude that $Z \subseteq S$.  This proves the last sentence of the proposition.

We now turn our attention to proving that $\Psi: \FFF_{n,k} \rightarrow \MMM_{n,k}$ is a bijection.
In order to prove that $\Psi$ is a bijection, we will construct its inverse $\Phi: \MMM_{n,k} \rightarrow \FFF_{n,k}$.
The map $\Phi$ will be defined by reversing the recursion used to define $\Psi$.

When $(n,k) = (1,0)$, there is only one choice for $\Phi$; we must define $\Phi: \MMM_{1,0} \rightarrow \FFF_{1,0}$
by
\begin{equation}
\Phi:  1 \mapsto (1).
\end{equation}
When $(n,k) = (1,1)$, since $\Phi$ is supposed to invert
the function $\Psi$, we are forced to define $\Phi: \MMM_{1,1} \rightarrow \FFF_{1,1}$ by
\begin{equation}
\Phi: x_1^c \mapsto (1^{r-c-1}),
\end{equation}
for $0 \leq c \leq r-1$.

In general, fix $k \leq n$ and assume inductively that the functions
\begin{equation*}
\begin{cases}
\Phi: \MMM_{n-1,k} \rightarrow \FFF_{n-1,k}, \\  \Phi: \MMM_{n-1,k-1} \rightarrow \FFF_{n-1,k-1}
\end{cases}
\end{equation*} have already
been defined.  We aim to define the function $\Phi: \MMM_{n,k} \rightarrow \FFF_{n,k}$.  To this end, let 
$m = x_1^{a_1} \cdots x_{n-1}^{a_{n-1}} x_n^{a_n} \in \MMM_{n,k}$ be a monomial.  Define a new monomial
$m' := x_1^{a_1} \cdots x_{n-1}^{a_{n-1}}$ by setting $x_n = 1$ in $m$.
Either $m' \in \MMM_{n-1,k}$ or $m' \notin \MMM_{n-1,k}$.

If $m' \in \MMM_{n-1,k}$, then $\Phi(m') = (B_1 \mid \cdots \mid B_{\ell}) \in \FFF_{n-1,k}^r$ is a
previously defined $G_{n-1}$-face.  Our definition of $\Phi(m)$ depends on the exponent $a_n$ of $x_n$ in $m$.
\begin{itemize}
\item  If $m' \in \MMM_{n-1,k}$ and $a_n < kr$, write $a_n = j \cdot r + (r-c-1)$ for a nonnegative integer $j$ and 
$0 \leq c \leq r-1$.  Let $\Phi(m)$ be obtained from $\Phi(m')$ by star inserting $n^c$ into the $j^{th}$ nonzero block
of $\Psi(m)$ from the left.
\item  If $m' \in \MMM_{n-1,k}$ and $a_n = kr$, let $\Phi(m)$ be obtained from $\Phi(m')$ by adding $n$ to 
the zero block of $\Phi(m')$ (creating a zero block if necessary).
\end{itemize}

If $m' \notin \MMM_{n-1,k}$, there exists a subset $S \subseteq [n-1]$ such that $|S| = n-k$ and
$\xx(S)^r \mid m'$.  Lemma~\ref{skip-monomial-unique} guarantees that the set $S$ is unique.

{\bf Claim:}  We have $\frac{m'}{\mm(S)^r} \in \MMM_{n-1,k-1}$.

Since $m \in \MMM_{n,k}$, we know that $\xx(T)^r \nmid \frac{m'}{\mm(S)^r}$ for all $T \subseteq [n-1]$
with $|T| = n-k+1$.  Let $1 \leq j \leq n-1$.  We need to show $x_j^{(k-1) \cdot r + 1} \nmid \frac{m'}{\mm(S)^r}$.
If $j \in S$ this is immediate from the fact that $x_j^{kr + 1} \nmid m'$.  If $j \notin S$ and 
$x_j^{(k-1) \cdot r + 1} \mid \frac{m'}{\mm(S)^r}$, then $x_j^{(k-1) \cdot r + 1} \mid m'$ and
$\xx(S \cup \{j\})^r \mid m'$, a contradiction to the assumption $m = m' \cdot x_n^{a_n} \in \MMM_{n,k}$.  This finishes the 
proof of the Claim.

By the Claim, we recursively have an $G_{n-1}$-face $\Phi \left( \frac{m'}{\mm(S)} \right) \in \FFF_{n-1,k-1}$.
Moreover, we have $a_n < kr$ (because otherwise $\xx(S \cup \{n\})^r \mid m$, contradicting $m \in \MMM_{n,k}$).
Write $a_n = j \cdot r + (r-c-1)$ for some nonnegative integer $j$ and $0 \leq c \leq r-1$.  
Form $\Phi(m)$ from $\Phi(m')$ by bar inserting the singleton block $\{n^c\}$ to the left of the $j^{th}$
nonzero block of $\Phi(m')$ from the left.

For an example of the map $\Phi$, let $(n,k,r) = (8,3,3)$ and
let $m = x_1^2 x_2^9 x_3^6 x_4^3 x_5^9 x_6^4 x_7^2 x_8^1 \in \MMM_{8,3}$.  The following table 
computes $\Phi(m) = (25 \mid 1^0 7^0 8^1 \mid 6^1 \mid 3^2 4^2)$.
Throughout this calculation, the nonzero blocks will successively become frozen (i.e., written in bold).

\begin{small}
\begin{center}
\begin{tabular}{l | l | l | l | l | l | l | l}
$m$ & $m'$ & $(n,k)$ & type & $S$ & $\frac{m'}{\mm(S)^r}$ & $(j,c)$ & $\Phi(m)$ \\ \hline
$x_1^2 x_2^9 x_3^6 x_4^3 x_5^9 x_6^4 x_7^2 x_8^1$ & 
$x_1^2 x_2^9 x_3^6 x_4^3 x_5^9 x_6^4 x_7^2$ & $(8,3)$ & star & & & $(0,1)$ & $(8^1 \mid \cdot \mid \cdot)$ \\
$x_1^2 x_2^9 x_3^6 x_4^3 x_5^9 x_6^4 x_7^2$ & 
$x_1^2 x_2^9 x_3^6 x_4^3 x_5^9 x_6^4$ & $(7,3)$ & star & & & $(0,0)$ & $(7^0 8^1 \mid \cdot \mid \cdot)$ \\
$x_1^2 x_2^9 x_3^6 x_4^3 x_5^9 x_6^4$ & 
$x_1^2 x_2^9 x_3^6 x_4^3 x_5^9$ & $(6,3)$ & bar & $235$ & $x_1^2 x_2^6 x_3^3 x_4^3 x_5^6$ 
& $(1,1)$ & $(7^0 8^1 \mid {\bf 6^1} \mid \cdot)$ \\
$x_1^2 x_2^6 x_3^3 x_4^3 x_5^6$ & $x_1^2 x_2^6 x_3^3 x_4^3$ & $(5,2)$ & zero & & & &
 $(5 \mid 7^0 8^1 \mid {\bf 6^1} \mid \cdot)$ \\
 $x_1^2 x_2^6 x_3^3 x_4^3$ & $x_1^2 x_2^6 x_3^3$ & $(4,2)$ & star & & & $(1,2)$ &
 $(5 \mid 7^0 8^1 \mid {\bf 6^1} \mid  4^2 )$ \\
  $x_1^2 x_2^6 x_3^3$ & $x_1^2 x_2^6$ & $(3,2)$ & bar & 2 & $x_1^2 x_2^3$ & $(1,2)$ &
 $(5 \mid 7^0 8^1 \mid {\bf 6^1} \mid  {\bf 3^2 4^2} )$ \\
 $x_1^2 x_2^3$ & $x_1^2$ & $(2,1)$ & zero & & & & $(25 \mid 7^0 8^1 \mid {\bf 6^1} \mid {\bf 3^2 4^2})$ \\
 $x_1^2$ & 1 & $(1,1)$ & bar & $\varnothing$ & 1 & $(0,0)$ &  $(25 \mid {\bf 1^0 7^0 8^1} \mid {\bf 6^1} \mid {\bf 3^2 4^2})$
\end{tabular}
\end{center}
\end{small}

To proceed from one row of the table to the next, we use the following procedure.
\begin{itemize}
\item  Define $m$ to be the monomial $m'$ from the above row (if the insertion type in the 
above row was star or zero) or the monomial
$\frac{m'}{\mm(S)^r}$ from the above row (if the insertion type in the above row was bar).
\item  Define $(n,k)$ in the current row to be $(n-1,k)$ from the above row (if the insertion type in the above
row was star or zero) or $(n-1,k-1)$ from the above row (if the insertion type in the above row was bar).
\item  Using the $(n,k)$ in the current row, define $m'$ from $m$ using the relation $m = m' \cdot x_n^{a_n}$.
\item  If $a_n = kr$, define the insertion type of the current row to be zero, let $\Phi(m)$ be obtained from the above
row by adjoining $n$ to its zero block (creating a new zero block if necessary), and move on to the next row.
\item  If $a_n < kr$, define $(j,c)$ by the relation $a_n = j \cdot r + (r-c-1)$, where $j$ is nonnegative and $0 \leq c \leq r-1$.
\item  If $a_n < kr$ and $m' \in \MMM_{n-1,k}$, define the insertion type of the current row to be star.  Let
$\Phi(m)$ obtained from the above row by inserting $n^c$ into the $j^{th}$ nonzero nonfrozen block from the left, and 
move on to the next row.
\item  If $a_n < kr$ and $m' \notin \MMM_{n-1,k}$, define the insertion type of the current row to be bar.  Let 
$S \subseteq [n-1]$ be the set defined by Lemma~\ref{skip-monomial-unique} as above.  Calculate $\frac{m'}{\mm(S)^r}$.
Let $\Phi(m)$ be obtained from the above row by inserting $n^c$ into the $j^{th}$ nonzero nonfrozen block from
the left and freezing that block.  Move on to the next row.
\end{itemize}

We leave it for the reader to check that the procedure defined above reverses the recursive definition of $\Psi$, so 
that $\Phi$ and $\Psi$ are mutually inverse maps.
The fact that $\Psi$ restricts to give a bijection $\OP_{n,k} \rightarrow \NNN_{n,k}$ follows from the assertion about 
zero blocks.
\end{proof}

We are ready to identify the standard monomial bases of our quotient rings $R_{n,k}$ and $S_{n,k}$.
The proof of the following result is analogous to the proof of \cite[Thm. 4.10]{HRS}.

\begin{theorem}
\label{m-is-basis}
Let $n \geq k$ be positive integers and
endow monomials in $\CC[\xx_n]$ with the lexicographic term order $<$.
\begin{itemize}
\item
The collection $\MMM_{n,k}$ of $(n,k)$-nonskip monomials in $\CC[\xx_n]$
is the standard monomial basis of $R_{n,k}$.
 \item
 The collection $\NNN_{n,k}$ of strongly $(n,k)$-nonskip monomials in $\CC[\xx_n]$ 
is the standard monomial basis of $S_{n,k}$.
 \end{itemize}
\end{theorem}

\begin{proof}
Let us begin with the case of $R_{n,k}$.  Recall the point set $Y_{n,k} \subseteq \CC^n$.  
Let $\BBB_{n,k}$ be the standard monomial basis of the quotient ring $\CC[\xx_n] / \TT(Y_{n,k})$.
Since $\dim(\CC[\xx_n] / \TT(Y_{n,k})) = | Y_{n,k} | = | \FFF_{n,k} |$, we have  
\begin{equation}
|\BBB_{n,k}| = | \FFF_{n,k} |.
\end{equation}

On the other hand, Lemma~\ref{i-contained-in-t} says that $I_{n,k} \subseteq \TT(Y_{n,k})$.  This leads
to the containment of initial ideals
\begin{equation}
\initial_<(I_{n,k}) \subseteq \initial_<(\TT(Y_{n,k})).
\end{equation}
If $\CCC_{n,k}$ is the standard monomial basis for $R_{n,k} = \CC[\xx_n] / I_{n,k}$, 
this implies
\begin{equation}
\BBB_{n,k} \subseteq \CCC_{n,k}.
\end{equation}
However, Lemma~\ref{skip-leading-terms} and the definition of $(n,k)$-nonskip monomials implies 
\begin{equation}
\CCC_{n,k} \subseteq \MMM_{n,k}.
\end{equation}
Proposition~\ref{psi-is-bijection} shows that $|\MMM_{n,k}| = |\FFF_{n,k}|$.  Since we already know
$\BBB_{n,k} \subseteq \MMM_{n,k}$ and $|\BBB_{n,k}| = |\FFF_{n,k}|$, we conclude that
\begin{equation}
\BBB_{n,k} = \MMM_{n,k},
\end{equation}
which proves the first assertion of the theorem.

The case of $S_{n,k}$ is similar.  An identical chain of reasoning, this time involving $Z_{n,k}$ instead
of $Y_{n,k}$, shows that $\NNN_{n,k}$ contains the standard monomial basis for $S_{n,k}$.  
Propososition~\ref{psi-is-bijection} implies that both $|\NNN_{n,k}|$ and
$\dim(S_{n,k})$ equal $|\OP_{n,k}|$.
\end{proof}

Theorem~\ref{m-is-basis} makes it easy to compute the Hilbert series of $R_{n,k}$ and $S_{n,k}$.

\begin{corollary}
\label{hilbert-series-corollary}
The graded vector spaces 
$R_{n,k}$ and $S_{n,k}$ have the following Hilbert series.
\begin{align}
\Hilb(R_{n,k}; q) &= \sum_{z = 0}^n   {n \choose z} q^{krz} \cdot \rev_q( [r]_q^{n-z} \cdot [k]!_{q^r} \cdot \Stir_{q^r}(n-z,k))   \\
 &=  \sum_{z = 0}^n  {n \choose z} q^{krz}  \cdot  [r]_q^{n-z}  \cdot [k]!_{q^r} \cdot  \rev_q(\Stir_{q^r}(n-z,k)). \\
\Hilb(S_{n,k}; q) &=   \rev_q ([r]_q^n \cdot [k]!_{q^r} \cdot \Stir_{q^r}(n,k) ) \\ &= 
[r]_q^n \cdot [k]!_{q^r} \cdot \rev_q (\Stir_{q^r}(n,k)).
\end{align}
\end{corollary}

\begin{proof}
By Theorem~\ref{m-is-basis} and Proposition~\ref{psi-is-bijection}, we have
\begin{align}
\Hilb(R_{n,k}; q) &= \sum_{\sigma \in \FFF_{n,k}} q^{\coinv(\sigma)}, \\
\Hilb(S_{n,k}; q) &= \sum_{\sigma \in \OP_{n,k}} q^{\coinv(\sigma)},
\end{align}
so that the proof of the corollary reduces to calculating the generating function of $\coinv$ on 
$\FFF_{n,k}$ and $\OP_{n,k}$.

It follows from the work of Steingr\'imsson  \cite{Stein}  that the generating function of $\coinv$ on $\OP_{n,k}$ is
\begin{equation}
\sum_{\sigma \in \OP_{n,k}} q^{\coinv(\sigma)} = \rev_q ([r]_q^n \cdot [k]!_{q^r} \cdot \Stir_{q^r}(n,k)),
\end{equation}
proving the desired expression for $\Hilb(S_{n,k}; q)$.  For the derivation of 
$\Hilb(R_{n,k}; q)$, simply note that a zero block $Z$ of an 
$G_n$-face $\sigma \in \FFF_{n,k}$ contributes $kr \cdot |Z|$ to $\coinv(\sigma)$.
\end{proof}

The proof of Theorem~\ref{m-is-basis} also gives the {\em ungraded} isomorphism
type of the $G_n$-modules  $R_{n,k}$ and $S_{n,k}$.

\begin{corollary}
\label{ungraded-isomorphism-type}
As {\em ungraded} 
$G_n$-modules we have
$R_{n,k} \cong \CC[\FFF_{n,k}]$ and $S_{n,k} \cong \CC[\OP_{n,k}]$.
\end{corollary}

\begin{proof}
We have the following isomorphisms of ungraded $G_n$-modules:
\begin{equation}
\CC[\xx_n]/\TT(Y_{n,k}) \cong \CC[\xx_n]/I_{n,k} \cong \CC[\FFF_{n,k}]
\end{equation}
and
\begin{equation}
\CC[\xx_n]/\TT(Z_{n,k}) \cong \CC[\xx_n]/J_{n,k} \cong \CC[\OP_{n,k}].
\end{equation}
The proof of Theorem~\ref{m-is-basis} shows that $\TT(Y_{n,k}) = I_{n,k}$ and
$\TT(Z_{n,k}) = J_{n,k}$.
\end{proof}

Theorem~\ref{m-is-basis} identifies the standard monomial bases $\MMM_{n,k}$ and
$\NNN_{n,k}$ for the quotient rings 
$R_{n,k}$ and $S_{n,k}$ with respect to the lexicographic term order.  However, checking whether
monomial $m \in \CC[\xx_n]$ is (strongly) $(n,k)$-nonskip involves checking whether $\xx(S)^r \mid m$
for all possible subsets $S \subseteq [n]$ with $|S| = n-k+1$.  The next result gives a more direct characterization
of the monomials of $\MMM_{n,k}$ and $\NNN_{n,k}$.

A {\em shuffle} of a pair of sequences
$(a_1, \dots, a_p)$ and $(b_1, \dots, b_q)$ is an interleaving $(c_1, \dots, c_{p+q})$ of these sequences
which preserves the relative order of the $a$'s and $b$'s.
The following result is an extension of \cite[Thm. 4.13]{HRS} to $r \geq 2$.

\begin{theorem}
\label{artin-basis}
We have
\begin{equation}
\MMM_{n,k} =
\left\{ x_1^{a_1} \cdots x_n^{a_n} \,:\, 
\begin{array}{c}
\text{$(a_1, \dots, a_n)$ is componentwise $\leq$ some shuffle of}  \\
\text{$(r-1, 2r-1, \dots, kr-1)$ and $(kr, \dots, kr)$}
\end{array}
\right\},  \\
\end{equation}
where there are $n-k$ copies of $kr$.
Moreover, we have
\begin{equation}
\NNN_{n,k} =
\left\{ x_1^{a_1} \cdots x_n^{a_n} \,:\, 
\begin{array}{c}
\text{$(a_1, \dots, a_n)$ is componentwise $\leq$ some shuffle of}  \\
\text{$(r-1, 2r-1, \dots, kr-1)$ and $(kr-1, \dots, kr-1)$}
\end{array}
\right\},  \\
\end{equation}
where there are $n-k$ copies of $kr-1$.
\end{theorem}

\begin{proof}
Let $\AAA_{n,k}$ and $\BBB_{n,k}$ denote the sets of monomials 
right-hand sides of the top and bottom asserted equalities,
respectively.  A direct check shows that any shuffle of $(r-1, 2r-1, \dots, kr-1)$ and $(kr, \dots, kr)$ is 
$(n,k)$-nonskip and that any shuffle of $(r-1, 2r-1, \dots, kr-1)$ and $(kr-1, \dots, kr-1)$ is
$(n,k)$-strongly nonskip.  This implies that $\AAA_{n,k} \subseteq \MMM_{n,k}$ 
and $\BBB_{n,k} \subseteq \NNN_{n,k}$.

To verify the reverse containment, consider the bijection $\Psi: \FFF_{n,k} \rightarrow \MMM_{n,k}$
of Proposition~\ref{psi-is-bijection}.  We argue that $\Psi(\FFF_{n,k}) \subseteq \AAA_{n,k}$.  
Let $\sigma \in \FFF_{n,k}$ be an $G_n$-face and let $\overline{\sigma}$ be the $G_{n-1}$-face obtained
by removing $n$ from $\sigma$.

{\bf Case 1:}  {\em $n$ is not contained in a nonzero singleton block of $\sigma$.}

In this case we have $\overline{\sigma} \in \FFF_{n-1,k}$.
We  inductively assume  $\Psi(\overline{\sigma}) \in \AAA_{n-1,k}$.  This means that there is some 
shuffle $(a_1, \dots, a_{n-1})$ of the sequences $(r-1, 2r-1, \dots, kr-1)$ and $(kr, \dots, kr)$ such that 
$\Psi(\overline{\sigma}) \mid x_1^{a_1} \cdots x_{n-1}^{a_{n-1}}$ (where there are $n-k-1$ copies of $kr$).
By the definition of $\Psi$ we have 
$\Psi(\sigma) \mid x_1^{a_1} \cdots x_{n-1}^{a_{n-1}} x_n^{kr}$, and 
$(a_1, \dots, a_{n-1}, kr)$ is a shuffle of $(r-1, 2r-1, \dots, kr-1)$ and $(kr, kr, \dots, kr)$,
where there are $n-k$ copies of $kr$.  We conclude that $\Psi(\sigma) \in \AAA_{n,k}$

{\bf Case 2:}  {\em $n$ is contained in a nonzero singleton block of $\sigma$.}

In this case we have $\overline{\sigma} \in \FFF_{n-1,k-1}$.  
We  inductively assume  $\Psi(\overline{\sigma}) \in \AAA_{n-1,k-1}$.
We have
$\Psi(\sigma) = \Psi(\overline{\sigma}) \cdot \mm(S)^r \cdot x_n^i$ for some $0 \leq i \leq kr-1$, where
$S \subseteq [n-1], |S| = n-k,$ and $\xx(S)^r \mid (\Psi(\overline{\sigma}) \cdot \mm(S)^r)$.  Consider the shuffle
$(a_1, \dots, a_n)$ of $(r-1, 2r-1, \dots, kr-1)$ and $(kr, kr, \dots, kr)$ determined by $a_j = kr$ if and only if 
$j \in S$.

We claim $\Psi(\sigma) \mid x_1^{a_1} \cdots x_n^{a_n}$, so that $\Psi(\sigma) \in \AAA_{n,k}$.  
To see this,
write $\Psi(\sigma) = x_1^{b_1} \cdots x_n^{b_n}$.
Since $\Psi(\sigma) \in \MMM_{n,k}$ we know that $0 \leq b_j \leq kr$ for all $1 \leq j \leq n$.
If $\Psi(\sigma) \nmid x_1^{a_1} \cdots x_n^{a_n}$, choose $1 \leq j \leq n$ with $a_j < b_j$; by the last sentence
we know $j \notin S$.  A direct check shows that $\xx(S \cup \{j\})^r \mid \Psi(\sigma)$, which contradicts 
 $\Psi(\sigma) \in \MMM_{n,k}$.  We conclude that $\Psi(\sigma) \in \AAA_{n,k}$.  This completes the 
 proof that $\Psi(\FFF_{n,k}) \subseteq \AAA_{n,k}$.
 
 To prove the second assertion of the theorem, one verifies  $\Psi(\OP_{n,k}) \subseteq \BBB_{n,k}$.  
 The argument follows a similar inductive pattern and is left to the reader.
\end{proof}

For example, consider the case $(n,k,r) = (5,3,2)$.  The shuffles of $(1,3,5)$ and $(6,6)$ are the ten sequences
\begin{center}
$\begin{array}{ccccc}
(1,3,5,6,6) & (1,3,6,5,6) & (1,6,3,5,6) & (6,1,3,5,6) & (1,3,6,6,5) \\
(1,6,3,6,5) & (6,1,3,6,5) & (1,6,6,3,5) & (6,1,6,3,5) & (6,6,1,3,5),
\end{array}$
\end{center}
so that the standard monomial basis $\MMM_{5,3}$ of $R_{5,3}$ with respect to the lexicographic 
term order consists of those monomials $x_1^{a_1} \cdots x_5^{a_5}$ whose exponent sequence
$(a_1, \dots, a_5)$ is componentwise $\leq$ at least one of these ten sequences.
On the other hand, the shuffles of $(1,3,5)$ and $(5,5)$ are the six sequences
\begin{center}
$\begin{array}{cccccc}
(1,3,5,5,5) & (1,5,3,5,5) & (5,1,3,5,5) & (1,5,5,3,5) & (5,1,5,3,5) & (5,5,1,3,5),
\end{array}$
\end{center}
so that the standard monomial basis $\NNN_{5,3}$ of $S_{5,3}$ consists of those monomials
$x_1^{a_1} \cdots x_5^{a_5}$ where $(a_1, \dots, a_5)$ is componentwise $\leq$ at least one of these 
six sequences.

The next result gives the reduced Gr\"obner bases of the ideals $I_{n,k}$ and $J_{n,k}$.  It is 
the extension of \cite[Thm. 4.14]{HRS} to $r \geq 2$.

\begin{theorem}
\label{groebner-basis}
Endow monomials in $\CC[\xx_n]$
with the lexicographic term order.

\begin{itemize}
\item
The variable powers
$x_1^{kr+1}, \dots, x_n^{kr+1}$, together with the polynomials 
\begin{equation*}
\overline{\kappa_{\overline{\gamma(S)}}(\xx_n^{r})}
\end{equation*}
for $S \subseteq [n]$ with $|S| = n-k+1$, form a Gr\"obner basis for the ideal $I_{n,k} \subseteq \CC[\xx_n]$.
If $n > k > 0$, this Gr\"obner basis is reduced.
\item  
The variable powers $x_1^{kr}, \dots, x_n^{kr}$, together with the polynomials
\begin{equation*}
\overline{\kappa_{\overline{\gamma(S)}}(\xx_n^{r})}
\end{equation*}
for $S \subseteq [n-1]$ with $|S| = n-k+1$, form a Gr\"obner basis for the ideal $J_{n,k} \subseteq \CC[\xx_n]$.
If $n > k > 0$, this Gr\"obner basis is reduced.
\end{itemize}
\end{theorem}

\begin{proof}
By Lemma~\ref{demazures-in-ideal}, the relevant polynomials $\overline{\kappa_{\overline{\gamma(S)}}(\xx_n^{r})}$
lie
in the ideals $I_{n,k}$ and $J_{n,k}$; the given variable powers are generators of these ideals.
By Theorem~\ref{m-is-basis}, the number of monomials which do not divide any of the initial terms 
of the given polynomials equals the dimension of the corresponding quotient ring in either case.
It follows that the given sets of polynomials are Gr\"obner bases for $I_{n,k}$ and $J_{n,k}$.

Suppose $n > k > 0$.  By Lemma~\ref{demazure-initial-term}, for any distinct polynomials $f, g$ listed in
either bullet point, the leading monomial of $f$ has coefficient $1$ and does not divide any 
monomial  in $g$.  This implies the claim about reducedness.
\end{proof}

\section{Generalized descent monomial basis}
\label{Descent}

\subsection{A straightening algorithm}
For an $r$-colored permutation $g = \pi_1^{c_1} \dots \pi_n^{c_n} \in G_n$,
let $d(g) = (d_1(g), \dots, d_n(g))$ be the sequence of nonnegative integers given by
\begin{equation}
\label{d-sequence-definition}
d_i(g) := | \{ j \in \Des(\pi_1^{c_1} \dots \pi_n^{c_n}) \,:\, j \geq i \} |.
\end{equation}
We have $d_1(g) = \des(g)$ and $d_1(g) \geq \cdots \geq d_n(g)$.
Following Bango and Biagioli \cite{BB},
we define the {\em descent monomial} 
$b_g \in \CC[\xx_n]$ by the equation
\begin{equation}
\label{gs-monomial-equation}
b_g := \prod_{i = 1}^n x_{\pi_i}^{r d_i(g) + c_i}.
\end{equation}

When $r = 1$, the monomials $b_g$ were introduced by Garisa \cite{Garsia}
and further studied by Garsia and Stanton \cite{GS}.  Garsia \cite{Garsia}
proved that the collection of monomials $\{b_g \,:\, g \in \symm_n\}$ descends to a basis for the 
coinvariant algebra attached to $\symm_n$.
When $r = 2$, a slightly different family of monomials was introduced by 
Adin, Brenti, and Roichman \cite{ABR}; they proved that their monomials descend to a basis
for the coinvariant algebra attached to the hyperoctohedral group.
Bango and Biagioli \cite{BB} introduced the collection of monomials above; they proved 
that they descend to a basis for the coinvariant algebra attached to $G_n$
(and, more generally, that an appropriate subset of them descend to a basis of the 
coinvariant algebra for the
$G(r,p,n)$ family of complex reflection groups).

We will find it convenient to extend the definition of $b_g$ somewhat to `partial colored permutations' 
$g = \pi_1^{c_1} \dots \pi_m^{c_m}$, where $\pi_1, \dots, \pi_m$ are distinct integers in $[n]$
and $0 \leq c_1, \dots, c_m \leq r-1$ are colors.  The formulae
(\ref{d-sequence-definition}) and (\ref{gs-monomial-equation}) still make sense in this case and
define a monomial $b_g \in \CC[\xx_n]$.

As an example of descent monomials, consider the case $(n,r) = (8,3)$ and 
$g = \pi_1^{c_1} \dots \pi_8^{c_8} = 3^2 7^0 1^1 6^1 8^1 2^0 4^2  5^1 \in G_8$.
We calculate $\Des(g) = \{2,6\}$, so that $d(g) = (2,2,1,1,1,1,0,0)$.
The monomial $b_g \in \CC[\xx_8]$ is given by
\begin{equation*}
b_g = x_3^8 x_7^6 x_1^4 x_6^4 x_8^4 x_2^3 x_4^2 x_5^1.
\end{equation*}
Let $\overline{g} = 6^1 8^1 2^0 4^2  5^1$ be the sequence obtained by erasing the first three letters of $g$.
We leave it for the reader to check that 
\begin{equation*}
b_{\overline{g}} =  x_6^4 x_8^4 x_2^3 x_4^2 x_5^1,
\end{equation*}
so that $b_{\overline{g}}$ is obtained by truncating $b_g$.  We formalize this as an observation.

\begin{observation}
\label{truncation-observation}
Let $g = \pi_1^{c_1} \dots \pi_n^{c_n} \in G_n$ and let 
$\overline{g} = \pi_m^{c_m} \dots \pi_n^{c_n}$ for some $1 \leq m \leq n$.  If
$b_g = x_{\pi_1}^{a_1} \cdots x_{\pi_n}^{a_n}$, then $b_{\overline{g}} = x_{\pi_m}^{a_m} \cdots x_{\pi_n}^{a_n}$.
\end{observation}

The most important property of the $b_g$ monomials will be a related
{\em Straightening Lemma} of Bango and Biagioli \cite{BB} (see also \cite{ABR}).
This  lemma  uses a certain partial order 
on monomials.
In order to define this partial order, we will attach colored permutations to monomials as follows.

\begin{defn}
\label{group-element-definition}
Let $m = x_1^{a_1} \cdots x_n^{a_n}$ be a monomial in $\CC[\xx_n]$.  Let
\begin{equation*}
g(m) = \pi_1^{c_1} \dots \pi_n^{c_n} \in G_n
\end{equation*}
be the $r$-colored permutation determined uniquely by the following 
conditions:
\begin{itemize}
\item  $a_{\pi_i} \geq a_{\pi_{i+1}}$ for all $1 \leq i < n$,
\item  if $a_{\pi_i} = a_{\pi_{i+1}}$ then $\pi_i < \pi_{i+1}$, and
\item  $a_i \equiv c_i$ (mod $r$).
\end{itemize}
\end{defn}

If $m = x_1^{a_1} \cdots x_n^{a_n}$ is a monomial in $\CC[\xx_n]$, let 
$\lambda(m) = (\lambda(m)_1 \geq \cdots  \geq \lambda(m)_n)$ be the  
nonincreasing
rearrangement of the 
exponent sequence $(a_1, \dots, a_n)$.
The following partial order on monomials was introduced in \cite[Sec. 3.3]{ABR}.

\begin{defn}
\label{partial-order-definition}
Let $m, m'  \in \CC[\xx_n]$ 
be monomials  and 
let $g(m) = \pi_1^{c_1} \dots \pi_n^{c_n}$ and $g(m') = \sigma_1^{e_1} \dots \sigma_n^{e_n}$ be the elements
of $G_n$ determined by Definition~\ref{group-element-definition}

We write $m \prec m'$ if  $\deg(m) = \deg(m')$
and one of the  following conditions holds:
\begin{itemize}
\item   $\lambda(m) <_{dom} \lambda(m')$, or 
\item $\lambda(m) = \lambda(m')$ and $\inv(\pi) > \inv(\sigma)$.
\end{itemize}
\end{defn}

Observe the numbers $\inv(\pi)$ and $\inv(\sigma)$ appearing in the second bullet
refer to the inversion numbers of the {\em uncolored} permutations $\pi, \sigma \in \symm_n$.





In order to state the Straightening Lemma, we will need to attach a length $n$ sequence 
$\mu(m) = (\mu(m)_1 \geq \cdots \geq \mu(m)_n)$ of nonnegative integers to any monomial
$m$.  The basic tool for doing this is as follows; its proof is similar to that of
\cite[Claim 5.1]{ABR}.

\begin{lemma}
\label{mu-lemma}
Let $m = x_1^{a_1} \cdots x_n^{a_n} \in \CC[\xx_n]$ be a monomial, let 
$g(m) = \pi_1^{c_1} \dots \pi_n^{c_n} \in G_n$
be the associated group element, and let 
$d(m) := d(g(m)) = (d_1 \geq \cdots \geq d_n)$.  The sequence
\begin{equation}
a_{\pi_1} - r d_1 - c_1, \dots, a_{\pi_n} - r d_n - c_n
\end{equation}
of exponents of $\frac{m}{b_{g(m)}}$ is a weakly decreasing sequence of nonnegative
multiples of $r$.
\end{lemma}

Lemma~\ref{mu-lemma} justifies the following definition.

\begin{defn}
\label{mu-definition}
Let $m = x_1^{a_1} \cdots x_n^{a_n}$ be a monomial and 
let $(a_{\pi_1} - r d_1 - c_1 \geq \dots \geq a_{\pi_n} - r d_n - c_n)$
be the weakly decreasing sequence of nonnegative multiples of $r$ guaranteed by 
Lemma~\ref{mu-lemma}.
Let $\mu(m) = (\mu(m)_1, \dots, \mu(m)_n)$ be the partition {\em conjugate to} the partition 
\begin{equation*}
\left( \frac{a_{\pi_1} - r d_1 - c_1}{r} , \dots, \frac{a_{\pi_n} - r d_n - c_n}{r} \right).
\end{equation*}
\end{defn}

As an example, consider $(n,r) = (8,3)$ and $m = x_1^7 x_2^3 x_3^{14} x_4^2 x_5^1 x_6^7 x_7^{12} x_8^7$.
We have $\lambda(m) = (14,12,7,7,7,3,2,1)$.
We calculate $g(m) \in G_8$ to be
$g(m) = 3^2 7^0 1^1 6^1 8^1 2^0 4^2 5^1$.  From this it follows that 
$d(m) = (2,2,1,1,1,1,0,0)$.  The sequence $\mu(m)$ is determined by the equation
\begin{equation*}
3 \cdot \mu(m)' = \lambda(m) - 3 \cdot d(m) - (2,0,1,1,1,0,2,1),
\end{equation*}
from which it follows that $\mu(m)' = (2,2,1,1,1,0,0,0)$ and $\mu(m) = (5,2,0,0,0,0,0,0)$.

The Straightening Lemma of Bango and Biagioli \cite{BB}
for monomials in $\CC[\xx_n]$ is as follows.

\begin{lemma}
\label{straightening-lemma}
(Bango-Biagioli \cite{BB})
Let $m = x_1^{a_1} \cdots x_n^{a_n}$ be a monomial in $\CC[\xx_n]$.  We have 
\begin{equation}
m = e_{\mu(m)}(\xx_n^r) \cdot b_{g(m)} + \Sigma,
\end{equation}
where $\Sigma$ is a linear combination of monomials $m' \in \CC[\xx_n]$ which
satisfy $m' \prec m$.
\end{lemma}

\subsection{The rings  $S_{n,k}$}
We are ready to introduce our descent-type monomials for the rings $S_{n,k}$.
This is an extension to $r \geq 1$ of the $(n,k)$-Garsia-Stanton monomials of \cite[Sec. 5]{HRS}.

\begin{defn}
\label{gs-monomial-definition}
Let  $n \geq k$. 
 The collection $\DDD_{n,k}$ of {\em $(n,k)$-descent monomials}
consists of all monomials in $\CC[\xx_n]$ of the form
\begin{equation}
b_g \cdot x_{\pi_1}^{r i_1} \cdots x_{\pi_{n-k}}^{r i_{n-k}},
\end{equation}
where $g  \in G_n$ satisfies
$\des(g) < k$ and the integer sequence $(i_1, \dots, i_{n-k})$ satisfies
\begin{equation*}
k - \des(g) > i_1 \geq \cdots \geq i_{n-k} \geq 0.
\end{equation*}
\end{defn}

As an example, consider $(n,k,r) = (7,5,2)$ and let 
$g = 2^1 5^0 6^1 1^0 3^1 4^0 7^0 \in G_7$.  It follows that 
$\Des(g) = \{2,4\}$ so that $\des(g) = 2$ and $k - \des(g) = 3$.  We have
\begin{equation*}
b_g = x_2^5 x_5^4 x_6^3 x_1^2 x_3^1,
\end{equation*}
so that Definition~\ref{gs-monomial-definition} gives rise to the following monomials in 
$\DDD_{7,5}$:
\begin{center}
$\begin{array}{ccc}
x_2^5 x_5^4 x_6^3 x_1^2 x_3^1, &
x_2^5 x_5^4 x_6^3 x_1^2 x_3^1 \cdot x_2^2, &
x_2^5 x_5^4 x_6^3 x_1^2 x_3^1 \cdot x_2^4,  \\ \\
x_2^5 x_5^4 x_6^3 x_1^2 x_3^1 \cdot x_2^2 x_5^2, &
x_2^5 x_5^4 x_6^3 x_1^2 x_3^1 \cdot x_2^4 x_5^2, &
x_2^5 x_5^4 x_6^3 x_1^2 x_3^1 \cdot x_2^4 x_5^4.
\end{array}$
\end{center}

By considering the possibilities for the sequence $(i_1 \geq \cdots \geq i_{n-k})$, we see that
\begin{equation}
|\DDD_{n,k}| \leq \sum_{g \in G_n} {n-\des(g)-1 \choose n-k} = 
\sum_{g \in G_n} {\asc(g) \choose n-k}
\end{equation}
(where we have an inequality because {\em a priori} two monomials produced by 
Definition~\ref{gs-monomial-definition} for different choices of $g$ could coincide).
If we consider an `ascent-starred' model for elements of $\OP_{n,k}$, e.g.
\begin{equation*}
2^1 _*5^1_*1^0 \, \, 6^3 \, \, 4^2_* 3^1 \in \OP_{6,3},
\end{equation*}
we see that 
\begin{equation}
\label{s-dimension-inequality}
|\DDD_{n,k}| \leq  |\OP_{n,k}| = \dim(S_{n,k}).
\end{equation}
Our next theorem implies  $|\DDD_{n,k}| = \dim(S_{n,k})$.

\begin{theorem}
\label{s-gs-basis-theorem}
The collection $\DDD_{n,k}$ of $(n,k)$-descent monomials descends to a basis of the quotient ring 
$S_{n,k}$.
\end{theorem}

\begin{proof}
By Equation~\ref{s-dimension-inequality}, we need only show that $\DDD_{n,k}$ descends to a spanning
set of the quotient ring $S_{n,k}$.  To this end, let $m = x_1^{a_1} \cdots x_n^{a_n} \in \CC[\xx_n]$ be a monomial.
We will show that the coset $m + J_{n,k}$ lies in the span of $\DDD_{n,k}$ by induction on the partial order $\prec$.

Suppose $m$ is minimal with respect to the partial order $\prec$.  Let us consider the exponent sequence
$(a_1, \dots, a_n)$ of $m$.  By $\prec$-minimality, we have  
\begin{equation*}
(a_1, \dots, a_n) = (\underbrace{a, \dots, a}_p, \underbrace{a+1, \dots, a+1}_{n-p})
\end{equation*}
for some integers $a \geq 0$ and $0 < p \leq n$.  Our analysis breaks into cases depending on the values of $a$ and $p$.
\begin{itemize}
\item 
If $a \geq r$ then
$e_n(\xx_n^r) \mid m$, so that $m \equiv 0$ in the quotient $S_{n,k}$.  
\item
If $0 \leq a < r$ and $p = n$, then $m = b_g$ where
\begin{equation*}
g = 1^a 2^a \dots n^a  \in G_n.  
\end{equation*}
\item
If $0 \leq a < r-1$ and $p < n$, then $m = b_g$ where
\begin{equation*}
g = (p+1)^{a+1} (p+2)^{a+1} \dots n^{a+1} 1^a 2^a \dots p^a \in G_n.
\end{equation*}
\item
If $a = r-1$ and $0 < p < n$, then $m = b_g$ where
\begin{equation*}
g = (p+1)^0 (p+2)^0 \dots n^0 1^{r-1} 2^{r-1} \dots p^{r-1} \in G_n.
\end{equation*}
\end{itemize}
We conclude that $m + J_{n,k}$ lies in the span of $\DDD_{n,k}$.

Now let $m = x_1^{a_1} \cdots x_n^{a_n}$ be an arbitrary monomial in $\CC[\xx_n]$.  We inductively
assume that for any monomial $m'$ in $\CC[\xx_n]$ which satisfies $m' \prec m$, the coset
$m' + J_{n,k}$ lies in the span of $\DDD_{n,k}$.  We apply the Straightening Lemma~\ref{straightening-lemma}
to $m$, which yields
\begin{equation*}
m = e_{\mu(m)}(\xx_n^r) \cdot b_{g(m)} + \Sigma,
\end{equation*}
where $\Sigma$ is a linear combination of monomials $m' \prec m$; by induction, the ring element $\Sigma + J_{n,k}$
lies in the span of $\DDD_{n,k}$.

Write $d(m) = (d_1, \dots, d_n)$ and $g(m) = (\pi_1 \dots \pi_n, c_1 \dots c_n)$.  
Since $d_1 = \des(g(m))$, if $\des(g(m)) \geq k$, we would have 
$x_{\pi_1}^{kr} \mid b_{g(m)}$, so that $m \equiv \Sigma$ modulo $J_{n,k}$ and $m$ lies in the span of $\DDD_{n,k}$.
Similarly, if $\mu(m)_1 \geq n-k+1$, then $e_{\mu(m)_1}(\xx_n^r) \mid (e_{\mu(m)}(\xx_n^r )\cdot b_{g(m)})$,
so that again $m \equiv \Sigma$ modulo $J_{n,k}$ and $m$ lies in the span of $\DDD_{n,k}$.

By the last paragraph, we may assume that
\begin{center}
$\des(g(m)) < k$ and $\mu(m)_1 \leq n-k$.
\end{center}
We have the identity
\begin{equation}
m = b_{g(m)} \cdot x_{\pi_1}^{r \cdot \mu(m)'_1} \cdots x_{\pi_n}^{r \cdot \mu(m)'_n},
\end{equation}
where $\mu(m)'$ is the partition conjugate to $\mu(m)$.  Since $\mu(m)_1 \leq n-k$, we may rewrite this identity as
\begin{equation}
m = b_{g(m)} \cdot x_{\pi_1}^{r \cdot \mu(m)'_1} \cdots x_{\pi_{n-k}}^{r \cdot \mu(m)'_{n-k}},
\end{equation}
where the sequence $\mu(m)'_1, \dots, \mu(m)'_{n-k}$ is weakly decreasing.  
If $\mu(m)'_1 < k - \des(g)$, we have $m \in \DDD_{n,k}$.  
If $\mu(m)'_1 \geq k - \des(g)$, since $r \cdot \des(g)$ is $\leq$ the power of $x_{\pi_1}$ in $b_{g(m)}$,
we have $x_{\pi_1}^{kr} \mid m$, so that $m \equiv \Sigma$ modulo $J_{n,k}$.  In either case,
we have that $m + J_{n,k}$ lies in the span of $\DDD_{n,k}$.
\end{proof}

\subsection{The rings $R_{n,k}$.}
Our aim is to expand our set of monomials $\DDD_{n,k}$ to a larger set of monomials $\ED_{n,k}$
(the `extended' descent monomials) which will descend to a basis for the rings $R_{n,k}$.  

\begin{defn}
\label{extended-gs-definition}
Let the {\em extended $(n,k)$-descent monomials}  $\ED_{n,k}$  be the set of monomials of the form
\begin{equation}
\label{extended-gs-equation}
\left( \prod_{j = 1}^z  x_{\pi_j}^{kr} \right) \cdot b_{\pi_{z+1}^{c_{z+1}} \dots \pi_n^{c_n}} \cdot
\left( x_{\pi_{z+1}}^{r \cdot i_{z+1}} x_{\pi_{z+2}}^{r \cdot i_{z+2}} \cdots x_{\pi_{n-k}}^{r \cdot i_{n-k}} \right),
\end{equation}
where
\begin{itemize}
\item  we have $0 \leq z \leq n-k$,
\item 
$\pi_1^{c_1} \dots \pi_n^{c_n} \in G_n$ is a colored permutation whose length $n-z$ suffix
 $\pi_{z+1}^{c_{z+1}} \dots \pi_n^{c_n}$ satisifes
$\des(\pi_{z+1}^{c_{z+1}} \dots \pi_n^{c_n}) < k$, and
\item  we have
\begin{equation*}
k - \des(\pi_{z+1}^{c_{z+1}} \dots \pi_n^{c_n}) > i_{z+1} \geq i_{z+2} \geq  \cdots \geq i_{n-k} \geq 0.
\end{equation*}
\end{itemize}

We also set $\ED_{n,0} := \{1\}$.
\end{defn}

As an example of Definition~\ref{extended-gs-definition}, let $(n,k,r) = (7,3,2)$, let $z = 2$, and consider
the group element 
$5^1 1^1 2^0 6^0 7^0 4^1 3^0 \in G_7$.
We have $\des(2^0 6^0 7^0 4^1 3^0) = 1$, so that 
$k - \des(2^0 6^0 7^0 4^1 3^0) = 2$.  Moreover, we have
\begin{equation*}
b_{2^0 6^0 7^0 4^1 3^0} = x_2^2 x_6^2 x_7^2 x_4^1,
\end{equation*}
so that we get the following monomials in $\ED_{7,3}$:
\begin{center}
$\begin{array}{ccc}
(x_5^6 x_1^6) \cdot (x_2^2 x_6^2 x_7^2 x_4^1), & 
(x_5^6 x_1^6) \cdot (x_2^2 x_6^2 x_7^2 x_4^1) \cdot (x_2^2), &
(x_5^6 x_1^6) \cdot (x_2^2 x_6^2 x_7^2 x_4^1) \cdot (x_2^2 x_6^2).
\end{array}$
\end{center}

Observe that the monomial defined in (\ref{extended-gs-equation}) depends only on the set of letters 
$\{\pi_1, \dots, \pi_z\}$ contained in the length $z$ prefix $\pi_1^{c_1} \dots \pi_z^{c_z}$
of $\pi_1^{c_1} \dots \pi_n^{c_n}$.
We can therefore form a typical monomial in $\ED_{n,k}$ by choosing $0 \leq z \leq n-k$, then choosing a set 
$Z \subseteq [n]$ with $|Z| = z$, then forming a typical element of $\DDD_{n-z,k}$ on the variable set 
$\{x_j \,:\, j \in [n] - Z\}$, and finally multiplying by the product $\prod_{j \in Z} x_j^{kr}$.
By Theorem~\ref{s-gs-basis-theorem}, there are $|\OP_{n-z,k}|$ monomials in $\DDD_{n-z,k}$, and all
of the exponents in these monomials are $< kr$.  It follows that 
\begin{equation}
|\ED_{n,k}| = \sum_{z = 0}^{n-k} {n \choose z} |\DDD_{n-z,k}| = \sum_{z = 0}^{n-k} {n \choose z} |\OP_{n-z,k}|
= |\FFF_{n,k}| = \dim(R_{n,k}).
\end{equation}
We will show $\ED_{n,k}$ descends to a spanning set of $R_{n,k}$, and hence descends to a basis
of $R_{n,k}$.

\begin{theorem}
\label{r-gs-basis-theorem}
The set $\ED_{n,k}$
of extended $(n,k)$-descent monomials descends to a basis of $R_{n,k}$.
\end{theorem}

\begin{proof}
Let $m = x_1^{a_1} \cdots x_n^{a_n}$ be a monomial in $\CC[\xx_n]$.  We argue that the coset
$m + I_{n,k} \in R_{n,k}$ lies in the  span of $\ED_{n,k}$.

Suppose first that $m$ is minimal with respect to $\prec$.  The exponent sequence $(a_1, \dots, a_n)$
has the form 
\begin{equation*}
(a_1, \dots, a_n) = (\underbrace{a, \dots, a}_p, \underbrace{a+1, \dots, a+1}_{n-p})
\end{equation*}
for some $a \geq 0$ and $0 < p \leq n$.
The same analysis as in the proof of Theorem~\ref{s-gs-basis-theorem} implies that $m \equiv 0$ (mod $I_{n,k})$
or $m \in \DDD_{n,k} \subseteq \ED_{n,k}$.

Now let $m = x_1^{a_1} \cdots x_n^{a_n} \in \CC[\xx_n]$ be an arbitrary monomial and form 
the sequence $d(m) = (d_1, \dots, d_n)$ and the colored permutation $g(m) = \pi_1^{c_1} \dots \pi_n^{c_n}$.
Apply the 
Straightening Lemma~\ref{straightening-lemma} to write
\begin{equation}
m = e_{\mu(m)} (\xx_n^r) \cdot b_{g(m)} + \Sigma,
\end{equation}
where $\Sigma$ is a linear combination of monomials $m' \in \CC[\xx_n]$ with $m' \prec m$.

We inductively assume that the ring element $\Sigma + I_{n,k}$ lies in the span of $\ED_{n,k}$.
If $\mu(m)_1 \geq n-k+1$, then $m \equiv \Sigma$ (mod $I_{n,k}$), so that $m + I_{n,k}$ lies in the 
span of $\ED_{n,k}$.  If $\des(g(m)) > k+1$, then $x_{\pi_1}^{(k+1)r} \mid b_{g(m)}$, so 
that again $m \equiv \Sigma$ (mod $I_{n,k}$) and $m + I_{n,k}$ lies in the span of $\ED_{n,k}$.

By the last paragraph, we may assume 
\begin{center}
$\mu(m)_1 \leq n-k$ and $\des(g(m)) \leq k$.
\end{center}
Our analysis breaks up into two cases depending on whether $\des(g(m)) < k$ or $\des(g(m)) = k$.

{\bf Case 1:}  {\em $\mu(m)_1 \leq n-k$ and $\des(g(m)) < k$.}

If any element in the exponent sequence $(a_1, \dots, a_n)$ of $m$ is $> kr$, then $m \equiv 0$ (mod $I_{n,k}$).
We may therefore  assume  $a_j \leq kr$ for all $j$.  

Since we have $\mu(m)_1 \leq n-k$, we have the identity
\begin{equation}
m = b_{g(m)} \cdot x_{\pi_1}^{r \cdot \mu(m)'_1} \cdots x_{\pi_{n-k}}^{r \cdot \mu(m)'_{n-k}}.
\end{equation}
If $\mu(m)'_1 < k - \des(g(m))$, we have
$m \in \DDD_{n,k} \subseteq \ED_{n,k}$.   If $\mu(m)_1' > k - \des(g(m))$, we have
$x_{\pi_1}^{(k+1) \cdot r} \mid m$, which contradicts $a_{\pi_1} \leq kr$.

By the last paragraph, we may assume  $\mu(m)'_1 = k - \des(g(m))$.  Since every term in 
the weakly decreasing sequence 
$(a_{\pi_1},  \dots, a_{\pi_n})$ is $\leq kr$, there exists an index $1 \leq z \leq n$ such that
$(a_{\pi_1}, \dots, a_{\pi_n}) = (kr, \dots, kr, a_{\pi_{z+1}}, \dots, a_{\pi_n})$,
where $a_{\pi_{z+1}} < kr$.  Since every exponent in $b_{g(m)}$ is $< kr$, we in fact have
$1 \leq z \leq n-k$.  

Let $\overline{g}$ be the partial colored permutation
$\overline{g} := \pi_{z+1}^{c_{z+1}} \dots \pi_n^{c_n}$.
Applying Observation~\ref{truncation-observation}, we have
\begin{align}
m &= b_{g(m)} \cdot x_{\pi_1}^{r \cdot \mu(m)'_1} \cdots x_{\pi_{n-k}}^{r \cdot \mu(m)'_{n-k}} \\
&=  \left( \prod_{j = 1}^z x_{\pi_j}^{kr}  \right) \cdot b_{\overline{g}}  
\cdot x_{\pi_{z+1}}^{r \cdot \mu(m)'_{z+1}} \cdots x_{\pi_{n-k}}^{r \cdot \mu(m)'_{n-k}},
\end{align}
for $1 \leq z \leq n-k$.   The monomial 
$b_{\overline{g}}  \cdot x_{\pi_{z+1}}^{r \cdot \mu(m)'_{z+1}} \cdots x_{\pi_{n-k}}^{r \cdot \mu(m)'_{n-k}}$
only involves the variables $x_{\pi_{z+1}}, \dots, x_{\pi_n}$, and every exponent in this product is
$< kr$.  If $\mu(m)'_{z+1} \geq k - \des(\overline{g})$, we would have the divisibility
$x_{\pi_{z+1}}^{kr} \mid 
b_{\overline{g}}  \cdot x_{\pi_{z+1}}^{r \cdot \mu(m)'_{z+1}} \cdots x_{\pi_{n-k}}^{r \cdot \mu(m)'_{n-k}}$,
which is a contradiction.
It follows that $\mu(m)'_{z+1} < k - \des(\overline{g})$, which implies that $m \in \ED_{n,k}$.

We conclude that the coset $m + I_{n,k}$ lies in the span of $\ED_{n,k}$, which completes this case.

{\bf Case 2:}  {\em $\mu(m)_1 \leq n-k$ and $\des(g(m)) = k$.}

As in the previous case, we may assume that every exponent appearing in the monomial $m$ is $\leq kr$.
We again write
\begin{equation}
m = b_{g(m)} \cdot x_{\pi_1}^{r \cdot \mu(m)'_1} \cdots x_{\pi_{n-k}}^{r \cdot \mu(m)'_{n-k}}
\end{equation}
and have $(a_{\pi_1} \geq \cdots \geq a_{\pi_n}) = (kr, \dots, kr, a_{\pi_{z+1}}, \dots, a_{\pi_n})$
for some $1 \leq z \leq n-k$.  Define the partial colored permutation
$\overline{g} := \pi_{z+1}^{c_{z+1}} \dots \pi_n^{c_n}$.

Since the exponent of $x_{\pi_{z+1}}$ in $m$
is $\geq r \cdot \des(\overline{g})$, we have $\des(\overline{g}) < k$.  If $\mu(m)'_{z+1} \geq k - \des(\overline{g})$,
the exponent of $x_{\pi_{z+1}}$ in $m$ would be $\geq kr$, so we must have 
$\mu(m)'_{z+1} < k - \des(\overline{g})$.
Using Observation~\ref{truncation-observation} to
write 
\begin{align}
m &= b_{g(m)} \cdot x_{\pi_1}^{r \cdot \mu(m)'_1} \cdots x_{\pi_{n-k}}^{r \cdot \mu(m)'_{n-k}} \\
&=  \left( \prod_{j = 1}^z x_{\pi_j}^{kr}  \right) \cdot b_{\overline{g}}  
\cdot x_{\pi_{z+1}}^{r \cdot \mu(m)'_{z+1}} \cdots x_{\pi_{n-k}}^{r \cdot \mu(m)'_{n-k}},
\end{align}
we see that $m \in \ED_{n,k}$.
\end{proof}

The following lemma involving expansions of monomials $m$ into the 
$\ED_{n,k}$ basis of $R_{n,k}$ will be useful in the next section.  For $0 \leq z \leq n-z$, let 
$\ED_{n,k}(z)$ be the subset of monomials in $\ED_{n,k}$ which contain exactly $z$ variables with power
$kr$.  We get a stratification
\begin{equation}
\ED_{n,k} = \ED_{n,k}(0) \uplus \ED_{n,k}(1) \uplus \cdots \uplus \ED_{n,k}(n-k).
\end{equation}
For convenience, we set $\ED_{n,k}(z) = \varnothing$ for $z > n-k$.

\begin{lemma}
\label{zero-stability-lemma}
Let $(a_1, \dots, a_n)$ satisfy $0 \leq a_i \leq kr$ for all $i$, let 
$m = x_1^{a_1} \cdots x_n^{a_n} \in \CC[\xx_n]$ be the corresponding monomial, and let 
$z := | \{1 \leq i \leq n \,:\, a_i = kr \} |$.  The expansion of $m + I_{n,k}$ in the basis $\ED_{n,k}$ of $R_{n,k}$
only involves terms in 
$\ED_{n,k}(0) \uplus \ED_{n,k}(1) \uplus \cdots \uplus \ED_{n,k}(z)$.
\end{lemma}

\begin{proof}
Applying the Straightening Lemma~\ref{straightening-lemma} to $m$, we get
\begin{equation}
m = e_{\mu(m)}(\xx_n^r) \cdot b_{g(m)} + \Sigma,
\end{equation}
where $\Sigma$ is a linear combination of monomials $m'$ in $\CC[\xx_n]$ which satisfy $m' \prec m$.
The proof of Theorem~\ref{r-gs-basis-theorem} shows that either
\begin{itemize}
\item the monomial $m$ is an element of $\ED_{n,k}$, and hence an element of $\ED_{n,k}(z)$, or
\item we have $m \equiv \Sigma$ (mod $I_{n,k}$).
\end{itemize}
If the first bullet holds, we are done.  We may therefore assume that $m \equiv \Sigma$ (mod $I_{n,k}$).

Let $m' = x_1^{a'_1} \cdots x_n^{a'_n}$ be a monomial
appearing in $\Sigma$.  
The dominance  relation $\lambda(m') \leq_{dom} \lambda(m)$ implies  
$| \{ 1 \leq i \leq n \,:\, a'_i = kr \} | \leq z$.  We may therefore apply the logic of the last paragraph to each such
monomial $m'$, and iterate. 
\end{proof}

\section{Frobenius series}
\label{Frobenius}

In this section we will determine the graded isomorphism types of the rings $R_{n,k}$ and $S_{n,k}$.
When $r = 1$, this was carried out for the rings $S_{n,k}$ in \cite[Sec. 6]{HRS}.  
It turns out that the methods developed in \cite[Sec. 6]{HRS} generalize fairly readily to the $S$ rings, but not
the $R$ rings.  Our approach will be to describe the $R$ rings in terms of the $S$ rings, and then
describe the isomorphism type of the $S$ rings.

\subsection{Relating  $R$ and $S$ }
In this section, we describe the graded isomorphism type of $R_{n,k}$ in terms of the rings
$S_{n,k}$.   The result here is as follows.

\begin{proposition}
\label{r-to-s-reduction}
We have an isomorphism of graded $G_n$-modules
\begin{equation}
R_{n,k} \cong \bigoplus_{z = 0}^{n-k} \Ind_{G_{(n-z,z)}}^{G_n}(S_{n-z,k}^r \otimes \CC_{krz}).
\end{equation}
Here $\CC_{krz}$ is a copy of the trivial $1$-dimensional representation of $G_z$ sitting in degree $krz$.

Equivalently, we have the identity
\begin{equation}
\grFrob(R_{n,k}; q) = \sum_{z = 0}^{n-k} q^{krz} \bm{s}_{(\varnothing, \dots, \varnothing, (z))}(\xx)
\cdot \grFrob(S_{n-z,k}^r; q).
\end{equation}
\end{proposition}

\begin{proof} 
For $0 \leq z \leq n-k$, let $R_{n,k}(z)$ be the subspace of $R_{n,k}$ given by
\begin{equation}
R_{n,k}(z) := \mathrm{span}_{\CC} \{ x_1^{a_1} \cdots x_n^{a_n} + I_{n,k} \,:\, 
\text{$0 \leq a_i \leq kr$ and at most $z$ of  $a_1, \dots, a_n$ equal $kr$} \}.
\end{equation}
It is clear that $R_{n,k}(z)$ is graded and stable under the action of $G_n$.  We also have a filtration
\begin{equation}
R_{n,k}(0) \subseteq R_{n,k}(1) \subseteq \cdots \subseteq R_{n,k}(n-k) = R_{n,k}.
\end{equation}
It follows that there is an isomorphism of graded $G_n$-modules 
\begin{equation}
R_{n,k} \cong Q_{n,k}^r(0) \oplus Q_{n,k}^r(1) \oplus \cdots \oplus Q_{n,k}^r(n-k),
\end{equation}
where $Q_{n,k}^r(z) := R_{n,k}(z)/R_{n,k}(z-1)$.

Consider the stratification $\ED_{n,k} = \ED_{n,k}(0) \uplus \ED_{n,k}(1) \uplus \cdots \uplus \ED_{n,k}(n-k)$
of the basis $\ED_{n,k}$ of $R_{n,k}$.  
The containment $\ED_{n,k}(z') \subseteq R_{n,k}(z)$ for $z' \leq z$ implies 
\begin{equation}
\dim(R_{n,k}(z)) \geq | \ED_{n,k}(0)| + |\ED_{n,k}(1)| + \cdots + |\ED_{n,k}(z)|.
\end{equation}
On the other hand, Lemma~\ref{zero-stability-lemma} implies that $R_{n,k}(z)$ is spanned by 
(the image of the monomials in)
$\biguplus_{z' = 0}^z \ED_{n,k}(z')$.
It follows that 
\begin{equation}
\dim(R_{n,k}(z)) = | \ED_{n,k}(0)| + |\ED_{n,k}(1)| + \cdots + |\ED_{n,k}(z)|.
\end{equation}
and $\biguplus_{z' = 0}^z \ED_{n,k}(z')$ descends to a basis of $R_{n,k}(z)$.
Consequently, the set $\ED_{n,k}(z)$ descends to a basis for $Q_{n,k}^r(z)$.

Fix $0 \leq z \leq n-k$.
It follows from the definition of 
$\ED_{n,k}(z)$ that
\begin{equation}
\dim(Q_{n,k}^r(z)) = |\ED_{n,k}(z)| = {n \choose z} \cdot |\OP_{n-z,k}| = {n \choose z} \cdot \dim(S_{n,k}),
\end{equation}
which coincides with the dimension of 
$\Ind_{G_{(n-z,z)}}^{G_n}(S_{n-z,k}^r \otimes \CC_{krz})$.  We claim that we have
an isomorphism of graded $G_n$-modules
\begin{equation}
\label{main-module-isomorphism}
Q_{n,k}^r(z) \cong \Ind_{G_{(n-z,z)}}^{G_n}(S_{n-z,k}^r \otimes \CC_{krz}).
\end{equation}

In order to prove the isomorphism (\ref{main-module-isomorphism}), 
for any $T \subseteq [n]$, let $G_{[n] - T}$ be the group of $r$-colored permutations on the index set $[n] - T$ and
let $S_{n-z,k}(T)$ be the module  $S_{n-z,k}$ 
in the variable set $\{x_j \,:\, j \in T\}$.
Any group element $g \in G_{[n] - T}$ acts trivially on the product
$\prod_{j \notin T} x_j^{kr}$.
We may therefore interpret the induction on the 
right-hand side of (\ref{main-module-isomorphism}) as 
\begin{equation}
\Ind_{G_{(z,n-z)}}^{G_n}(S_{n-z,k} \otimes \CC_{krz}) \cong
\bigoplus_{|T| = n-z} S_{n-z,k}(T) \otimes \mathrm{span} \left\{ \prod_{j \notin T} x_j^{kr} \right\},
\end{equation}
which reduces our task to proving 
\begin{equation}
\label{modified-module-isomorphism}
Q_{n,k}^r(z) \cong \bigoplus_{|T| = n-z} S_{n-z,k}(T) \otimes \mathrm{span} \left\{ \prod_{j \notin T} x_j^{kr} \right\}.
\end{equation}

The set  of monomials $\EGS_{n,k}(z)$ in $\CC[\xx_n]$ descends to a vector space basis of the
graded modules appearing on either side of
 (\ref{modified-module-isomorphism}); the corresponding identification of cosets
gives rise to an isomorphism 
\begin{equation}
\varphi: Q_{n,k}^r(z) \rightarrow 
\bigoplus_{|T| = n-z} S_{n-z,k}^r(T) \otimes \mathrm{span} \left\{ \prod_{j \notin T} x_j^{kr} \right\}.
\end{equation}
of graded vector spaces.
It is clear that $\varphi$ commutes with the action of the diagonal subgroup 
$\ZZ_r \times \cdots \times \ZZ_r \subseteq G_n$; we need only show that $\varphi$ commutes with the action 
of $\symm_n$.

The proof that the map $\varphi$ commutes with the action of $\symm_n$ uses straightening. 
Let $m = x_1^{a_1} \cdots x_n^{a_n} \in \ED_{n,k}(z)$ be a typical
basis element and let $\pi.m = x_{\pi_1}^{a_1} \cdots x_{\pi_n}^{a_n}$ be the image of $m$ under
a typical permutation $\pi \in \symm_n$.

If $\pi.m \in \ED_{n,k}(z)$ the definition of $\varphi$ yields $\varphi(\pi.m) = \pi.\varphi(m)$.

If
$\pi.m \notin \ED_{n,k}(z)$,
by Lemma~\ref{straightening-lemma}  we can write
$\pi.m = e_{\mu(\pi.m)}(\xx_n^r) \cdot b_{g(\pi.m)} + \Sigma$, where $\Sigma$ is a linear 
combination of monomials in $\CC[\xx_n]$ which are $\prec \pi.m$.
As in the proof of Lemma~\ref{zero-stability-lemma}, since $m \in \ED_{n,k}(z)$ but
$\pi.m \notin \ED_{n,k}(z)$, we know that
$\pi.m \equiv \Sigma$ in the modules on either side of Equation~\ref{modified-module-isomorphism}.
Iterating this procedure, we see that $\pi.m$ has the same expansion into the bases induced from
$\ED_{n,k}(z)$ on either side of Equation~\ref{modified-module-isomorphism}.  
This proves that the map $\varphi$ is $\symm_n$-equivariant, so that
$\varphi$ is an isomorphism of graded $G_n$-modules.
\end{proof}

\subsection{The rings $S_{n,k,s}$}

By Proposition~\ref{r-to-s-reduction}, the graded isomorphism type of $R_{n,k}$ is
determined by the graded isomorphism type of $S_{n,k}$.  The remainder 
of this section will focus on the rings $S_{n,k}$.  
As in \cite[Sec. 6]{HRS}, to determine the graded isomorphism type of $S_{n,k}$
we will introduce a more general class of quotients.

\begin{defn}
Let $n, k, s$ be positive integers with $n \geq k \geq s$.  
Define $J_{n,k,s} \subseteq \CC[\xx_n]$ to be the ideal
\begin{equation*}
J_{n,k,s} := \langle x_1^{kr}, \dots , x_n^{kr}, e_n(\xx_n^r), e_{n-1}(\xx_n^r), \dots, e_{n-s+1}(\xx_n^r) \rangle.
\end{equation*}
Let $S_{n,k,s} := \CC[\xx_n]/J_{n,k,s}$ be the corresponding quotient ring.
\end{defn}

When $s = k$ we have $J_{n,k,k} = J_{n,k}$, so that $S_{n,k,k} = S_{n,k}$.
Our aim for the remainder of this section is to build a combinatorial model for the quotient
$S_{n,k,s}$ using the point orbit technique of Section~\ref{Hilbert}.
To this end, for $n \geq k \geq s$ let $\OP_{n,k,s}$ denote the collection of $r$-colored $k$-block
ordered set partitions $\sigma = (B_1 \mid \cdots \mid B_k)$ of $[n + (k-s)]$ such that,
for $1 \leq i \leq k-s$, we have $n+i \in B_{s+i}$ and $n+i$ has color $0$.
For example, we have
\begin{equation*}
( 2^0 3^2  \mid 1^2 6^0  \mid {\bf 7^0} \mid 5^1 7^2  {\bf 8^0} \mid 4^1   {\bf 9^0} ) \in \OP^3_{6,5,2}.
\end{equation*}

Given $\sigma \in \OP_{n,k,s}$, we will refer to the letters $n+1, n+2, \dots, n+(k-s)$ as {\em big};
the remaining letters will be called {\em small}.  
The group $G_n$ acts on $\OP_{n,k,s}$ by acting on the small letters.
We model this action with a point set as follows.

\begin{defn}
Fix positive real numbers
$0 < \alpha_1 < \cdots < \alpha_k$.
 Let $Z_{n,k,s} \subseteq \CC^{n+(k-s)}$ be the collection of 
points $(z_1, \dots, z_n, z_{n+1}, \dots, z_{n+k-s})$ such that
\begin{itemize}
\item  we have $z_i \in \{ \zeta^c  \alpha_j \,:\, 0 \leq c \leq r-1, \, \, 1 \leq j \leq k\}$ for all $1 \leq i \leq n + (k-s)$,
\item  we have $\{\alpha_1, \dots, \alpha_k\} = \{|z_1|, \dots, |z_n| \}$, and
\item  we have $z_{n+i} = \alpha_{s+i}$ for all $1 \leq i \leq k-s$.
\end{itemize}
\end{defn}

It is evident that the point set $Z_{n,k,s}$ is stable under the action of $G_n$ on the first $n$ 
coordinates of $\CC^{n + (k-s)}$ and that $Z_{n,k,s}$ is isomorphic to the action of 
$G_n$ on $\OP_{n,k,s}$.

Let $\II(Z_{n,k,s}) \subseteq \CC[\xx_{n+(k-s)}]$ be the ideal of polynomials which vanish on $Y_{n,k,s}$ and let
$\TT(Y_{n,k,s}) \subseteq \CC[\xx_{n+(k-s)}]$ be the corresponding top component ideal.
Since $x_{n+i} - \alpha_{n+i} \in \II(Y_{n,k,s})$ for all $1 \leq i \leq k-s$, we have $x_{n+i} \in \TT(Y_{n,k,s})$.
Let $\varepsilon: \CC[\xx_{n+(k-s)}] \twoheadrightarrow \CC[\xx_n]$ be the map which evaluates $x_{n+i} = 0$ for all
$1 \leq i \leq k-s$ and let $T_{n,k,s} := \varepsilon(\TT(Y_{n,k,s}))$ be the image of $\TT(Y_{n,k,s})$ under 
$\varepsilon$.
Then $T_{n,k,s}$ is an ideal in $\CC[\xx_n]$ and we have an identification of 
$G_n$-modules
\begin{equation*}
\CC[\OP_{n,k,s}] \cong \CC[\xx_{n+(k-s)}]/\II(Y_{n,k,s}) \cong 
\CC[\xx_{n+(k-s)}]/\TT(Y_{n,k,s}) \cong \CC[\xx_n]/T_{n,k,s}.
\end{equation*}

It will develop that $J_{n,k,s} = T_{n,k,s}$.  We can generalize
Lemma~\ref{i-contained-in-t} to prove one containment right away.

\begin{lemma}
\label{j-contained-in-t-generalized}
We have $J_{n,k,s} \subseteq T_{n,k,s}$.
\end{lemma}

\begin{proof}
We show that every generator of $J_{n,k,s}$ is contained in $T_{n,k,s}$.

For $1 \leq i \leq n$ we have
$\prod_{j = 1}^r \prod_{c = 0}^{r-1} (x_i - \zeta^c  \alpha_i) \in \II(Y_{n,k,s})$, so that
$x_i^{kr} \in T_{n,k,s}$.

The proof of Lemma~\ref{i-contained-in-t} shows that $e_j(\xx_{n+(k-s)}^r) \in \TT(Y_{n,k,s})$
for all $j \geq n-s+1$.  Applying the evaluation map $\varepsilon$ gives
$\varepsilon: e_j(\xx_{n+(k-s)}^r) \mapsto e_j(\xx_n^r) \in T_{n,k,s}$.
\end{proof}

Proving the equality $J_{n,k,s} = T_{n,k,s}$ will involve a dimension count.  To facilitate this,
let us identify some terms in the initial ideal of $I_{n,k,s}$.
The following is a generalization of Lemma~\ref{skip-leading-terms}; its proof is left to the reader.

\begin{lemma}
\label{skip-leading-terms}
Let $<$ be the lexicographic term order on monomials in $\CC[\xx_n]$ and let 
$\initial_<(J_{n,k,s})$ be the initial ideal of $J_{n,k,s}$.  We have 
\begin{itemize}
\item  $x_i^{kr} \in \initial_<(J_{n,k,s})$ for $1 \leq i \leq n$, and
\item  $\xx(S)^r \in \initial_<(J_{n,k,s})$ for all $S \subseteq [n]$ with $|S| = n-s+1$.
\end{itemize}
\end{lemma}

Lemma~\ref{skip-leading-terms} motivates the following generalization of strongly
$(n,k)$-nonskip monomials.

\begin{defn}
Let $\NNN_{n,k,s}$ be the collection of monomials $m \in \CC[\xx_n]$ such that
\begin{itemize}
\item $x_i^{kr} \nmid m$ for all $1 \leq i \leq m$, and
\item $\xx(S)^r \nmid m$ for all $S \subseteq [n]$ with $|S| = n-s+1$.
\end{itemize}
\end{defn}

By Lemma~\ref{skip-leading-terms}, the set $\NNN_{n,k,s}$ contains the standard monomial basis
of $S_{n,k,s}$; we will prove that these two sets of monomials coincide.
Let us first observe a relationship between the monomials in $\NNN_{n,k,s}$ and those
in $\NNN_{n+(k-s),k}$.

\begin{lemma}
\label{nonskip-monomial-factor}
If $x_1^{a_1} \cdots x_n^{a_n} x_{n+1}^{a_{n+1}} \cdots x_{n+(k-s)}^{a_{n+(k-s)}} \in \NNN_{n+(k-s),k}$,
then $x_1^{a_1} \cdots x_n^{a_n} \in \NNN_{n,k,s}$.  

Conversely, if $x_1^{a_1} \cdots x_n^{a_n} \in \NNN_{n,k,s}$ and
$0 \leq a_{n+1} < a_{n+2} < \cdots < a_{n+(k-s)} < kr$ satisfy 
\begin{equation*}
a_{n+1} \equiv a_{n+2} \equiv \cdots \equiv a_{n+(k-s)} \equiv i \text{ (mod $r$)}
\end{equation*}
for some $0 \leq i \leq r-1$   , then
$x_1^{a_1} \cdots x_n^{a_n} x_{n+1}^{a_{n+1}} \cdots x_{n+(k-s)}^{a_{n+(k-s)}} \in \NNN_{n+(k-s),k}$.
\end{lemma}

\begin{proof}
The first statement is clear from the definitions of $\NNN_{n+(k-s),k}$ and $\NNN_{n,k,s}$.  For the second statement,
let $m' := x_1^{a_1} \cdots x_n^{a_n} \in \NNN_{n,k,s}$ and let $0 \leq a_{n+1} < a_{n+2} < \cdots < a_{n+(k-s)} < kr$
be as in the statement of the lemma.
We argue that $m := x_1^{a_1} \cdots x_n^{a_n} x_{n+1}^{a_{n+1}} \cdots x_{n+(k-s)}^{a_{n+(k-s)}} \in \NNN_{n+(k-s),k}$.

Since $m' \in \NNN_{n,k,s}$, we know that $x_i^{kr} \nmid m$ for $1 \leq i \leq n + (k-s)$.  Let $S \subseteq [n + (k-s)]$ 
satisfy $|S| = n+(k-s)$.  We need to show $\xx(S)^r \nmid m$.  
If $S \subseteq [n]$, then $\xx(S)^r \nmid m$ because
$\xx(S)^r \nmid m'$.  On the other hand, if $n + i \in S$ for some $1 \leq i \leq k-r$, 
the power $p_{n+i}$ of $x_{n+i}$
in $\xx(S)^r$ is $\geq r \cdot (s+i)$.    However, our assumptions on $(a_{n+1}, a_{n+2}, \dots, a_{n+(k-s)})$ force
$a_{n+i} < r \cdot (k - (s-i))  \leq r \cdot (s+i)$, which implies $\xx(S)^r \nmid m$.
\end{proof}

We use the map $\Psi$ from Section~\ref{Hilbert} to count $\NNN_{n,k,s}$.

\begin{lemma}
\label{size-of-n}
We have $|\NNN_{n,k,s}| = |\OP_{n,k,s}|$.
\end{lemma}

\begin{proof}
Consider the bijection $\Psi: \OP_{n+(k-s),k} \rightarrow \NNN_{n+(k-s),k}$ from Section~\ref{Hilbert}.
We have $\OP_{n,k,s} \subseteq \OP_{n+(k-s),k}$.  We leave it for the reader to check that
\begin{equation*}
\Psi(\OP_{n,k,s}) = \NNN'_{n,k,s},
\end{equation*}
where $\NNN'_{n,k,s}$ consists of those monomials  
$x_1^{a_1} \cdots x_{n}^{a_n} x_{n+1}^{a_{n+1}} \cdots x_{n+(k-s)}^{a_{n+(k-s)}} \in \NNN_{n+(k-s),k}$ 
which satisfy
\begin{equation*}
(a_{n+1}, a_{n+2}, \dots, a_{n+(k-s)}) = (rs + (r-1), r(s+1) + (r-1), \dots, r(k-1) + (r-1)).
\end{equation*}
(The $+(r-1)$ terms come from the fact that the letters $n+1, \dots, n+(k-s)$ all have color $0$ and
$\Psi$ involves a {\em complementary} color contribution.)
Lemma~\ref{nonskip-monomial-factor} applies to show $|\NNN'_{n,k,s}| = |\NNN_{n,k,s}|$.
\end{proof}

We are ready to determine the ungraded isomorphism type of the $G_n$-module
$S_{n,k,s}$.

\begin{lemma}
\label{s-dimension-lemma-generalized}
We have $S_{n,k,s} \cong \CC[\OP_{n,k,s}]$.  In particular, we have
$\dim(S_{n,k,s}) = |\OP_{n,k,s}|$.
\end{lemma}

\begin{proof}
By Lemma~\ref{j-contained-in-t-generalized} we have $\dim(S_{n,k,s}) \geq |\OP_{n,k,s}|$.
Lemma~\ref{skip-leading-terms} and
Lemma~\ref{size-of-n} imply that the standard monomial basis of $S_{n,k,s}$ with respect to the 
lexicographic term order has size $\leq |\NNN_{n,k,s}| = |\OP_{n,k,s}|$, so that 
$\dim(S_{n,k,s}) = |\OP_{n,k,s}|$.  Lemma~\ref{j-contained-in-t-generalized} gives a 
$G_n$-module surjection $S_{n,k,s} \twoheadrightarrow \CC[\OP_{n,k,s}]$;
dimension counting shows that this surjection is an isomorphism.
\end{proof}

\subsection{Idempotents and $e_j(\xx^{(i^*)})^{\perp}$}
For $1 \leq j \leq n$ and  $1 \leq i \leq r$, 
we want to develop a module-theoretic analog of acting by the operator 
$e_j(\xx^{(i^*)})^{\perp}$ on Frobenius images.
If $V$ is a $G_n$-module, acting by $e_j(\xx^{(i^*)})^{\perp}$ on 
$\Frob(V)$ will correspond to taking the image of $V$ under a certain group algebra
idempotent $\epsilon_{i,j} \in \CC[G_n]$.

Let $1 \leq j \leq n$ and consider the corresponding parabolic subgroup 
$G_{(n-j,j)} = G_{n-j} \times G_j$ of $G_n$.  
The factor $G_j$ acts on the {\em last} $j$ letters $n-j+1, \dots, n-1, n$ of $\{1, 2, \dots, n\}$.

For $1 \leq j \leq n$ and $1 \leq i \leq r$,
let $\epsilon_{i,j}$ be the idempotent in the group algebra of $G_n$ given by
\begin{equation}
\epsilon_{i,j} := \frac{1}{r^j \cdot j!}  
\sum_{g \in \ZZ_r \wr \symm_j} \sign(g) \cdot \overline{\chi(g)^i} \cdot g  \in \CC[G_n].
\end{equation}
(Recall that $\chi(g)$ is the product of the nonzero entries in the $j \times j$ monomial matrix $g$.)
The idempotent $\epsilon_{i,j}$ commutes with the action of $G_{n-j}$.  In particular,
if $V$ is a $G_n$-module, then $\epsilon_{i,j} V$ is a 
$G_{n-j}$-module.
The relationship between $\Frob(V)$ and $\Frob(\epsilon_{i,j}V)$ is as follows.

\begin{lemma}
\label{e-perp-on-v}
Let $V$ be a $G_n$-module, let $1 \leq j \leq n$, and let $1 \leq i \leq r$.
We have
\begin{equation}
\Frob(\epsilon_{i,j} V) = e_j(\xx^{(i^*)})^{\perp} \Frob(V).
\end{equation}
In particular, if $V$ is graded, we have
\begin{equation}
\grFrob(\epsilon_{i,j} V; q) = e_j(\xx^{(i^*)})^{\perp} \grFrob(V; q).
\end{equation}
\end{lemma}

\begin{proof}
The proof is a standard application of Frobenius reciprocity
and symmetric function theory (and can be found in \cite{GP} in the case $r = 1$).

It suffices to prove this lemma when $V$ is irreducible, so let $V = \bm{S^{\lambda}}$ for some $r$-partition
$\blambda \vdash_r n$.  Consider the parabolic
subgroup $G_{(n-j,j)} \subseteq G_n$.  Irreducible representations 
of $G_{(n-j,j)}$ have the form $\bm{S^{\mu}} \otimes \bm{S^{\nu}}$
for $\bm{\mu} \vdash_r n-j$  and $\bm{\nu} \vdash_r j$.  By Frobenius reciprocity, 
we have
\begin{align*}
\text{(multiplicity of $\bm{S^{\mu}} \otimes \bm{S^{\nu}}$ in 
$\mathrm{Res}^{G_n}_{G_{(n-j,j)}} \bm{S^{\lambda}}$)} &=
\text{(multiplicity of $\bm{S^{\lambda}}$ in 
$\mathrm{Ind}^{G_n}_{G_{(n-j,j)}} \bm{S^{\mu}} \otimes \bm{S^{\nu}}$)} \\
&= 
\text{(coefficient of $\bm{s_{\lambda}(x)}$ in $\bm{s_{\mu}(x)} \cdot \bm{s_{\nu}(x)}$)}.
\end{align*}
The coefficient of $\bm{s_{\lambda}(x)}$ in the Schur expansion of $\bm{s_{\mu}(x)} \cdot \bm{s_{\nu}(x)}$ is
\begin{equation*}
\bm{c^{\lambda}_{\mu,\nu}} := c_{\mu^{(1)}, \nu^{(1)}}^{\lambda^{(1)}} \cdots c_{\mu^{(r)}, \nu^{(r)}}^{\lambda^{(r)}},
\end{equation*}
where the numbers $c_{\mu^{(1)}, \nu^{(1)}}^{\lambda^{(1)}}, \dots, c_{\mu^{(r)}, \nu^{(r)}}^{\lambda^{(r)}}$ are 
Littlewood-Richardson coefficients.

By the last paragraph, we have the isomorphism of $G_{(n-j,j)}$-modules
\begin{equation}
\mathrm{Res}^{G_n}_{G_{(n-j,j)}} \bm{S^{\lambda}} \cong
\bigoplus_{\substack{ \bm{\mu} \vdash_r n-j \\ \bm{\nu} \vdash_r j}} 
\bm{c_{\mu,\nu}^{\lambda}} (\bm{S^{\mu}} \otimes \bm{S^{\nu}}),
\end{equation}
which implies the isomorphism of $G_{n-j}$-modules
\begin{equation}
\epsilon_{i,j} \bm{S^{\lambda}} \cong 
\bigoplus_{\substack{ \bm{\mu} \vdash_r n-j \\ \bm{\nu} \vdash_r j}} 
 \bm{c_{\mu,\nu}^{\lambda}} (\bm{S^{\mu}} \otimes \epsilon_{i,j} \bm{S^{\nu}}).
\end{equation}
However, since the idempotent $\epsilon_{i,j}$ projects onto the 
$\bm{\nu_0} := (\varnothing, \dots, (1^j), \dots, \varnothing)$-isotypic component of any 
$G_j$-module (where the nonempty
partition is in position $i$), we have
\begin{equation}
\epsilon_{i,j} \bm{S^{\nu}} = \begin{cases}
\bm{S^{\nu_0}} & \bm{\nu} = \bm{\nu_0} \\
0 & \bm{\nu} \neq \bm{\nu_0}.
\end{cases}
\end{equation}
Since  $\bm{S^{\nu_0}}$ is 1-dimensional, we deduce
\begin{equation}
\epsilon_{i,j} \bm{S^{\lambda}} \cong 
\bigoplus_{\bm{\mu} \vdash_r n-j} 
\bm{c_{\mu, \nu_0}^{\lambda}}  \bm{S^{\mu}},
\end{equation}
or 
\begin{equation}
\Frob(\epsilon_{i,j} \bm{S^{\lambda}}) = \sum_{\bm{\mu} \vdash_r n-j} 
\bm{c_{\mu, \nu_0}^{\lambda}} \bm{s_{\mu}}(\xx).
\end{equation}
To complete the proof,  observe that $\Frob(S^{\bm{\nu_0}}) = e_j(\xx^{(i)})$ and
apply the definition of adjoint operators (together with the dualizing operation $i \mapsto i^*$ 
in the relevant inner product $\langle \cdot, \cdot \rangle$).
\end{proof}

We will need to consider the action of the idempotent $\epsilon_{i,j}$ on polynomials in $\CC[\xx_n]$.
Our basic tool is 
the following lemma describing the action of $\epsilon_{i,j}$ on monomials in the variables 
$x_{n-j+1}, \dots, x_n$.

\begin{lemma}
\label{last-variable-lemma}
Let $(a_{n-j+1}, \dots, a_n)$ be a length $j$ sequence of nonnegative integers and consider the corresponding 
monomial $x_{n-j+1}^{a_{n-j+1}} \cdots x_n^{a_n}$.  Unless the numbers $a_{n-j+1}, \dots, a_n$ are distinct 
and all congruent to $-i$ modulo $r$, we have
\begin{equation}
\epsilon_{i,j} \cdot (x_{n-j+1}^{a_{n-j+1}} \cdots x_n^{a_n}) = 0.
\end{equation}
Furthermore, if $(a'_{n-j+1}, \dots, a'_n)$ is a rearrangement of $(a_{n-j+1}, \dots, a_n)$, we have
\begin{equation}
\epsilon_{i,j} \cdot (x_{n-j+1}^{a_{n-j+1}} \cdots x_n^{a_n}) = 
\pm \epsilon_{i,j} \cdot (x_{n-j+1}^{a'_{n-j+1}} \cdots x_n^{a'_n}).
\end{equation}
\end{lemma}

\begin{proof}
Recall that $G_n$ acts on $\CC[\xx_n]$ by linear substitutions.  
In particular, if $1 \leq \ell \leq n$ and $\pi \in \symm_n \subseteq G_n$, we have
$\pi.x_{\ell} = x_{\pi_{\ell}}$.  Moreover, if $g = \mathrm{diag}(g_1, \dots, g_n) \in G_n$ 
is a diagonal matrix, we have
$g.x_{\ell} = g_{\ell}^{-1} x_{\ell}$.  Using these rules, the lemma is  a routine computation.
\end{proof}

The group $G_j$ acts on the quotient ring 
$V_{n,k,j} := \CC[x_{n-j+1}, \dots, x_n] / \langle x_{n-j+1}^{kr}, \dots, x_n^{kr} \rangle$.  For any $1 \leq i \leq r$, let 
$\epsilon_{i,j} V_{n,k,j}$ be the image of $V_{n,k,j}$ under $\epsilon_{i,j}$.  Then 
$\epsilon_{i,j} V_{n,k,j}$ is a graded vector space on which the idempotent
$\epsilon_{i,j}$ acts as the identity operator.
As a consequence of Lemma~\ref{last-variable-lemma}, the set of polynomials
\begin{equation}
\{ \epsilon_{i,j} \cdot (x_{n-j+1}^{a_{n-j+1}} \cdots x_n^{a_n}) \,:\, 
0 \leq a_{n-j+1} < \cdots < a_n < kr,   \text{ $a_{\ell} \equiv -i$ (mod $r$) for all $\ell$} \}
\end{equation}
descends to a basis for $\epsilon_{i,j} V_{n,k,j}$. 
Counting the degrees of the monomials appearing in the above set,
we have the Hilbert series
\begin{equation}
\Hilb(\epsilon_{i,j} V_{n,k,j}; q) = q^{j \cdot (r-i) + r \cdot {j \choose 2}}  {k \brack j}_{q^r}.
\end{equation}
The following generalization of \cite[Lem. 6.8]{HRS} uses the spaces 
$\epsilon_{i,j} V_{n,k,j}$ to relate the modules $\epsilon_{i,j} S_{n,k}$ and
$S_{n-j,k,k-j}$.

\begin{lemma}
\label{tensor-isomorphism}
As graded $G_j$-modules we have 
$\epsilon_{i,j} S_{n,k} \cong S_{n-j,k,k-j} \otimes \epsilon_{i,j} V_{n,k,j}$.
\end{lemma}

\begin{proof}
Write $\yy_{n-j} = (y_1, \dots, y_{n-j}) = (x_1, \dots, x_{n-j})$ and 
$\zz_j = (z_1, \dots, z_j) = (x_{n-j+1}, \dots, x_n)$, so that 
$\CC[\xx_n] = \CC[\yy_{n-j}, \zz_j]$.
The operator $\epsilon_{i,j} \in \CC[G_j]$ acts on the $\zz$ variables and commutes with the $\yy$ variables.

There is a natural multiplication map
\begin{equation}
\widetilde{\mu}: \CC[\yy_{n-j}] \otimes \epsilon_{i,j} V_{n,k,j} \rightarrow \epsilon_{i,j} \CC[\xx_n] / \epsilon_{i,j} J_{n,k} 
\cong \epsilon_{i,j} S_{n,k}
\end{equation}
coming from the assignment $f(\yy_{n-j}) \otimes g(\zz_j) \mapsto f(\yy_{n-j}) g(\zz_j)$.  
The map $\widetilde{\mu}$ commutes with the action of $G_{n-j}$ on the $\yy$ variables.
We show that $\widetilde{\mu}$ descends to the desired isomorphism.

We calculate
\begin{equation}
\epsilon_{i,j}(e_d(\yy_{n-j}^r, \zz_j^r)) = \sum_{a + b = d} e_a(\yy_{n-j}^r) \epsilon_{i,j}(e_b(\zz_j^r)) = 
e_d(\yy_{n-j}^r)
\end{equation}
for any $d > 0$.  It follows that $e_d(\yy_{n-j}^r) \in \epsilon_{i,j} J_{n,k}$ for all $d > n-k$.
For any $f(\zz_j) \in \epsilon_{i,j} V_{n,k,j}$ we have
\begin{equation}
\widetilde{\mu}(y_{\ell}^{kr} \otimes f(\zz_j)) = y_{\ell}^{kr} f(\zz_j) 
= y_{\ell}^{kr} \epsilon_{i,j} (f(\zz_j)) = \epsilon_{i,j} (y_{\ell}^{kr} f(\zz_j)) \in \epsilon_{i,j} J_{n,k},
\end{equation}
where we used the fact that $\epsilon_{i,j}$ acts as the identity operator on 
$\epsilon_{i,j} V_{n,k,j}$.

By the last paragraph, we have $J_{n-j,k,k-j} \otimes \epsilon_{i,j} V_{n,k,j} \subseteq \mathrm{Ker}(\widetilde{\mu})$.
The map $\widetilde{\mu}$ therefore induces a map 
\begin{equation}
\mu: S_{n-j,k,k-j} \otimes \epsilon_{i,j} V_{n,k,j} \rightarrow \epsilon_{i,j} \CC[\xx_n]/\epsilon_{i,j} J_{n,k} 
\cong \epsilon_{i,j} S_{n,k}.
\end{equation}

To determine the dimension of the target of $\mu$, consider the action of $\epsilon_{i,j}$ 
on $\CC[\OP_{n,k}]$.  Given $\sigma \in \OP_{n,k}$, we have $\epsilon_{i,j}.\sigma = 0$
if and only if two of the big letters $n-j+1, \dots, n-1, n$ lie in the same block of $\sigma$.  
Moreover, if $\sigma'$ is obtained from $\sigma$ by rearranging the letters 
$n-j+1, \dots, n-1, n$ and/or changing their colors, then $\epsilon_{i,j}.\sigma'$ is a scalar multiple
of $\epsilon_{i,j}.\sigma$. 
By Theorem~\ref{ungraded-isomorphism-type},
the dimension of the target of $\mu$ is
\begin{equation}
\label{mu-dimension}
\dim(\epsilon_{i,j} S_{n,k}) = \dim(\epsilon_{i,j} \CC[\OP_{n,k}]) = {k \choose j} \cdot |\OP_{n-j,k,k-j}|,
\end{equation}
where the binomial coefficient ${k \choose j}$ comes from deciding which of the $k$ blocks of $\sigma$
receive the $j$ big letters.
On the other hand, Lemma~\ref{s-dimension-lemma-generalized} and 
the discussion after Lemma~\ref{last-variable-lemma} imply that the domain of $\mu$ 
also has dimension given by (\ref{mu-dimension}).  
To prove that $\mu$ gives the desired isomorphism, it is therefore enough to show that $\mu$ 
is surjective.

To see that $\mu$ is surjective, let $\CCC_{n,k,j}$ be the set 
of polynomials of the form $\epsilon_{i,j} m(\xx_n)$, where
$m(\xx_n) = m(\yy_{n-j}) \cdot m(\zz_j) \in \NNN_{n,k}$ has the property that 
$m(\zz_j) = z_1^{a_1} \cdots z_j^{a_j}$ with $a_1 < \cdots < a_j$ and 
$a_{\ell} \equiv -i$ (mod $r$) for all $\ell$.  We claim that $\CCC_{n,k,j}$ descends to a basis of 
$\epsilon_{i,j} S_{n,k}$.

Since $\NNN_{n,k}^r$ is a basis of $S_{n,k}$, the set $\{ \epsilon_{i,j} m(\xx_n) \,:\, m(\xx_n) \in \NNN_{n,k} \}$
spans $\epsilon_{i,j} S_{n,k}$.
Let $m(\xx_n) = m(\yy_{n-j}) \cdot m(\zz_j) \in \NNN_{n,k}$.  
By Lemma~\ref{last-variable-lemma}, we have $\epsilon_{i,j} m(\xx_n) = 0$ unless 
$m(\zz_j) = z_1^{a_1} \cdots z_j^{a_j}$ with $(a_1, \dots, a_j)$ distinct and 
$a_{\ell} \equiv -i$ (mod $r$) for all $\ell$.
Also, if $m(\zz_j)' = z_1^{a_1'} \cdots z_j^{a_j'}$ for any permutation $(a_1', \dots, a_j')$
of $(a_1, \dots, a_j)$, then $\epsilon_{i,j} m(\xx_n) = \pm \epsilon_{i,j} m(\yy_{n-j}) \cdot m(\zz_j)'$.
It follows that $\CCC_{n,k,j}$ descends to a spanning set of $\epsilon_{i,j} S_{n,k}$.

Lemmas~\ref{nonskip-monomial-factor}, \ref{size-of-n}, and 
\ref{s-dimension-lemma-generalized} imply
\begin{equation}
|\CCC_{n,k,j}| = {k \choose j} \cdot |\OP_{n-j,k,k-j}| = \dim(\epsilon_{i,j} S_{n,k}).
\end{equation}
It follows that $\CCC_{n,k,j}$ descends to a basis of $\epsilon_{i,j} S_{n,k}$.

Consider a typical element $\epsilon_{i,j} m(\xx_n) = m(\yy_{n-j}) \cdot \epsilon_{i,j} m(\zz_j) \in \CCC_{n,k,j}$.
We have
\begin{equation}
\mu(m(\yy_{n-j}) \otimes \epsilon_{i,j} m(\zz_j)) = m(\yy_{n-j}) \cdot \epsilon_{i,j} m(\zz_j) = \epsilon_{i,j} m(\xx_n),
\end{equation}
so that $\epsilon_{i,j} m(\xx_n)$ lies in the image of $\mu$.  It follows that $\mu$ is surjective.
\end{proof}

By Lemma~\ref{tensor-isomorphism}, we have 
\begin{align}
e_j(\xx^{(i^*)})^{\perp} \grFrob(S_{n,k}; q) &= 
\Hilb(\epsilon_{i,j} V_{n,k,j}^r; q) \cdot \grFrob(S_{n-j,k,k-r}^r; q) \\ 
&=  q^{j \cdot (r-i) + r \cdot {j \choose 2}} {k \brack j}_{q^r} \cdot \grFrob(S_{n-j,k,k-r}^r; q).
\end{align}
It we want $\grFrob(S_{n,k}; q)$ to satisfy the same recursion that $\bm{D_{n,k}}(\xx;q)$ satisfies
from Lemma~\ref{d-under-e-perp}, our goal is therefore
\begin{lemma}
\label{target-lemma}
\begin{equation}
\label{target-equation}
\grFrob(S_{n-j,k,k-j};q) = \sum_{m = \max(1,k-j)}^{\min(k,n-j)}
q^{r \cdot (k-m) \cdot (n-j-m)} {j \brack k-m}_{q^r} \grFrob(S_{n-j,m}; q).
\end{equation}
\end{lemma}

\begin{proof}
This is proven using the same reasoning as in the proofs of \cite[Lem. 6.9, Lem. 6.10]{HRS};
one just makes the change of variables $(x_1, \dots, x_n) \mapsto (x_1^r, \dots, x_n^r)$
and $q \mapsto q^r$.
\end{proof}

We are ready to describe the graded isomorphism types of $S_{n,k}$ and $R_{n,k}$.

\begin{theorem}
\label{graded-isomorphism-type}
Let $n, k,$ and $r$ be positive integers with $n \geq k$ and $r \geq 2$.  
We have
\begin{equation}
\grFrob(S_{n,k}; q) = \bm{D_{n,k}}(\xx; q)
\end{equation}
and
\begin{equation}
\grFrob(R_{n,k}; q) = \sum_{z = 0}^{n-k} q^{krz} \cdot \bm{s}_{\varnothing, \dots, \varnothing, (z)}(\xx) \cdot
 \bm{D_{n-z,k}}(\xx; q).
\end{equation}
\end{theorem}

When $k = n$, the graded Frobenius image of $R_{n,n} = S_{n,n}$ was calculated by 
Stembridge \cite{Stembridge}.

\begin{proof}
By Lemma~\ref{target-lemma} (and the discussion preceding it), Lemma~\ref{d-under-e-perp},
and induction, we see that 
\begin{equation}
e_j(\xx^{(i^*)})^{\perp}   \grFrob(S_{n,k}; q) =
e_j(\xx^{(i^*)})^{\perp}   \bm{D_{n,k}}(\xx; q)
\end{equation}
for all $j \geq 1$ and $1 \leq i \leq r$.  Lemma~\ref{e-perp-lemma} therefore gives the first statement.
The second statement is a consequence of Proposition~\ref{r-to-s-reduction}.
\end{proof}

\begin{example}
Theorem~\ref{graded-isomorphism-type} may be verified directly in the case $n = k = 1$.  We have 
$S_{1,1} = R_{1,1} = \CC[x_1]/\langle x_1^r \rangle$.  The group $G_1 \cong G = \langle \zeta \rangle$ acts on
$S_{1,1}$ by $\zeta.x_1^i = \zeta^{-i} x_1^i$ for $0 \leq i < r$.  Recalling our convention for the characters of the 
cyclic group $G$, we have
\begin{equation}
\grFrob(S_{1,1}; q) = \bm{s}_{\varnothing, \dots, \varnothing, (1)} \cdot q^0 + \cdots 
+ \bm{s}_{\varnothing, (1), \dots, \varnothing} \cdot q^{r-2} +
 \bm{s}_{(1), \varnothing, \dots, \varnothing} \cdot q^{r-1}.
\end{equation}

On the other hand, the elements of $\SYT^r(1)$ are the tableaux
\begin{equation*}
(\varnothing, \varnothing, \dots, \, \,
\begin{Young}
1
\end{Young} \,), \, \, \dots, \, \,
(\varnothing, \begin{Young} 1 \end{Young} \, , \dots \, \varnothing), 
(\begin{Young} 1 \end{Young} \, , \varnothing, \dots, \varnothing).
\end{equation*}
The major indices of these tableaux are (from left to right) $r-1, \dots, 1, 0$.  By Proposition~\ref{d-schur-expansion}
we have
\begin{equation}
\bm{D_{1,1}}(\xx;q) = \rev_q \left[ \bm{s}_{\varnothing, \dots, \varnothing, (1)} \cdot q^{r-1} + \cdots 
+ \bm{s}_{\varnothing, (1), \dots, \varnothing} \cdot q^{1} +
 \bm{s}_{(1), \varnothing, \dots, \varnothing} \cdot q^{0} \right],
\end{equation}
which agrees with Theorem~\ref{graded-isomorphism-type}.
\end{example}

\begin{example}
Let us consider Theorem~\ref{graded-isomorphism-type} in the case $(n,k,r) = (3,2,2)$.
By Proposition~\ref{d-schur-expansion}, the only elements of $\SYT^2(3)$ which contribute to 
$\bm{D_{3,2}}(\xx;q)$ are those with $\geq 1$ descent.

\begin{small}
\begin{equation*}
\begin{Young}
1 \\ 2 \\ 3 \\ \end{Young} \, , \, \varnothing  \hspace{0.3in}  \begin{Young} 1 & 2 \\ 3 \end{Young} \, , \, \varnothing  
\hspace{0.3in}
\begin{Young} 1 & 3 \\ 2 \end{Young} \, , \, \varnothing  \hspace{0.3in}  
\begin{Young}  1 \\ 2 \end{Young} \, , \, \begin{Young} 3 \end{Young}  \hspace{0.3in}
\begin{Young} 1 & 2 \end{Young} \, , \, \begin{Young} 3 \end{Young}  \hspace{0.3in}
\begin{Young} 1 \\ 3 \end{Young} \, , \, \begin{Young} 2 \end{Young}  \hspace{0.3in}
\begin{Young} 1 & 3 \end{Young} \, , \, \begin{Young} 2 \end{Young} \hspace{0.3in}
\begin{Young} 2 \\ 3  \end{Young} \,  , \, \begin{Young} 1 \end{Young}
\end{equation*}
\begin{equation*}
\begin{Young}  1 \end{Young} \, , \, \begin{Young} 2 \\ 3 \end{Young}  \hspace{0.3in}
\begin{Young} 1 \end{Young} \, , \, \begin{Young} 2 & 3 \end{Young} \hspace{0.3in}
\begin{Young} 2 \end{Young} \, , \, \begin{Young} 1 \\ 3 \end{Young} \hspace{0.3in}
\begin{Young} 2 \end{Young} \, , \, \begin{Young} 1 & 3 \end{Young} \hspace{0.3in}
\begin{Young} 3 \end{Young} \, , \, \begin{Young} 1 \\ 2 \end{Young} \hspace{0.3in}
\varnothing \, , \, \begin{Young}  1 & 2 \\ 3 \end{Young}  \hspace{0.3in}
\varnothing \, , \, \begin{Young} 1 & 3 \\ 2 \end{Young} \hspace{0.3in} 
\varnothing \, , \, \begin{Young} 1 \\ 2 \\ 3 \end{Young}
\end{equation*}
\end{small}
The major indices of these tableaux are (in matrix format)
$\begin{pmatrix}
6 & 4 & 2 & 7 & 5 & 3 & 3 & 5 \\
8 & 4 & 6 & 6 & 4 & 7 & 5 & 9
\end{pmatrix}$ while the  descent numbers are
$\begin{pmatrix}
2 & 1 & 1 & 2 & 1 & 1 & 1 & 1 \\ 
2 & 1 & 1 & 1 & 1 & 1 & 1 & 2
\end{pmatrix}$.  The statistic $\maj(\bT) + r {n-k \choose 2} - r(n-k) \des(\bT)$ appearing in the exponent 
in Proposition~\ref{d-schur-expansion} is therefore
$
\begin{pmatrix}
2 & 2 & 0 & 3 & 3 & 1 & 1 & 3 \\
4 & 2 & 4 & 4 & 2 & 5 & 3 & 5
\end{pmatrix}.
$
If we apply $\omega$ and multiply by ${\des(\bT) \brack n-k}_{q^r} = [\des(\bT)]_{q^2}$, we see that 
$\bm{D_{3,2}}(\xx;q)$ is the $q$-reversal of
\begin{multline}
\bm{s}_{(3), \varnothing} \cdot (q^2 + q^4) + \bm{s}_{(2,1), \varnothing} \cdot q^2 + 
\bm{s}_{(2,1), \varnothing} \cdot q^0 + \bm{s}_{(2), (1)} \cdot (q^3 + q^5) \\
+  \bm{s}_{(1,1), (1)} \cdot q^3 + \bm{s}_{(2), (1)} \cdot q^1 +
\bm{s}_{(1,1), (1)} \cdot q^1 + \bm{s}_{(2), (1)} \cdot q^3 \\
+ \bm{s}_{(1), (2)} \cdot (q^4 + q^6) + \bm{s}_{(1), (1,1)} \cdot q^2 +
\bm{s}_{(1), (2)} \cdot q^4 + \bm{s}_{(1), (1,1)} \cdot q^4 \\
+ \bm{s}_{(1), (2)} \cdot q^2 + \bm{s}_{\varnothing, (2,1)} \cdot q^5 +
\bm{s}_{\varnothing, (2,1)} \cdot q^3 + \bm{s}_{\varnothing, (3)} \cdot (q^5 + q^7).
\end{multline}
Collecting  powers of $q$ and applying $\rev_q$, the graded Frobenius image $\grFrob(S_{3,2}; q)$ is
\begin{multline}
\label{small-expression}
 \bm{s}_{\varnothing, (3)} \cdot q^0 + \bm{s}_{(1), (2)} \cdot q^1
+ (\bm{s}_{(2), (1)} + \bm{s}_{\varnothing, (2,1)} + \bm{s}_{\varnothing, (3)}) \cdot q^2 \\
+ (\bm{s}_{(3), \varnothing} + 2 \bm{s}_{(1), (2)} + \bm{s}_{(1), (1,1)}) 
\cdot q^3  +
(2 \bm{s}_{(2), (1)} + \bm{s}_{(1,1), (1)} + \bm{s}_{\varnothing, (2,1)}) \cdot q^4 \\
+ (\bm{s}_{(3), \varnothing} + \bm{s}_{(2,1), \varnothing} + \bm{s}_{(1), (1,1)} + \bm{s}_{(1), (2)})
\cdot q^5  +
(\bm{s}_{(2), (1)} + \bm{s}_{(1,1), (1)}) \cdot q^6 
+ \bm{s}_{(2,1), \varnothing} \cdot q^7.
\end{multline}

Let us calculate $\grFrob(R_{3,2}; q)$.
A shorter calculation (left to the reader) shows that $\bm{D_{2,2}}(\xx; q)$ is given by
\begin{equation}
\label{new-expression}
 \bm{s}_{\varnothing, (2)} \cdot q^0 + \bm{s}_{(1), (1)} \cdot q^1 
+ (\bm{s}_{(2), \varnothing} + \bm{s}_{\varnothing, (1,1)}) \cdot q^2 +
\bm{s}_{(1), (1)} \cdot q^3 + \bm{s}_{(1,1), \varnothing} \cdot q^4.
\end{equation}
By Theorem~\ref{graded-isomorphism-type}, the Frobenius image $\grFrob(R_{3,2}; q)$ is given
by adding the product of (\ref{new-expression}) and $\bm{s}_{(\varnothing, (1))}(\xx) \cdot q^4$ to 
(\ref{small-expression}).  Applying the Pieri rule we see that the Schur expansion of 
$\grFrob(R_{3,2}; q)$ is
\begin{multline}
\text{\rm{(expression in (}\ref{small-expression}))} \,  + 
(\bm{s}_{\varnothing, (3)} + \bm{s}_{\varnothing, (2,1)}) \cdot q^4 +
(\bm{s}_{(1), (2)} + \bm{s}_{(1), (1,1)}) \cdot q^5  \\
+ (\bm{s}_{(2), (1)} + \bm{s}_{\varnothing, (2,1)} + \bm{s}_{\varnothing, (1,1,1)}) \cdot q^6 +
(\bm{s}_{(1), (2)} + \bm{s}_{(1), (1,1)}) \cdot q^7 
+ \bm{s}_{(1,1), (1)} \cdot q^8.
\end{multline} 

\end{example}

\section{Conclusion}
\label{Conclusion}

In this paper we introduced a quotient $R_{n,k}$ of the polynomial ring $\CC[\xx_n]$ whose structure 
is governed by the combinatorics of the set of $k$-dimensional faces $\FFF_{n,k}$ in the Coxeter complex 
attached to $G_n$, where $G_n = \ZZ_r \wr \symm_n$ is a wreath product.

\begin{problem}
\label{reflection-group-generalization}
Let $W \subset GL_n(\CC)$ be a complex reflection group and let $0 \leq k \leq n$.  Find a graded $W$-module
$R_{W,k}$ which generalizes $R_{n,k}$.
\end{problem}

The quotient $R_{W,k}$ in Problem~\ref{reflection-group-generalization} should have combinatorics governed
by the $k$-dimensional faces $\FFF_{W,k}$ of some Coxeter complex-like object attached to $W$.
A natural collection of groups $W$ to look at is the $G(r,p,n)$ family of reflection groups.  Recall that, for positive
integers $r, p, n$ with $p \mid r$, the group $G(r,p,n)$ is defined by
\begin{equation}
G(r,p,n) := \{ g \in G_n \,:\, \text{the product of the nonzero entries in $g$ is a $(r/p)^{th}$ root of unity} \}.
\end{equation}
It is well known that the $G(r,p,n)$-invariant polynomials $\CC[\xx_n]^{G(r,p,n)}$ have algebraically independent
generators $e_1(\xx_n^r), e_2(\xx_n^r), \dots, e_{n-1}(\xx_n^r),$ and $(x_1 \cdots x_n)^{r/p}$.
However, even in the case of $G(2,2,n)$ which is isomorphic to the real reflection group of type $D_n$,
 the authors have been unable to construct a quotient of $\CC[\xx_n]$ which carries an action of 
 $G(2,2,n)$ whose dimension is given by the number of $k$-dimensional faces in the $D_n$-Coxeter complex.
 
 If $W$ is any {\em real} reflection group and $\mathbb{F}$ is any field, there is an $\mathbb{F}$-algebra
 $H_W(0)$ of dimension $|W|$ called the {\em 0-Hecke algebra} attached to $W$.
 When $W$ is the symmetric group $\symm_n$, 
there is an action of $H_W(0)$ on the polynomial ring $\mathbb{F}[\xx_n]$ given by the 
isobaric Demazure operators (see \cite{HuangRhoades}).  When $W = \symm_n$, Huang and Rhoades 
proved that the ideal
\begin{equation}
\langle h_k(x_1), h_k(x_1, x_2), \dots, h_k(x_1, x_2, \dots, x_n), e_n(\xx_n), e_{n-1}(\xx_n), \dots, e_{n-k+1}(\xx_n) \rangle
\subseteq \mathbb{F}[\xx_n]
\end{equation}
is stable under this action, and that the corresponding quotient of $\mathbb{F}[\xx_n]$ gives a graded version of 
a natural action of $H_{\symm_n}(0)$ on $k$-block ordered set partitions of $[n]$.
This suggests the following problem.

\begin{problem}
\label{zero-hecke-problem}
Let $W$ be a real reflection group of rank $n$, let $H_W(0)$ be the 0-Hecke algebra attached to $W$, and let 
$0 \leq k \leq n$.  Describe a natural action of $W$ on the set of $k$-dimensional faces in the Coxeter complex of $W$.
Give a graded this action as a $W$-stable quotient of $\mathbb{F}[\xx_n]$.
\end{problem}

Another possible direction for future research is motivated by the Delta Conjecture and the {\em Parking Conjecture}
of Armstrong, Reiner, and Rhoades \cite{ARR}.  Let $W$ be an irreducible real reflection group with reflection
representation $V$ and Coxeter number $h$, and consider a homogeneous system of parameters 
$\theta_1, \dots, \theta_n \in \CC[V]_{h+1}$ of degree $h+1$ carrying the dual $V^*$ of the reflection 
representation.  Armstrong et. al. introduce an inhomogeneous deformation $(\Theta - \xx)$ of the ideal
$(\Theta) = (\theta_1, \dots, \theta_n) \subseteq \CC[V]$ generated by the $\theta_i$ and conjecture a 
relationship between the quotient $\CC[V]/(\Theta - \xx)$ and $(W \times \ZZ_h)$-set $\mathsf{Park}^{NC}_W$
of `$W$-noncrossing parking functions' defined via Coxeter-Catalan theory.

When $W = \symm_n$ is the symmetric group,
the `classical' h.s.o.p. quotient $\CC[V]/(\Theta)$ is known to have graded Frobenius image given by 
(the image under $\omega$ of, after a $q$-shift) the Delta conjecture in the case $k = n$ at the specialization $t = 1/q$.
In \cite[Prob. 7.8]{HRS} the problem was posed of finding a `$k \leq n$' extension of the Parking Conjecture
for any real reflection group $W$.  The authors are hopeful that the quotients studied in this paper
will be helpful in this endeavor.

\section{Acknowledgements}
\label{Acknowledgements}

B. Rhoades was partially supported by NSF Grant DMS-1500838.
This work was performed as an REU at UCSD which was supported by NSF Grant DMS-1500838.


\begin{thebibliography}{99}
 
 \bibitem{ABR} R. Adin, F. Brenti, and Y. Roichman.  Descent representations and multivariate statisitcs.
 {\it Trans. Amer. Math. Soc.}, {\bf 357} (2005), 3051--3082.
 
 \bibitem{ARR}  D. Armstrong, V. Reiner, and B. Rhoades.  Parking spaces.
 {\it Adv. Math.}, {\bf 269} (2015), 647--706.
 
 \bibitem{Artin}  E. Artin.  {\it Galois Theory,} Second edition.
 Notre Dame Math Lectures, no. 2.  Notre Dame: University of Notre Dame, 1944.
 
 \bibitem{BB} E. Bango and R. Biagioli.  Colored-descent representations of complex
 reflection groups $G(r,p,n)$.
 {\it Israel J. Math.}, {\bf 160 (1)}  (2007), 317--347.
 
 \bibitem{Bergeron}  F. Bergeron.  {\it Algebraic Combinatorics and Coinvariant Spaces.}
 CMS Treatises in Mathematics.
Boca Raton:  Taylor and Francis, 2009.

 
 
 \bibitem{C}  C. Chevalley.  Invariants of finite groups generated by reflections.
 {\it Amer. J. Math.}, {\bf 77 (4)} (1955), 778--782.
 
 \bibitem{Dowling}  T. A. Dowling.  A class of geometric lattices based on finite groups.
 {\it J. Combin. Theory Ser. B}, {\bf 14} (1973), 61--86.
 

 \bibitem{Garsia}  A. M. Garsia.  Combinatorial methods in the theory of Cohen-Macaulay rings.
 {\it Adv. Math.}, {\bf 38} (1980), 229--266.
 
 \bibitem{GP}  A. M. Garsia and C. Procesi.  On certain graded $S_n$-modules and the $q$-Kostka
 polynomials.  {\it Adv. Math.}, {\bf 94 (1)} (1992), 82--138.
 
 \bibitem{GS}  A. M. Garsia and D. Stanton.  Group actions on Stanley-Reisner rings and invariants of permutation
 groups.  {\it Adv. Math.}, {\bf 51 (2)} (1984), 107--201.
 
 
 \bibitem{HHL}  J. Haglund, M. Haiman, and N. Loehr.  A combinatorial formula for nonsymmetric 
 Macdonald polynomials.  {\it Amer. J. of Math.}, {\bf 103} (2008), 359--383.
 

\bibitem{HLR}  J. Haglund, N. Loehr, and J. Remmel.  Statistics on wreath products, perfect matchings,
and signed words.  {\it European J. Combin.}, {\bf 26} (2005), 835--868.
 

\bibitem{HRW}  J. Haglund, J. Remmel, and A. T. Wilson.  The Delta Conjecture.  
Accepted, {\it Trans. Amer. Math. Soc.}, 2016.  {\tt arXiv:1509.07058}.

\bibitem{HRS}  J. Haglund, B. Rhoades, and M. Shimozono.  Ordered set partitions,
generalized coinvariant algebras, and the Delta conjecture.  Preprint, 2016.



\bibitem{HuangRhoades}  J. Huang and B. Rhoades.  Ordered set partitions and the 0-Hecke algebra.
Preprint, 2016.  {\tt arXiv:1611.01251}.


\bibitem{Macdonald}  I. G. Macdonald.  {\it Symmetric Functions and Hall Polynomials,} Second edition.
Oxford Mathematican Monographs.
New York: The Clarendon Press Oxford University Press, 1995.
With contributions by A. Zelevinsky, Oxford Science Publications.



\bibitem{RW}  J. Remmel and A. T. Wilson.  An extension of MacMahon's Equidistribution
Theorem to ordered set partitions.
{\it J. Combin. Theory Ser. A}, {\bf 134} (2015), 242--277.

\bibitem{Rhoades}  B. Rhoades.  Ordered set partition statistics and the Delta Conjecture.
Preprint, 2016. {\tt arXiv:1605.04007}.


\bibitem{Specht}  W. Specht.  Eine Verallgemeinerung der symmetrischen Gruppe.
{\it Schriften Math. Seminar (Berlin)}, {\bf 1} (1932), 1--32.

\bibitem{Stanley}  R. P. Stanley.  Invariants of finite groups and their applications to combinatorics.
{\it Bull. Amer. Math. Soc.}, {\bf 1} (1979), 475--511.

\bibitem{Stein} E. Steingr\'imsson.  Statistics on Ordered Partitions of Sets.  Preprint, 2014.
{\tt arXiv:0605670}.

\bibitem{Stembridge} J. Stembridge.  On the eigenvalues of reflection groups and wreath products.
{\it Pacific J. Math.}, {\bf 140} (1989), 353--396.

\bibitem{Sturmfels}  B. Sturmfels.  {\it Algorithms in Invariant Theory}.
Springer-Verlag, Berlin, 1993.


\bibitem{WMultiset}  A. T. Wilson.  An extension of MacMahon's Equidistribution Theorem
to ordered multiset partitions. 
{\it Electron. J. Combin.}, {\bf 23 (1)} (2016), P1.5.


  
\end{thebibliography}
\end{document}